\documentclass[11pt,reqno]{amsart}
\usepackage[foot]{amsaddr}
\usepackage[margin=1in]{geometry}
\usepackage{amsthm}
\usepackage{amsmath}
\usepackage{amssymb}
\usepackage{bbm}
\usepackage{enumerate}
\usepackage{graphicx}
\usepackage{algorithm}
\usepackage{algorithmic}
\usepackage{breqn}
\usepackage{booktabs}
\usepackage{multirow}
\usepackage{siunitx}
\usepackage{comment}
\usepackage[hidelinks]{hyperref}

\numberwithin{equation}{section}
\theoremstyle{definition}
\newtheorem{theorem}{Theorem}[section]
\newtheorem{proposition}[theorem]{Proposition}
\newtheorem{lemma}[theorem]{Lemma}

\newtheorem*{definition*}{Definition}
\newtheorem{assumption}[theorem]{Assumption}

\newtheorem*{remark*}{Remark}
\newtheorem{example}[theorem]{Example}

\newcommand{\R}{\mathbb{R}}
\newcommand{\C}{\mathbb{C}}
\newcommand{\E}{\mathbb{E}}
\renewcommand{\P}{\mathbb{P}}
\newcommand{\N}{\mathcal{N}}
\newcommand{\eps}{\varepsilon}
\newcommand{\1}{\mathbbm{1}}
\newcommand{\Id}{\operatorname{Id}}
\renewcommand{\Im}{\operatorname{Im}}
\renewcommand{\Re}{\operatorname{Re}}
\newcommand{\der}{\mathrm{d}}
\renewcommand{\O}{O_{\prec}}

\newcommand{\Tr}{\operatorname{Tr}}
\newcommand{\col}{\operatorname{col}}
\renewcommand{\vec}{\operatorname{vec}}
\newcommand{\diag}{\operatorname{diag}}

\newcommand{\supp}{\operatorname{supp}}
\newcommand{\dist}{\operatorname{dist}}
\newcommand{\ordereddist}{\operatorname{ordered-dist}}
\newcommand{\spec}{\operatorname{spec}}
\newcommand{\HS}{{\text{HS}}}

\newcommand{\ba}{\mathbf{a}}
\newcommand{\bb}{\mathbf{b}}
\newcommand{\e}{\mathbf{e}}
\renewcommand{\i}{\mathbf{i}}
\renewcommand{\j}{\mathbf{j}}
\newcommand{\bt}{\mathbf{t}}
\renewcommand{\u}{\mathbf{u}}
\renewcommand{\v}{\mathbf{v}}
\newcommand{\w}{\mathbf{w}}
\newcommand{\x}{\mathbf{x}}
\newcommand{\y}{\mathbf{y}}
\newcommand{\z}{\mathbf{z}}

\newcommand{\bmu}{\boldsymbol{\mu}}
\newcommand{\balpha}{\boldsymbol{\alpha}}
\newcommand{\bbeta}{\boldsymbol{\beta}}
\newcommand{\bgamma}{\boldsymbol{\gamma}}
\newcommand{\bdelta}{\boldsymbol{\delta}}
\newcommand{\beps}{\boldsymbol{\varepsilon}}
\newcommand{\one}{\mathbf{1}}

\renewcommand{\a}{\alpha}
\renewcommand{\b}{\beta}
\newcommand{\g}{\gamma}

\renewcommand{\r}{\rho}

\newcommand{\vB}{\check{B}}

\newcommand{\vF}{\check{F}}
\newcommand{\vG}{\check{G}}

\newcommand{\vM}{\check{M}}
\newcommand{\vU}{\check{U}}

\newcommand{\vY}{\check{Y}}

\newcommand{\vm}{\check{m}}
\newcommand{\vn}{\check{n}}

\newcommand{\valpha}{\check{\alpha}}

\newcommand{\cC}{\mathcal{C}}

\newcommand{\cE}{\mathcal{E}}

\newcommand{\I}{\mathcal{I}}
\newcommand{\cM}{\mathcal{M}}

\newcommand{\Q}{\mathcal{Q}}

\newcommand{\cS}{\mathcal{S}}
\newcommand{\cT}{\mathcal{T}}
\newcommand{\cY}{\mathcal{Y}}
\newcommand{\cZ}{\mathcal{Z}}

\newcommand{\hf}{\hat{f}}
\newcommand{\hg}{\hat{g}}
\newcommand{\hv}{\mathbf{\hat{v}}}

\newcommand{\hl}{\hat{\lambda}}
\newcommand{\hmu}{\hat{\mu}}

\newcommand{\hK}{\widehat{K}}

\newcommand{\hLambda}{\widehat{\Lambda}}
\newcommand{\hSigma}{\widehat{\Sigma}}
\newcommand{\hcE}{\widehat{\mathcal{E}}}
\newcommand{\hcV}{\widehat{\mathcal{V}}}
\newcommand{\oS}{\mathring{S}}
\newcommand{\oV}{\mathring{V}}
\newcommand{\oX}{\mathring{X}}
\newcommand{\MS}{\text{MS}}

\title[Spiked covariances in random effects models]
{Spiked covariances and principal components analysis in high-dimensional
random effects models}
\author{Zhou Fan}
\author{Iain M. Johnstone}
\address{Department of Statistics, Stanford University}
\author{Yi Sun}
\address{Department of Mathematics, Columbia University}
\email{zhoufan@stanford.edu, imj@stanford.edu, yisun@math.columbia.edu}

\begin{document}

\maketitle
\begin{abstract}
We study principal components analyses in multivariate random and mixed effects
linear models, assuming a spherical-plus-spikes structure for the covariance
matrix of each random effect. We characterize the behavior of outlier sample
eigenvalues and eigenvectors of MANOVA variance components estimators in such
models under a high-dimensional asymptotic regime. Our results show that an 
aliasing phenomenon may occur in high dimensions, in which eigenvalues and
eigenvectors of the MANOVA estimate for one variance component may be
influenced by the
other components. We propose an alternative procedure for estimating the true
principal eigenvalues and eigenvectors that asymptotically corrects for
this aliasing problem.
\end{abstract}

\section{Introduction}
We study multivariate random and mixed effects linear models.
As a simple example, consider a twin study measuring $p$ quantitative traits in 
$n$ individuals, consisting of $n/2$ pairs of identical twins. We may model the
observed traits of the $j^{\text{th}}$ individual in the $i^{\text{th}}$ pair as
\begin{equation}\label{eq:oneway}
\y_{i,j}=\bmu+\balpha_i+\beps_{i,j} \in \R^p.
\end{equation}
Here, $\bmu$ is a deterministic vector of mean trait values in the population,
and
\[\balpha_i \overset{iid}{\sim} \N(0,\Sigma_1), \qquad \beps_{i,j}
\overset{iid}{\sim} \N(0,\Sigma_2)\]
are unobserved, independent random vectors modeling trait variation at the pair
and individual levels. Assuming the absence of shared environment, the
covariance matrices $\Sigma_1,\Sigma_2 \in \R^{p \times p}$ may be interpreted
as the genetic and environmental components of variance.

Since the pioneering work of R.\ A.\ Fisher \cite{fisher}, such models have been
widely used to decompose the variation of quantitative traits into constituent
variance components. Genetic variance is commonly further decomposed into
additive, dominance, and epistatic components \cite{wright}. Components of
environmental variance may be individual-specific or potentially
also shared within families or batches of an experimental protocol. In many
applications, for example measuring the heritability of traits, predicting
evolutionary response to selection, and correcting for confounding variation
from experimental procedures, it is of interest to estimate the individual
variance components \cite{falconermackay,lynchwalsh,visscheretal}.
Classically, this may be done by examining the resemblance between relatives
in simple classification designs \cite{fisher,comstockrobinson}. In modern
genome-wide association studies, where genotypes are observed at a set
of genetic markers, this is often done using models which treat
contributions of single-nucleotide polymorphisms to
polygenic traits as independent and unobserved random effects
\cite{yangetal,zhouetal,moseretal,lohetal,finucaneetal}.

These types of mixed effects linear models are often applied in univariate
contexts, $p=1$, to study the genetic basis of individual traits.
However, certain questions arising in evolutionary biology require an
understanding of the joint variation of multiple, and oftentimes many,
quantitative phenotypes \cite{lande,landearnold,houle,houleetal}. In such
multivariate contexts, it is
often natural to interpret the covariance matrices of the variance components
in terms of their principal component
decompositions \cite{blows,blowsmcguigan}. For example, the largest
eigenvalues and effective rank of the additive genetic component of covariance
indicate the extent to which evolutionary response to natural selection is
genetically constrained to a lower dimensional phenotypic subspace, and the
principal eigenvectors indicate likely directions of phenotypic response
\cite{mezeyhoule,hineblows,walshblows,hineetal,blowsetal}.
Similar interpretations apply to the spectral structure
of variance components that capture variation due to genetic mutation
\cite{mcguiganmutation}. In studies involving gene-expression
phenotypes, trait dimensionality in the several thousands is common
\cite{mcguiganetal,colletetal}.

\begin{figure}
\includegraphics[width=0.33\textwidth,page=1]{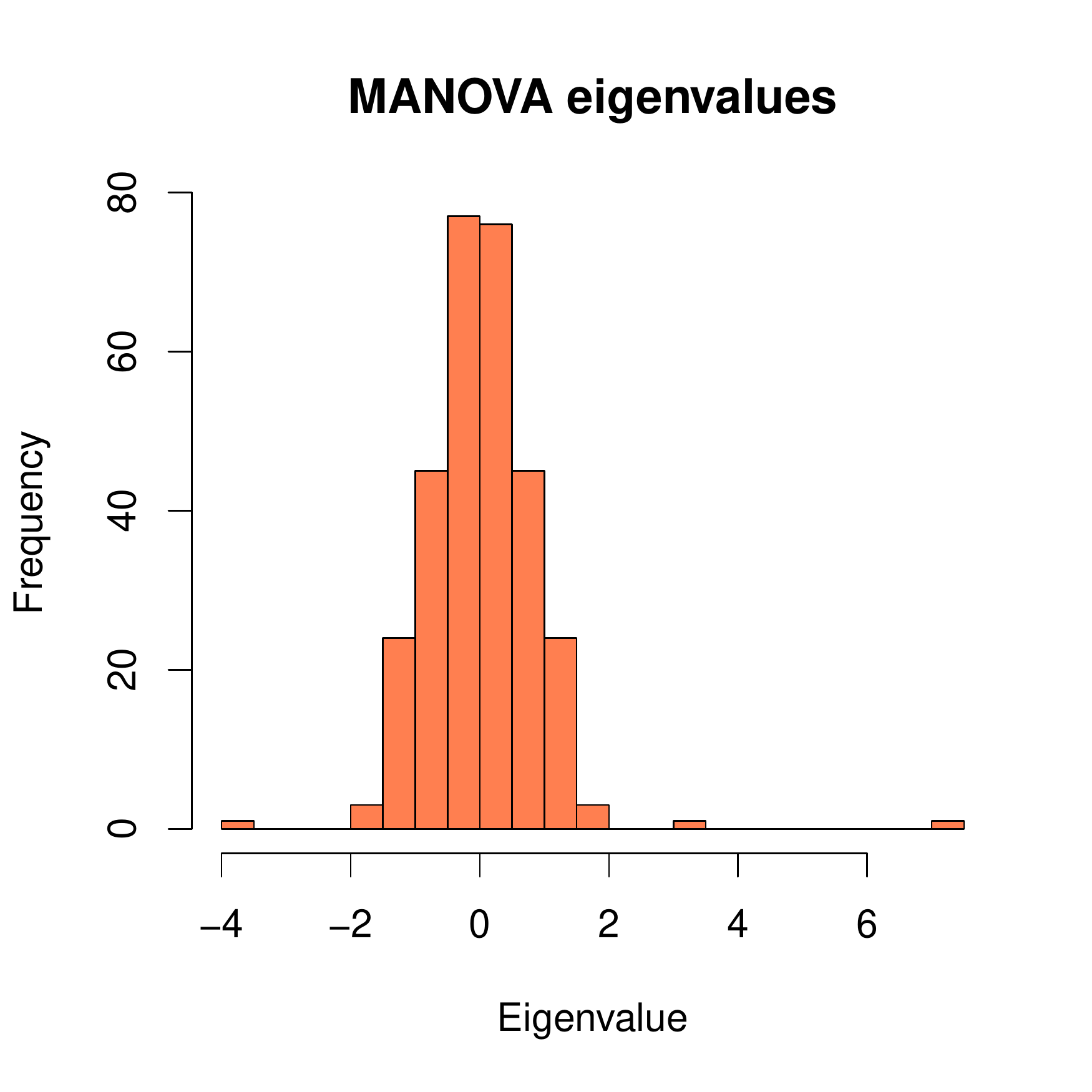}%
\includegraphics[width=0.33\textwidth,page=2]{introduction.pdf}%
\includegraphics[width=0.33\textwidth,page=3]{introduction.pdf}
\caption{Eigenvalues and principal eigenvector of the MANOVA and REML estimates
of $\Sigma_1$ in a one-way design with $I=300$ groups of size $J=2$ and $p=300$
traits. The true group covariance is
$\Sigma_1=6\e_1\e_1'$ (rank one), and the true error covariance is
$\Sigma_E=29\v\v'+\Id$ where $\v=\frac{1}{2}\e_1+\frac{\sqrt{3}}{2}\e_2$.
Histograms display eigenvalues averaged
across 100 simulations. The rightmost plot displays the empirical
mean and 90\% ellipsoids for the first two coordinates of the unit-norm
principal eigenvector (MANOVA in red and REML in blue), with $\e_1$ and $\v$
shown in black.}\label{fig:introduction}
\end{figure}

Motivated by these applications, we study in this work the spectral behavior of 
variance components estimates when the number of traits $p$ is large. To
illustrate some of the problems that may arise,
Figure \ref{fig:introduction} depicts
the eigenvalues and principal eigenvector of the multivariate analysis of
variance (MANOVA) \cite{searlerounsaville,searleetal} and multivariate
restricted maximum likelihood (REML) \cite{klotzputter,meyerREML} estimates of
$\Sigma_1$ in the balanced one-way model (\ref{eq:oneway}). REML estimates
were computed by the post-processing procedure described in \cite{amemiya}.
In this example,
the true group covariance $\Sigma_1$ has rank one, representing a single
direction of variation. The true error covariance $\Sigma_2$
also represents a single direction of variation which is partially aligned
with that of $\Sigma_1$, plus additional isotropic noise. Partial alignment of
eigenvectors of $\Sigma_2$ with those of $\Sigma_1$ may be common, for example,
in sibling designs where the additive genetic covariance contributes both to
$\Sigma_1$ and $\Sigma_2$. We observe several problematic
phenomena concerning either the MANOVA or REML estimate $\hSigma_1$:\\

\noindent {\bf Eigenvalue dispersion.} The eigenvalues of $\hSigma_1$ are widely
dispersed, even though all but one true eigenvalue of $\Sigma_1$ is non-zero.\\

\noindent {\bf Eigenvalue aliasing.} The estimate $\hSigma_1$ exhibits multiple
outlier eigenvalues which indicate significant directions of variation, even
though the true matrix $\Sigma_1$ has rank one.\\

\noindent {\bf Eigenvalue bias.} The largest eigenvalue of $\hSigma_1$ is biased
upwards from the true eigenvalue of $\Sigma_1$.\\

\noindent {\bf Eigenvector aliasing.} The principal eigenvector of
$\hSigma_1$ is not aligned with the true eigenvector of
$\Sigma_1$, but rather is biased in the direction of the eigenvector
of $\Sigma_2$.\\

Several eigenvalue shrinkage and rank-reduced estimation procedures have been
proposed to address some of these shortcomings, with associated simulation
studies of their performance in low-to-moderate dimensions
\cite{hayeshill,kirkpatrickmeyerdirect,meyerkirkpatricksmoothed,meyerkirkpatrickperils,meyerkirkpatrickbending}.
In this work, we will focus on higher-dimensional applications and study
these phenomena theoretically and from an asymptotic viewpoint.

We focus on MANOVA-type estimators, with particular attention on balanced
classification designs where such estimators are canonically defined
(cf.\ Section \ref{sec:balanced}). We leave the study of REML and
likelihood-based estimation as an important avenue for future work.
We consider the asymptotic regime where $n,p \to \infty$ proportionally,
and the number of realizations of each random effect also
increases proportionally with $n$ and $p$. In this setting, the
dispersion of sample eigenvalues was studied in \cite{fanjohnstonebulk}.
We study here the latter three phenomena above, under the simplifying assumption
of a spiked covariance model where the noise exhibited by each random effect is
isotropic \cite{johnstone}. It was observed in \cite{fanjohnstoneedges} that
in this isotropic setting, the equations describing eigenvalue dispersion reduce
to the Marcenko-Pastur equation \cite{marcenkopastur}, and we review this
in Section \ref{sec:model}.

In Section \ref{sec:outliers}, we provide a probabilistic characterization of
the behavior of outlier eigenvalues and eigenvectors. We show that in the
presence of high-dimensional noise,
each outlier eigenvalue $\hl$ of a MANOVA estimate $\hSigma$ is
close to an eigenvalue of a certain surrogate matrix which is a linear
combination of different population variance components.
When $\hl$ is an isolated eigenvalue, we show furthermore that it exhibits
asymptotic Gaussian fluctuations on the $n^{-1/2}$ scale, and
its corresponding eigenvector $\hv$ is partially aligned with the eigenvector
of this surrogate.

These results describe quantitatively the aliasing phenomena exhibited
in Figure \ref{fig:introduction}---eigenvalues and eigenvectors of the MANOVA
estimate for one variance component may be influenced by the other components.
In Section \ref{sec:estimation}, we propose a
new procedure for estimating the true principal eigenvalues and eigenvectors of 
a single variance component, by identifying alternative matrices $\hSigma$
in the linear span of the classical MANOVA estimates
where the surrogate matrix depends only on the single component being estimated.
We prove theoretically that the resulting eigenvalue estimates are
consistent in the high-dimensional asymptotic regime. The
eigenvector estimates remain inconsistent due to the high-dimensional noise,
but we show that they are asymptotically void of aliasing effects. We provide
finite-sample simulations of the performance of this algorithm in the one-way
design (\ref{eq:oneway}) for moderately large $n$ and $p$.

Proofs are contained in Section \ref{sec:proofs} and Appendix
\ref{sec:resolventapprox}. Our probabilistic results are analogous to
those regarding outlier eigenvalues and
eigenvectors for the spiked sample covariance model, studied in
\cite{baiketal,baiksilverstein,paul,nadler,baiyaoCLT}, and our proofs use
the matrix perturbation approach of \cite{paul} which is similar also to the
approaches of
\cite{benaychgeorgesnadakuditi,benaychgeorgesetal,baiyaogeneralized}.
An extra ingredient needed in our proof is a deterministic approximation for
arbitrary linear and quadratic functions of entries of the resolvent in
the Marcenko-Pastur model. We establish this for spectral arguments
separated from the limiting support, building on the local laws for this setting
in \cite{bloemendaletal,knowlesyin} and using a fluctuation averaging idea
inspired by \cite{erdosbernoulli,erdosyauyin,erdoslocalSC,erdosER}.
We note that new qualitative phenomena emerge in our model which are not
present in the setting of spiked sample covariance matrices---outliers may
depend on the alignments between population spike eigenvectors in different
variance components, and a single spike may generate multiple outliers. This
latter phenomenon was observed in a
different context in \cite{belinschietal}, which studied sums and products of
independent unitarily invariant matrices in spiked settings. We discuss two
points of contact between our results and those of
\cite{benaychgeorgesnadakuditi} and \cite{belinschietal} in Examples
\ref{ex:multperturb} and \ref{ex:freeaddition}.

\subsection*{Notational conventions}
For a square matrix $X$, $\spec(X)$ is its multiset of eigenvalues (counting
multiplicity). For a law $\mu_0$ on $\R$, we denote its closed support
\[\supp(\mu_0)=\{x \in \R: \mu_0([x-\eps,x+\eps])>0 \text{ for all } \eps>0\}.\]

$\e_i$ is the $i^{\text{th}}$ standard basis vector, $\Id$ is the identity
matrix, and $\one$ is the all-1's column vector, where dimensions are understood
from context. We use $\Id_n$ and $\one_n$ to explicitly emphasize the dimension
$n$.

$\|\cdot\|$ is the Euclidean norm for vectors and the Euclidean operator norm
for matrices. $\|\cdot\|_\HS$ is the matrix Hilbert-Schmidt norm.
$A \otimes B$ is the matrix tensor product.
When $Y$ and $M$ are matrices, $Y \sim \N(M,A \otimes B)$ is
shorthand for $\vec(Y') \sim \N(\vec(M'),A \otimes B)$, where
$\vec(Y')$ and $\vec(M')$ are the row-wise vectorizations of $Y$ and $M$.
$X'$ is the transpose of $X$, $\col(X)$ is its column span, and $\ker(X)$ is
its kernel or null space.

For subspaces $U$ and $V$, $\dim(U)$ is the dimension of $U$,
$U \oplus V$ is the orthogonal direct sum, and
$V \ominus U$ is the orthogonal complement of $U$ in $V$.

For $z \in \C$, we typically write $z=E+i\eta$ where $E=\Re z$ and $\eta=\Im z$.
For $A \subset \C$, $\dist(z,A)=\inf\{|y-z|:y \in A\}$ is the distance from $z$
to $A$.

\subsection*{Acknowledgments}
We thank quantitative geneticist Mark W.\ Blows for introducing us to this
problem and its applications in evolutionary biology. ZF was supported by a Hertz Foundation Fellowship. IMJ was supported in
part by grants NIH R01 EB001988 and NSF DMS 1407813. YS was supported by a Junior Fellow award from the Simons Foundation and NSF Grant DMS-1701654. ZF
and YS would like to acknowledge the Park City Mathematics
Institute (NSF grant DMS:1441467) where part of this research was conducted.

\section{Model}\label{sec:model}
We consider observations $Y \in \R^{n \times p}$ of $p$
traits in $n$ individuals, modeled by a Gaussian mixed effects linear model
\begin{equation}\label{eq:mixedmodel}
Y = X\beta+U_1\alpha_1+\ldots+U_k\alpha_k, \qquad \alpha_r \sim \N(0,
\Id_{m_r} \otimes \Sigma_r) \quad \text{for } r=1,\ldots,k.
\end{equation}
The matrices $\alpha_1,\ldots,\alpha_k$ are independent, with
each matrix $\alpha_r \in \R^{m_r \times p}$ having independent rows,
representing $m_r$ (unobserved) realizations of a $p$-dimensional random effect
with distribution $\N(0,\Sigma_r)$.
The incidence matrix $U_r \in \R^{n \times m_r}$, which is known from the
experimental protocol, determines how the
random effect contributes to the observations $Y$. The first term $X\beta$
models possible additional fixed effects, where $X \in \R^{n \times q}$ is a
known design matrix of $q$ regressors and $\beta \in \R^{q \times p}$ contains
the corresponding regression coefficients.

This model is usually written with an additional residual error term
$\eps \in \R^{n \times p}$. We incorporate this
by allowing the last random effect to be $\alpha_k=\eps$ and $U_k=\Id_n$.
For example, the one-way model (\ref{eq:oneway}) corresponds
to (\ref{eq:mixedmodel}) where $k=2$. Supposing there are $I$ groups of equal
size $J$, we set $m_1=I$, $m_2=n=IJ$, stack the vectors
$\y_{i,j}$, $\balpha_i$, and $\beps_{i,j}$ as the rows of
$Y$, $\alpha_1$, and $\alpha_2$, and identify
\begin{equation}\label{eq:onewaymatrix}
X=\one_n, \qquad \beta=\bmu',\qquad
U_1=\Id_I \otimes \one_J=\begin{pmatrix}
\one_J & & \\
& \ddots & \\
& & \one_J \\
\end{pmatrix},\qquad U_2=\Id_n.
\end{equation}
Here, $X$ is a single all-1's regressor, and $U_1$ has $I$ columns indicating
the $I$ groups. We discuss examples with $k \geq 3$ random effects in Section
\ref{sec:balanced}.

Under the general model (\ref{eq:mixedmodel}),
$Y$ has the multivariate normal distribution
\begin{equation}\label{eq:multivariatenormal}
Y \sim \N(X\beta,\;U_1U_1' \otimes \Sigma_1+\ldots+U_kU_k' \otimes \Sigma_k).
\end{equation}
The unknown parameters of the model are
$(\beta,\Sigma_1,\ldots,\Sigma_k)$. We study estimators of
$\Sigma_1,\ldots,\Sigma_k$ which are invariant to $\beta$ and take the form
\begin{equation}\label{eq:hatSigma}
\hSigma=Y'BY,
\end{equation}
where the estimation matrix $B \in \R^{n \times n}$ is symmetric and satisfies
$BX=0$. To obtain an estimate of $\Sigma_r$,
observe that $\E[\alpha_r'M\alpha_r]=(\Tr M)\Sigma_r$ for any matrix $M$.
Then, as $\alpha_1,\ldots,\alpha_k$ are independent with mean 0,
\begin{equation}\label{eq:unbiasedestimation}
\E[Y'BY]=\sum_{r=1}^k
\E[\alpha_r'U_r'BU_r\alpha_r]=\sum_{r=1}^k \Tr(U_r'BU_r)\Sigma_r.
\end{equation}
So $\hSigma$ is an unbiased estimate of $\Sigma_r$ when $B$ satisfies
$\Tr U_r'BU_r=1$ and $\Tr U_s'BU_s=0$ for all $s \neq r$.

In balanced classification designs, discussed in greater detail in Section
\ref{sec:balanced}, the classical MANOVA estimators are obtained by setting
$B$ to be combinations of projections onto subspaces of $\R^n$.
For example, in the one-way model corresponding to (\ref{eq:onewaymatrix}),
defining $\pi_1,\pi_2 \in \R^{n \times n}$ as
the orthogonal projections onto $\col(U_1) \ominus \col(\one_n)$ and
$\R^n \ominus \col(U_1)$, the MANOVA estimators of $\Sigma_1$
and $\Sigma_2$ are given by
\begin{equation}\label{eq:MANOVAoneway}
\hSigma_1=Y'\left(\frac{1}{J} \cdot \frac{\pi_1}{I-1}-\frac{1}{J} \cdot
\frac{\pi_2}{n-I} \right)Y,
\qquad \hSigma_2=Y'\frac{\pi_2}{n-I}Y.
\end{equation}
In unbalanced designs and more general models, various alternative choices of
$B$ lead to estimators in the generalized MANOVA \cite{searleetal}
and MINQUE/MIVQUE families \cite{rao,lamotte,swallowsearle}.

We study spectral properties of the matrix (\ref{eq:hatSigma}) in a
high-dimensional asymptotic regime, assuming a spiked model for each variance
component $\Sigma_r$.
\begin{assumption}\label{assump:main}
The number of effects $k$ is fixed while $n,p,m_1,\ldots,m_k \to \infty$. There
are constants $C,c,\bar{C}>0$ such that
\begin{enumerate}[(a)]
\item (Number of traits) $c<p/n<C$.
\item (Model design) $c<m_r/n<C$ and $\|U_r\|<C$ for each $r=1,\ldots,k$.
\item (Estimation matrix) $B=B'$, $BX=0$, and $\|B\|<C/n$.
\item (Spiked covariance) For each $r=1,\ldots,k$,
\[\Sigma_r=\sigma_r^2\Id+V_r\Theta_r V_r',\]
where $V_r \in \R^{p \times l_r}$ has orthonormal columns,
$\Theta_r \in \R^{l_r \times l_r}$ is diagonal, 
$0 \leq \sigma_r^2<C$, $0 \leq l_r<C$, and $\|\Theta_r\|<\bar{C}$. (We set
$V_r\Theta_rV_r'=0$ when $l_r=0$.)
\end{enumerate}
\end{assumption}

Under Assumption \ref{assump:main}(d),
each $\Sigma_r$ has an isotropic noise level $\sigma_r^2$ (possibly 0 if
$\Sigma_r$ is low-rank)
and a bounded number of signal eigenvalues greater than this noise level.
We allow $\sigma_r^2$, $l_r$, $V_r$, and $\Theta_r$ to vary with $n$ and $p$.
We will be primarily interested in scenarios where at least one variance
$\sigma_1^2,\ldots,\sigma_k^2$ is of size $O(1)$, although let us remark
that setting $\sigma_1^2=\ldots=\sigma_k^2=0$ also recovers the classical
low-dimensional asymptotic regime where the ambient dimension of the data is
bounded as $n \to \infty$.

In classification designs, Assumption \ref{assump:main}(b) holds when the
number of outer-most groups is proportional to $n$, and groups (and sub-groups) 
are bounded in size. This encompasses typical designs in classical
settings \cite{robertsona,robertsonb}, and we discuss several examples in
Section \ref{sec:balanced}. In models where $U_r$ is
a matrix of genotype values at $m_r$ SNPs, Assumption \ref{assump:main}(b) 
holds if $m_r \asymp n$ and $U_r$ is entrywise bounded by $C/\sqrt{n}$. This
latter condition is satisfied if genotypes at each SNP are normalized to mean 0
and variance $1/n$, and SNPs with minor allele frequency below a
constant threshold are removed. Under Assumption \ref{assump:main}(b),
the scaling $\|B\|<1/n$ in Assumption \ref{assump:main}(c) is then
natural to ensure $\Tr U_r'BU_r$ is bounded for each $r=1,\ldots,k$, and
hence $\E[Y'BY]$ is on the same scale as $\Sigma_1,\ldots,\Sigma_k$.

Throughout, we denote by
$\cS \subset \R^p$ the combined column span of $V_1,\ldots,V_k$,
where $\cS=\emptyset$ if $l_1=\ldots=l_k=0$.
$P_\cS$ and $P_{\cS^\perp}$ denote the
orthogonal projections onto $\cS$ and its orthogonal complement. We set
\[L=\dim \cS,\qquad N=p-L,\qquad M=m_1+\ldots+m_k,\]
and define a block matrix $F$ by
\begin{equation}\label{eq:F}
F_{rs}=N\sigma_r\sigma_s U_r'BU_s \in \R^{m_r \times m_s},\qquad
F=\begin{pmatrix} F_{11} & \cdots & F_{1k} \\ \vdots & \ddots & \vdots \\
F_{k1} & \cdots & F_{kk} \end{pmatrix} \in \R^{M \times M}.
\end{equation}
The ``null'' setting of no spikes, $L=0$, was studied in
\cite{fanjohnstoneedges}, which made the following simple observation.
\begin{proposition}\label{prop:nullform}
If $L=0$, then $\hSigma$ is equal in law to $X'FX$ where
$X \in \R^{M \times N}$ has i.i.d.\ $\N(0,1/N)$ entries.
\end{proposition}
\begin{proof}
If $L=0$, we may write $\alpha_r=\sqrt{N}\sigma_r X_r$,
where $X_r \in \R^{m_r \times N}$ has i.i.d.\ $\N(0,1/N)$ entries. Then,
applying $BX=0$,
\[\hSigma=Y'BY=\sum_{r,s=1}^k \alpha_r'U_r'BU_s\alpha_s=\sum_{r,s=1}^k X_r'
(N\sigma_r\sigma_s U_r'BU_s)X_s=\sum_{r,s=1}^k X_r'F_{rs}X_s.\]
The result follows upon 
stacking $X_1,\ldots,X_k$ row-wise as $X \in \R^{M \times N}$.
\end{proof}

In this case, the asymptotic spectrum of $\hSigma$ is described by the
Marcenko-Pastur equation:
\begin{theorem}\label{thm:MP}
For each $z \in \C^+$, there is a unique value $m_0(z) \in \C^+$ which
satisfies
\begin{equation}\label{eq:MP}
z=-\frac{1}{m_0(z)}+\frac{1}{N}\Tr \Big( F[\Id+m_0(z)F]^{-1} \Big).
\end{equation}
This function $m_0:\C^+ \to \C^+$ defines the Stieltjes transform of a
probability distribution $\mu_0$ on $\R$.

Under Assumption \ref{assump:main}, if $L=0$, then
$\mu_{\hSigma}-\mu_0 \to 0$ weakly almost surely as $n,p,m_1,\ldots,m_k \to
\infty$, where $\mu_{\hSigma}$
is the empirical distribution of eigenvalues of $\hSigma$.
\end{theorem}
\begin{proof}
See \cite{marcenkopastur,silverstein,silversteinbai} in the setting
where $M/N$ converges to a positive constant and the spectral distribution of
$F$ converges to a fixed limit. The above formulation follows from
Prohorov's theorem and a subsequence argument.
\end{proof}

Denote the $\delta$-neighborhood of the support of $\mu_0$ by
\[\supp(\mu_0)_\delta=\{x \in \R:\;\dist(x,\supp(\mu_0))<\delta\}.\]
We emphasize that $\mu_0$ and its support depend on $N,M,F$, although we
suppress this dependence notationally.
Then, if $L=0$, all eigenvalues of $\hSigma$ fall within
$\supp(\mu_0)_\delta$ with high probability:
\begin{theorem}\label{thm:sticktobulk}
Fix any constants $\delta,D>0$. Under Assumption \ref{assump:main}, if $L=0$,
then for a constant $n_0(\delta,D)>0$ and all $n \geq n_0(\delta,D)$,
\[\P\big[\,\spec(\hSigma) \subset \supp(\mu_0)_\delta\,\big]>1-n^{-D}.\]
\end{theorem}
\begin{proof}
See \cite{baisilverstein,knowlesyin} for positive definite $F$, and
\cite[Theorem 2.8]{fanjohnstoneedges} for the extension to $F$ having negative
eigenvalues.
\end{proof}

More generally, when $L>0$ so that $\Sigma_1,\ldots,\Sigma_k$ exhibit a
bounded number of spike eigenvalues, the bulk eigenvalue distribution of
$\hSigma$ is still described by the above law $\mu_0$, and
Theorem \ref{thm:sticktobulk} implies that only
a bounded number of eigenvalues of $\hSigma$ should fall far from
$\supp(\mu_0)$. These outlier eigenvalues and their corresponding eigenvectors
are the focus of our study.

\section{Outlier eigenvalues and eigenvectors}\label{sec:outliers}
In this section, we describe results that characterize
the asymptotic behavior of outlier eigenvalues
and eigenvectors of a general matrix $\hSigma$ of the form (\ref{eq:hatSigma}).

Let $m_0(z)$ be the Stieltjes transform of the law $\mu_0$ in Theorem
\ref{thm:MP}, defined for all $z \in \C \setminus \supp(\mu_0)$ via
\begin{equation}\label{eq:m0stieltjes}
m_0(z)=\int_\R \frac{1}{x-z}\,\mu_0(dx).
\end{equation}
Let $\Tr_r$ denote the trace of the
$(r,r)$ block in the $k \times k$ block decomposition of $\C^{M \times M}$
corresponding to $M=m_1+\ldots+m_k$.
For $z \in \C \setminus \supp(\mu_0)$, define
\begin{equation} \label{eq:TSigma}
T(z)=z\Id-\sum_{r=1}^k t_r(z)\Sigma_r, \qquad
t_r(z)=\frac{1}{N\sigma_r^2} \Tr_r \Big(F[\Id+m_0(z)F]^{-1}\Big).
\end{equation}
Here, if $\sigma_r^2=0$, then $t_r(z)$ remains well-defined by the identity
\begin{equation}\label{eq:Fidentity}
F[\Id+m_0(z)F]^{-1}=-m_0(z)F[\Id+m_0(z)F]^{-1}F+F
\end{equation}
and the definition of $F$ in (\ref{eq:F}). Let
\begin{equation}\label{eq:Lambda0}
\Lambda_0=\left[\;\lambda \in \R \setminus \supp(\mu_0):\;
0=\det(T(\lambda))\;\right]
\end{equation}
be the multiset of real roots of the function
$z \mapsto \det(T(z))$, counted with their
analytic multiplicities. We record here the following alternative definition
of $T(z)$, and properties of $T(z)$ and $\Lambda_0$.

\begin{proposition}[Properties of $T(z)$]\label{prop:Tproperties}
\hspace{1in}
\begin{enumerate}[(a)]
\item The matrix $T(z)$ is equivalently defined as
\begin{equation}\label{eq:TTheta}
T(z)=-\frac{1}{m_0(z)}\Id-\sum_{r=1}^k t_r(z)V_r\Theta_r V_r'.
\end{equation}
\item For each $z \in \C \setminus \supp(\mu_0)$, $\ker T(z) \subseteq \cS$.
\item For $\lambda \in \R \setminus \supp(\mu_0)$, 
$\partial_\lambda T(\lambda)-\Id$ is positive semi-definite.
\item For $\lambda \in \Lambda_0$,
its multiplicity as a root of $0=\det(T(\lambda))$ is equal to
$\dim \ker T(\lambda)$.
\end{enumerate}
\end{proposition}
\begin{proof}
By conjugation symmetry and continuity, the Marcenko-Pastur identity
(\ref{eq:MP}) holds for each $z \in \C \setminus \supp(\mu_0)$.
Part (a) then follows from substituting
$\Sigma_r=\sigma_r^2\Id+V_r\Theta_rV_r'$ into (\ref{eq:TSigma})
and applying (\ref{eq:MP}).
Part (b) follows from (a), as $T(z)$ is the direct sum of an operator on $\cS$
and a non-zero multiple of $\Id$ on the orthogonal complement
$\cS^\perp$. Differentiating
(\ref{eq:m0stieltjes}), $\partial_\lambda m_0(\lambda)>0$ for each $\lambda \in
\R \setminus \supp(\mu_0)$, so $\partial_\lambda t_r
=-(N\sigma_r^2)^{-1}(\partial_\lambda m_0) \Tr_r F(\Id+m_0F)^{-2}F \leq 0$.
Then part (c) follows from (\ref{eq:TSigma}).
For $\lambda \in \Lambda_0$, this implies each
eigenvalue $\mu_i(\lambda)$ of $T(\lambda)$ satisfies
$\mu_i(\lambda)-\mu_i(\lambda') \asymp (\lambda-\lambda')$
as $\lambda' \to \lambda$, so $|\det T(\lambda')| \asymp |\lambda-\lambda'|^d$
for $d=\dim\ker T(\lambda)$. This yields (d).
\end{proof}

For two finite multisets $A,B \subset \R$, define
\[\ordereddist(A,B)=\begin{cases} \infty & \text{ if } |A| \neq |B|\\
\max_i (|a_{(i)}-b_{(i)}|) & \text{ if } |A|=|B|,\end{cases}\]
where $a_{(i)}$ and $b_{(i)}$ are the ordered values of $A$ and $B$ counting
multiplicity.
The following shows that the outlier eigenvalues of $\hSigma$ are close to
the elements of $\Lambda_0$. Note that by (\ref{eq:TSigma}),
each $\lambda \in \Lambda_0$ is an eigenvalue of the matrix
\begin{equation}\label{eq:surrogate}
t(\lambda) \cdot \Sigma \equiv t_1(\lambda)\Sigma_1+\ldots+t_k(\lambda)\Sigma_k.
\end{equation}
When $\hSigma$ is the MANOVA
estimator of a variance component $\Sigma_r$, we may interpret this
matrix as a ``surrogate'' for the true matrix $\Sigma_r$ of interest.

\begin{theorem}[Outlier locations]\label{thm:spikemapping}
Fix constants $\delta,\eps,D>0$. Under Assumption \ref{assump:main},
for a constant $n_0(\delta,\eps,D)>0$ and
all $n \geq n_0(\delta,\eps,D)$, with probability at least $1-n^{-D}$
there exist $\Lambda_\delta \subseteq \Lambda_0$
and $\hLambda_\delta \subseteq \spec(\hSigma)$, containing
all elements of these multisets outside $\supp(\mu_0)_\delta$, such that
\[\ordereddist(\Lambda_\delta, \hat{\Lambda}_\delta)<n^{-1/2+\eps}.\]
\end{theorem}

The multiset $\Lambda_0$ represents a theoretical prediction for the locations
of the outlier eigenvalues of $\hSigma$---this is depicted in
Figure \ref{fig:outliers} for an example of the one-way design. 
We clarify that $\Lambda_0$ is deterministic but $n$-dependent, and it may
contain values arbitrarily close to $\supp(\mu_0)$. Hence we
state the result as a matching between two sets $\Lambda_\delta$ and
$\hLambda_\delta$ rather than the convergence of outlier eigenvalues of
$\hSigma$ to a fixed set $\Lambda_0$. We
allow $\Lambda_\delta$ and $\hLambda_\delta$ to contain values within
$\supp(\mu_0)_\delta$ so as to match values of the other set close to the
boundaries of $\supp(\mu_0)_\delta$.

\begin{figure}
\centering
\includegraphics[width=0.33\textwidth,page=1]{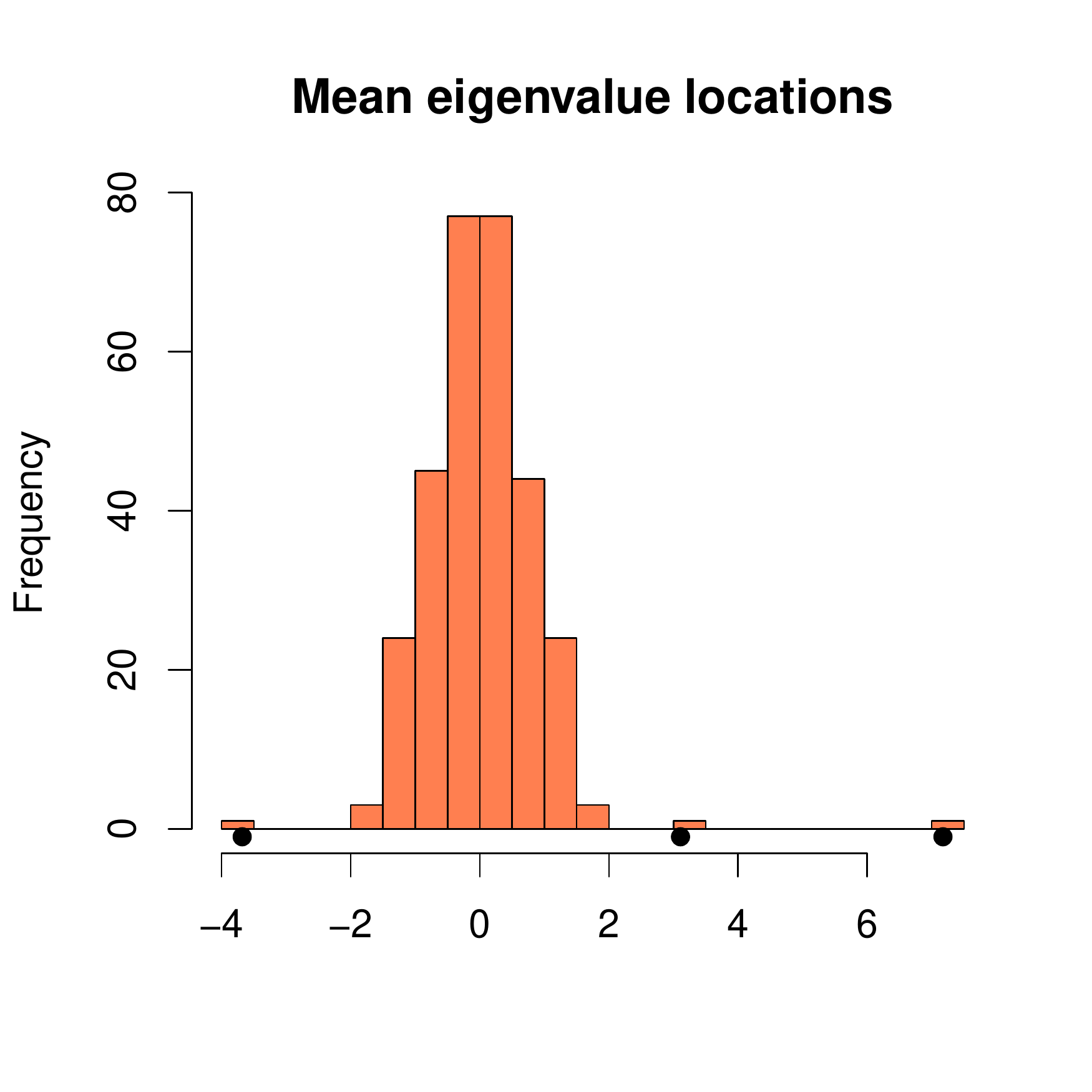}%
\includegraphics[width=0.33\textwidth,page=2]{outliers.pdf}
\caption{Outlier predictions for the MANOVA estimate $\hSigma_1$ in a one-way
design. The population covariances are $\Sigma_1=6\e_1\e_1'$ and
$\Sigma_2=29\v\v'+\Id$, where
$\v=\frac{1}{2}\e_1+\frac{\sqrt{3}}{2}\e_2$. Left: Mean eigenvalue locations of
$\hSigma_1$ across 10000 simulations, with black dots on the axis indicating
the predicted values $\lambda \in \Lambda_0$. Right: Means and 90\% ellipsoids
for the projections of the three outlier eigenvectors
onto $\cS=\col(\e_1,\e_2)$, with black dots indicating the predictions
of Theorem \ref{thm:eigenvectors}. The simulated setting is $I=300$ groups of
size $J=2$, and $p=300$ traits.}\label{fig:outliers}
\end{figure}

\begin{remark*}\label{remark:phasetransition}
In the setting of sample covariance matrices $\hSigma$ for i.i.d.\ multivariate
samples, there is a phase transition phenomenon in which spike values greater
than a certain threshold yield outlier eigenvalues in $\hSigma$, while
spike values less than this threshold do not
\cite{baiketal,baiksilverstein,paul}. This phenomenon occurs also in our
setting and is implicitly captured by the cardinality $|\Lambda_0|$, which
represents the number of predicted outlier eigenvalues of $\hSigma$.
In particular, $\Lambda_0$ will be
empty if the spike values of $\Theta_1,\ldots,\Theta_k$ are sufficiently
small. However, the phase transition thresholds and predicted outlier
eigenvalue locations in our setting are
defined jointly by $\Theta_1,\ldots,\Theta_k$ and the alignments between
$V_1,\ldots,V_k$, rather than by the individual spectra of
$\Sigma_1,\ldots,\Sigma_k$.
\end{remark*}

We next describe eigenvector projections and eigenvalue fluctuations for
isolated outliers.

\begin{theorem}[Eigenvector projections]\label{thm:eigenvectors}
Fix constants $\delta,\eps,D>0$. Suppose $\lambda \in \Lambda_0 \setminus
\supp(\mu_0)_\delta$ has multiplicity one, and
$|\lambda-\lambda'| \geq \delta$ for all other $\lambda'
\in \Lambda_0$. Let $\v$ be the unit vector in $\ker T(\lambda)$, and let
$\hv$ be the unit eigenvector of the eigenvalue $\hl$ of $\hSigma$
closest to $\lambda$.
Then, under Assumption \ref{assump:main},
\begin{enumerate}[(a)]
\item For all $n \geq n_0(\delta,\eps,D)$ and some choice of sign for
$\v$, with probability at least $1-n^{-D}$,
\[\big\|P_\cS \hv-(\v'\partial_\lambda T(\lambda) \v)^{-1/2}\v\big\|
<n^{-1/2+\eps}.\]
\item $P_{\cS^\perp}\hv/\|P_{\cS^\perp} \hv\|$ is uniformly distributed over
unit vectors in $\cS^\perp$ and is independent of $P_{\cS}\hv$.
\end{enumerate}
\end{theorem}

Thus $(\v'\partial_\lambda T(\lambda)\v)^{-1/2}\v$ represents a theoretical
prediction for the projection of the sample eigenvector $\hv$ onto the subspace
$\cS$---this is also displayed in Figure \ref{fig:outliers} for the one-way
design. Here, $(\v'\partial_\lambda T(\lambda)\v)^{-1/2}$ is the predicted
inner-product alignment between $\v$ and $\hv$, which by
Proposition \ref{prop:Tproperties}(c) is at most 1.

Next, let $\|\cdot\|_{rs}$ denote the Hilbert-Schmidt norm of the $(r,s)$
block in the $k \times k$ block decomposition of $\C^{M \times M}$ corresponding
to $M=m_1+\ldots+m_k$. Define
\begin{equation}\label{eq:wrs}
w_{rs}(z)=\frac{\|F(\Id+m_0(z)F)^{-1}\|_{rs}^2}{N\sigma_r^2\sigma_s^2},
\end{equation}
where this is again well-defined by (\ref{eq:Fidentity}) even
if $\sigma_r^2=0$ and/or $\sigma_s^2=0$.
\begin{theorem}[Gaussian fluctuations]\label{thm:CLT}
Fix a constant $\delta>0$.
Suppose $\lambda \in \Lambda_0 \setminus \supp(\mu_0)_\delta$
has multiplicity one, and $|\lambda-\lambda'| \geq \delta$
for all other $\lambda' \in \Lambda_0$. Let $\v$ be the unit vector in $\ker
T(\lambda)$, and let $\hl$ be the eigenvalue of $\hSigma$ closest to $\lambda$.
Then under Assumption \ref{assump:main},
\[\nu(\lambda)^{-1/2}(\hl-\lambda) \to \N(0,1)\]
where
\[\nu(\lambda)=\frac{2}{N(\v'\partial_\lambda T(\lambda)\v)^2}
\left(\frac{(\v'\partial_\lambda T(\lambda) \v-1)^2}{\partial_\lambda
m_0(\lambda)}+\sum_{r,s=1}^k w_{rs}(\lambda)(\v'\Sigma_r \v)(\v'\Sigma_s\v)
\right).\]
Furthermore, $\nu(\lambda)>c/n$ for a constant $c>0$.
\end{theorem}

Figure \ref{fig:CLT} illustrates the accuracy of this Gaussian approximation for
two settings of the one-way design. We observe that the approximation is fairly
accurate in a setting with a single outlier, but (in the simulated sample sizes
$n=600$ and $p=300$) does not adequately
capture a skew in the outlier distribution in a setting with an additional
positive outlier produced by a large spike in $\Sigma_2$. This skew is reduced
in examples where there is increased separation between these two positive
outliers.

\begin{figure}
\centering
\includegraphics[width=0.4\textwidth,page=3]{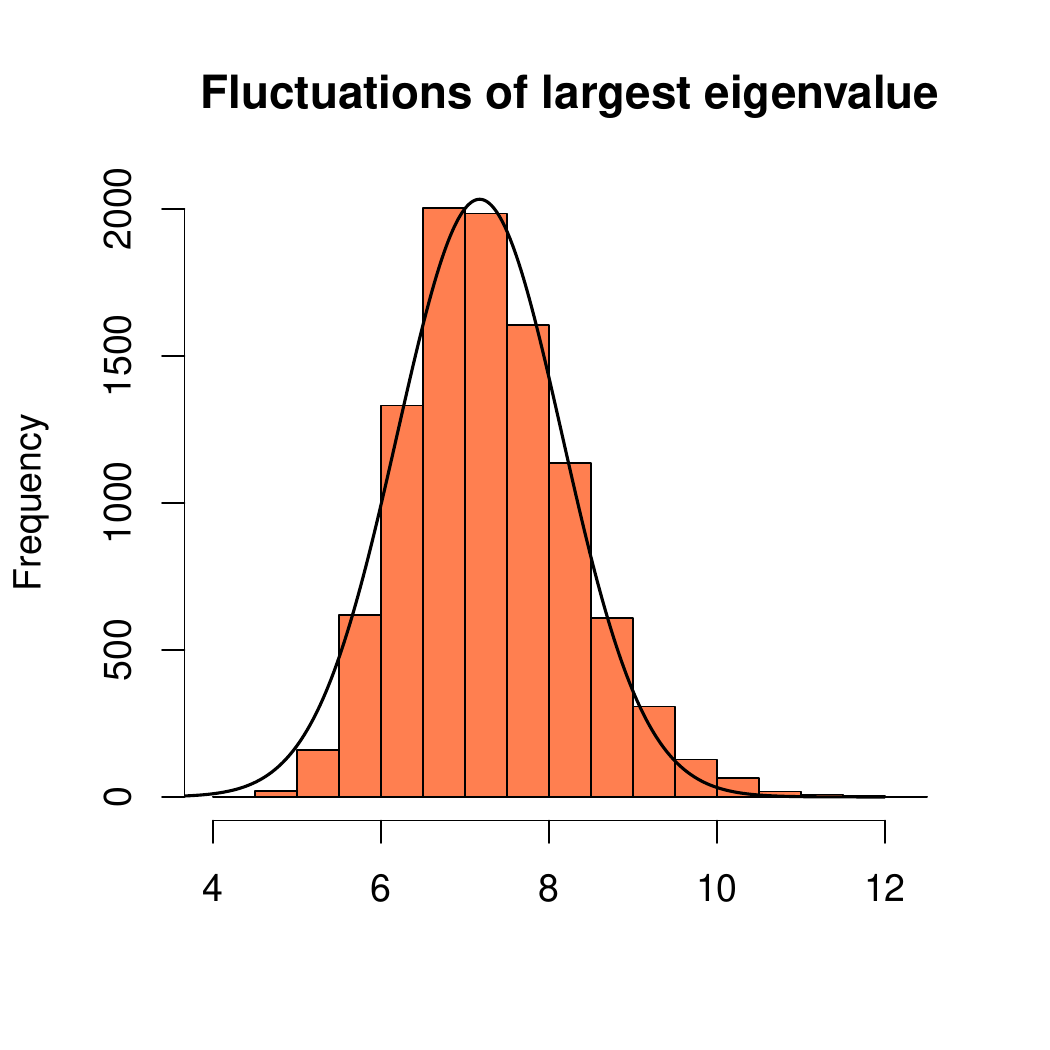}%
\includegraphics[width=0.4\textwidth,page=4]{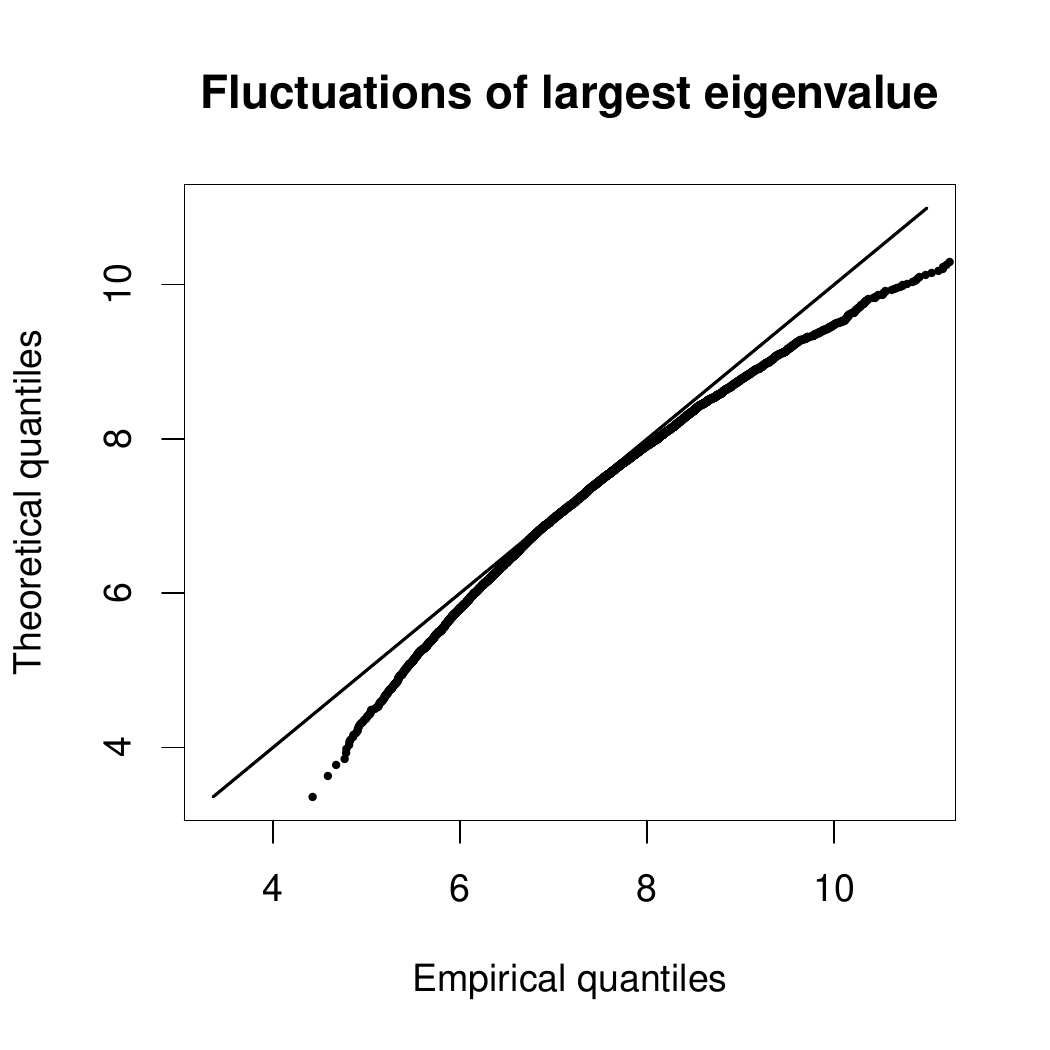}
\includegraphics[width=0.4\textwidth,page=3]{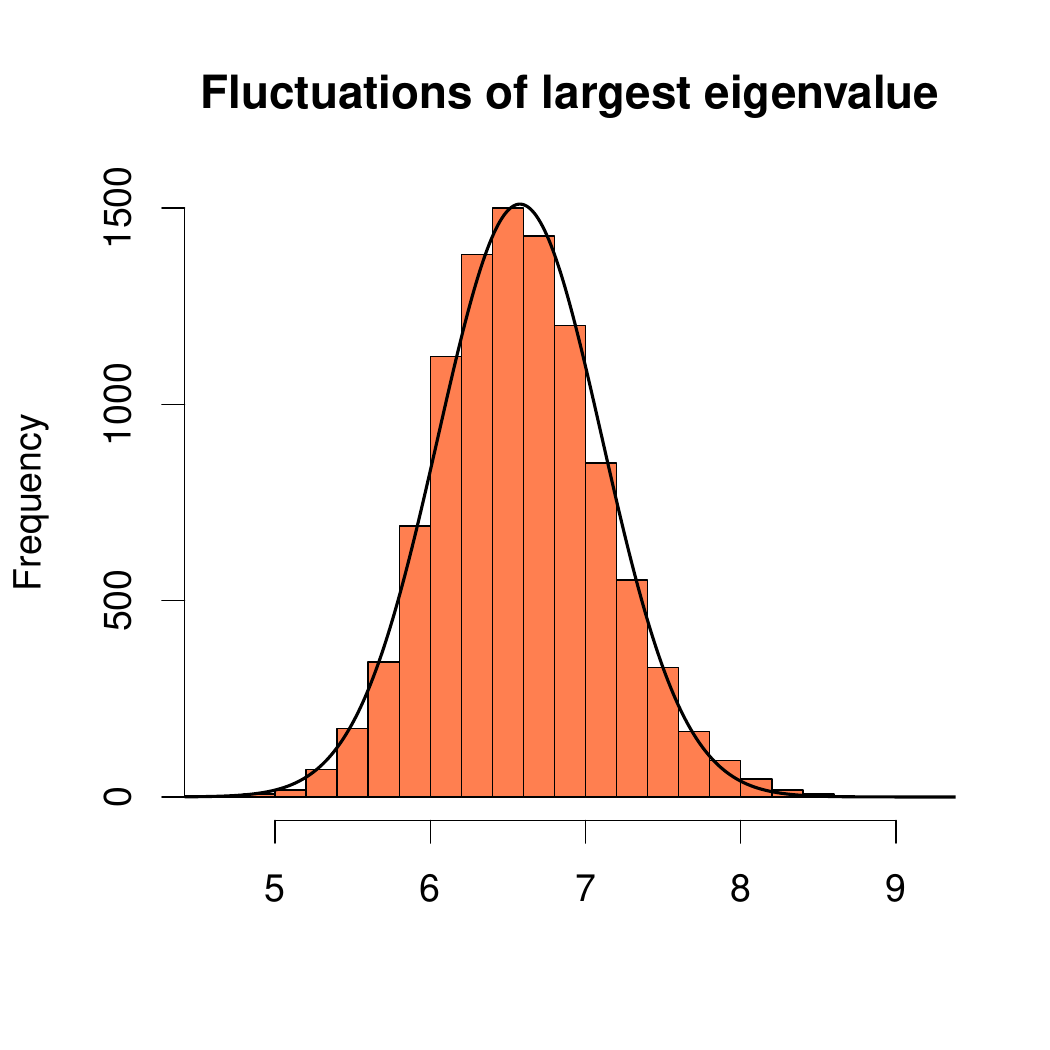}%
\includegraphics[width=0.4\textwidth,page=4]{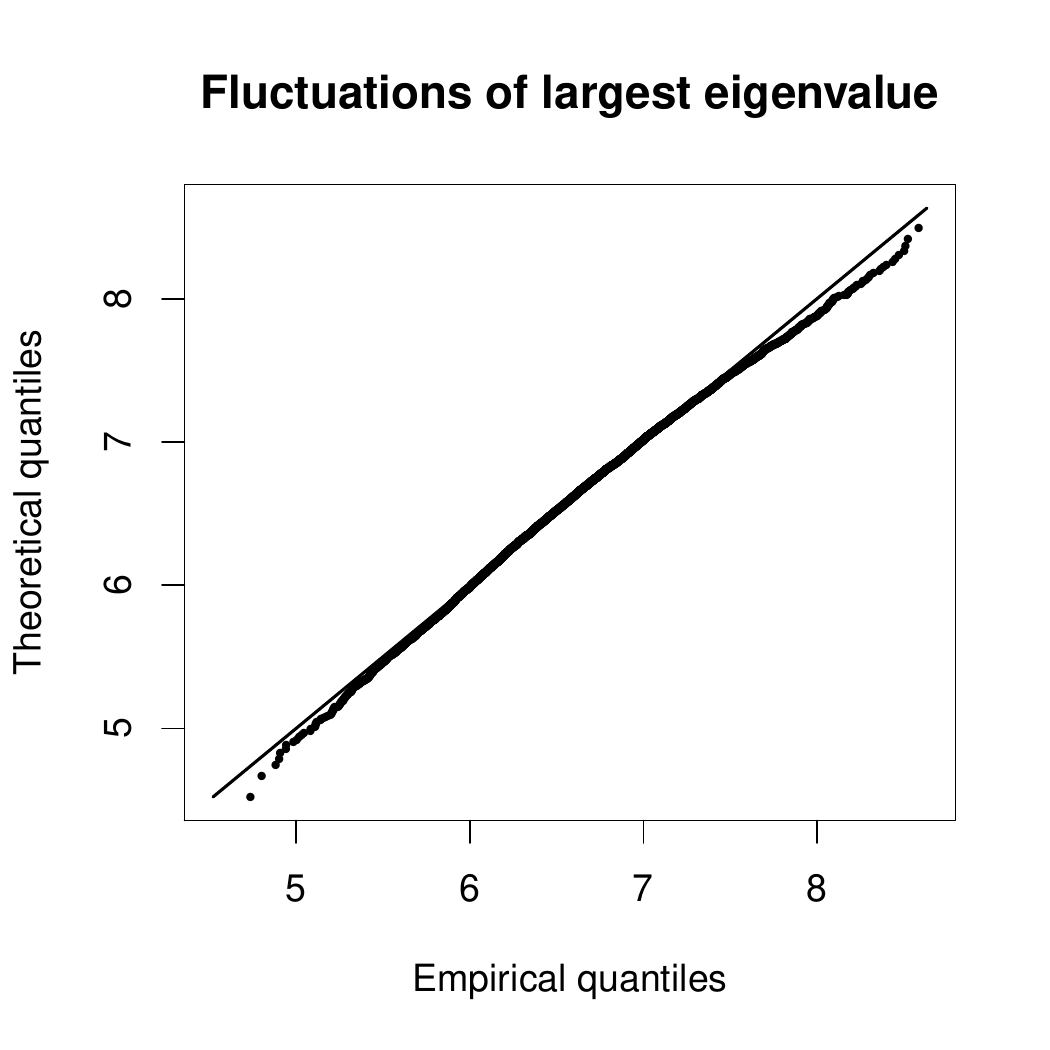}
\caption{Outlier eigenvalue fluctuations in a one-way design.
Displayed are fluctuations of the largest outlier eigenvalue
of $\hSigma_1$ across 10000 simulations, compared with the density function
and quantiles of the Gaussian distribution with mean and variance given in
Theorem \ref{thm:CLT}. The simulated setting is $I=300$ groups of size $J=2$,
$p=300$ traits, and (top) $\Sigma_1=6\e_1\e_1'$ and
$\Sigma_2=29\v\v'+\Id$ where
$\v=\frac{1}{2}\e_1+\frac{\sqrt{3}}{2}\e_2$, or (bottom)
$\Sigma_1=6\,\e_1\e_1'$ and $\Sigma_2=\Id$.}\label{fig:CLT}
\end{figure}

\begin{example}
In the setting of large population spike eigenvalues,
it is illustrative to understand the predictions of
Theorem \ref{thm:spikemapping} using a Taylor expansion. Let us carry this
out for the MANOVA estimator $\hSigma_1$ in the setting of a balanced one-way
design (\ref{eq:oneway}) with $I$ groups of $J$ individuals.

Recalling the form (\ref{eq:MANOVAoneway}) for $\hSigma_1$, the
computation in Proposition \ref{prop:balanced}(b) for general
balanced designs will yield, in this setting, the explicit expressions
\begin{align*}
t_1(\lambda)&=\frac{(I-1)J}{(I-1)J+N(J\sigma_1^2+\sigma_2^2)m_0(\lambda)},\\
t_2(\lambda)&=\frac{I-1}{(I-1)J+N(J\sigma_1^2+\sigma_2^2)m_0(\lambda)}
-\frac{n-I}{(n-I)J-N\sigma_2^2m_0(\lambda)}.
\end{align*}
Suppose first that there is a single large spike eigenvalue
$\mu=\theta+\sigma_1^2$ in $\Sigma_1$,
and no spike eigenvalues in $\Sigma_2$. Theorem \ref{thm:spikemapping} and
the form (\ref{eq:TTheta}) for $T(\lambda)$
indicate that outlier eigenvalues should appear near the locations
\[\Lambda_0=[\;\lambda \in \R \setminus
\supp(\mu_0):m_0(\lambda)t_1(\lambda)=-1/\theta\;].\]
It is known that $m_0$ is injective on $\R
\setminus \supp(\mu_0)$ (see \cite[Theorems 4.1 and 4.2]{silversteinchoi}).
Hence $m_0(\lambda)t_1(\lambda)$ is also injective by the above explicit form,
so $|\Lambda_0| \leq 1$. Applying a Taylor expansion around $\lambda=\infty$,
we obtain from (\ref{eq:MP})
\begin{align*}
m_0(\lambda)&=-\frac{1}{\lambda}-\frac{1}{\lambda^2} \cdot \frac{1}{N}\Tr F
+O(1/\lambda^3)=-\frac{1}{\lambda}-\frac{\sigma_1^2}{\lambda^2}
+O(1/\lambda^3),\\
m_0(\lambda)t_1(\lambda)&=-\frac{1}{\lambda}-\frac{1}{\lambda^2}
\left(\sigma_1^2+\frac{N}{(I-1)J}(J\sigma_1^2+\sigma_2^2)\right)
+O(1/\lambda^3),
\end{align*}
where $N=p-1$. For large $\theta$ and $\mu$, solving
$m_0(\lambda)t_1(\lambda)=-1/\theta$ yields
\[\lambda \approx \theta+\sigma_1^2+c_1=\mu+c_1,
\qquad c_1=\frac{N}{(I-1)J}(J\sigma_1^2+\sigma_2^2).\]
So we expect to observe one outlier with an approximate upward bias of $c_1$.

Next, suppose there is a single large spike eigenvalue $\mu=\theta+\sigma_2^2$
in $\Sigma_2$, and no spike eigenvalues in $\Sigma_1$. Then we expect
outlier eigenvalues near the locations
\[\Lambda_0=[\;\lambda \in \R \setminus
\supp(\mu_0):m_0(\lambda)t_2(\lambda)=-1/\theta\;].\]
Since $m_0(\lambda)$ is injective and the condition
$m_0(\lambda)t_2(\lambda)=-1/\theta$ is quadratic in $m_0(\lambda)$, we obtain
$|\Lambda_0| \leq 2$. Taylor expanding around $|\lambda|=\infty$,
we have after some simplification
\[m_0(\lambda)t_2(\lambda)=-\frac{1}{\lambda^2} \cdot \frac{N}{(I-1)J}
\left(\sigma_1^2+\frac{n-1}{n(J-1)}\sigma_2^2\right)+O(1/|\lambda|^3).\]
Then for large $\theta$, solving $m_0(\lambda)t_2(\lambda)=-1/\theta$ yields
two predicted outlier eigenvalues near
\[\lambda \approx \pm \sqrt{c_2\theta}, \qquad
c_2=\frac{N}{(I-1)J}\left(\sigma_1^2+\frac{n-1}{n(J-1)}\sigma_2^2\right).\]
Let us emphasize that these predictions are in the asymptotic regime where
$n,N \to \infty$ and $\lambda$ is a large but fixed constant, rather than
$\lambda \to \infty$ jointly with $n,N$.

Finally, consider a single spike $\mu_1=\theta_1+\sigma_1^2$ in $\Sigma_1$
and a single spike $\mu_2=\theta_2+\sigma_2^2$ in $\Sigma_2$. Letting the
corresponding spike eigenvectors have inner-product $\rho$, we expect outliers
near
\begin{align*}
\Lambda_0&=\left[\;\lambda:
0=\det\left(-\frac{1}{m_0(\lambda)}\Id_2-
t_1(\lambda)\theta_1\begin{pmatrix} 1 & 0 \\ 0 & 0 \end{pmatrix}
-t_2(\lambda)\theta_2\begin{pmatrix} \rho^2 & \rho\sqrt{1-\rho^2} \\
\rho\sqrt{1-\rho^2} & 1-\rho^2 \end{pmatrix}\right)\right]\\
&=\left[\;\lambda:
0=1+m_0(\lambda)\Big(t_1(\lambda)\theta_1+
t_2(\lambda)\theta_2\Big)+m_0(\lambda)^2
t_1(\lambda)t_2(\lambda)\theta_1\theta_2(1-\rho^2) \;\right].
\end{align*}
This is a cubic condition in $m_0(\lambda)$, so $|\Lambda_0| \leq 3$. Applying
the above Taylor expansions around $\lambda=\infty$, this condition becomes
\[0=1-\frac{\theta_1}{\lambda}-\frac{\theta_1(\sigma_1^2+c_1)}{\lambda^2}
-\frac{\theta_2c_2}{\lambda^2}+\frac{\theta_1\theta_2(1-\rho^2)c_2}{\lambda^3}
+O\left(\frac{\theta_1+\theta_2}{\lambda^3}
+\frac{\theta_1\theta_2}{\lambda^4}\right).\]
In a setting where $\theta_1$ and $\theta_2$ are large and of comparable size,
there is a predicted outlier $\lambda$ near $\theta_1$. More precisely,
expanding the above around $\lambda=\theta_1$, the location of this outlier is
\[\lambda \approx \theta_1+\sigma_1^2+c_1+(\theta_2/\theta_1)\rho^2c_2
=\mu_1+c_1+(\theta_2/\theta_1)\rho^2c_2.\]
Thus the upward bias of this outlier is increased from $c_1$, when there
are no spikes in $\Sigma_2$, to $c_1+(\theta_2/\theta_1)\rho^2c_2$.
\end{example}

We conclude this section by describing two points of contact between
Theorem \ref{thm:spikemapping} and the results of
\cite{benaychgeorgesnadakuditi} and \cite{belinschietal}.

\begin{example}\label{ex:multperturb}
Consider the model (\ref{eq:mixedmodel}) with $X=0$, $k=1$,
$U_1=\Id_n$, and $\sigma_1^2=1$. Then (\ref{eq:F}) yields $F=NB$. Writing
$\alpha_1=\sqrt{N}X_1\Sigma_1^{1/2}$ where $X_1$ has i.i.d.\ $\N(0,1/N)$
entries, we then simply have
\[\hSigma \overset{L}{=}\Sigma_1^{1/2}X_1'FX_1\Sigma_1^{1/2},\qquad
\Sigma_1=\Id+V_1\Theta_1V_1'.\]
The law $\mu_0$ approximates the empirical spectral distribution of $X_1'FX_1$.
Applying (\ref{eq:TSigma}), (\ref{eq:MP}), and (\ref{eq:m0stieltjes}) we obtain
\[m_0(z)t_1(z)=N^{-1}\Tr
\Big[m_0(z)F(\Id+m_0(z)F)^{-1}\Big]=zm_0(z)+1=\int \frac{x}{x-z}\mu_0(dx).\]
The function $-m_0(z)t_1(z)$ is the ``$T$-transform'' of
$\mu_0$. By the form (\ref{eq:TTheta}),
the outlier eigenvalues of $\hSigma$ are predicted by
\[\Lambda_0=\bigcup_{i=1}^{l_1} [\lambda \in \R \setminus \supp(\mu_0):
-m_0(\lambda)t_1(\lambda)=1/\theta_i],\]
where $\theta_1,\ldots,\theta_{l_1}$ are the diagonal entries of $\Theta_1$.
This matches the multiplicative perturbation result of
\cite[Theorem 2.7]{benaychgeorgesnadakuditi}. Depending on $F$, the function
$m_0(\lambda)t_1(\lambda)$ is not necessarily injective over $\lambda \in \R
\setminus \supp(\mu_0)$, and hence a single spike
$\theta_i$ can generate multiple outlier eigenvalues of $\hSigma$ even
in this simple setting.
\end{example}

\begin{example}\label{ex:freeaddition}
Consider the model (\ref{eq:mixedmodel}) with $X=0$, $k=2$, $m_1+m_2=n$,
and the columns of $U_1,U_2$ together forming an orthonormal basis of $\R^n$.
(Thus $U_1'U_1=\Id_{m_1}$, $U_2'U_2=\Id_{m_2}$, and $U_1'U_2=0$.)
Consider $\sigma_1^2=\sigma_2^2=1$ and
\[B=\frac{a_1}{N}U_1U_1'+\frac{a_2}{N}U_2U_2'\]
for any real scalars $a_1,a_2 \neq 0$. Then $U_1'BU_2=0$. Writing
$\alpha_r=\sqrt{N}X_r\Sigma_r^{1/2}$, we have
\[\hSigma \overset{L}{=}\hSigma_1+\hSigma_2,
\qquad \hSigma_r=\Sigma_r^{1/2}(a_rX_r'X_r)\Sigma_r^{1/2} \quad \text{ for }
r=1,2,\]
where $\Sigma_r=\Id+V_r\Theta_rV_r'$, and $X_1$ and $X_2$
have i.i.d.\ $\N(0,1/N)$ entries.

Suppose now that the spike eigenvectors of $\Sigma_1$ and $\Sigma_2$ are
unaligned, i.e.\ $V_1'V_2=0$. Then (\ref{eq:TTheta}) implies that
the outlier eigenvalues of $\hSigma$ are predicted by
\begin{equation}\label{eq:Lambda0subordination}
\Lambda_0=\bigcup_{i=1}^{l_1} 
[\lambda: m_0(\lambda)t_1(\lambda)=-1/\theta_{1,i}] \cup \bigcup_{j=1}^{l_2}
[\lambda: m_0(\lambda)t_2(\lambda)=-1/\theta_{2,j}],
\end{equation}
where $\theta_{1,i}$ and $\theta_{2,j}$ are the diagonal entries of $\Theta_1$
and $\Theta_2$.
In this setting, $F=\diag(a_1\Id_{m_1},a_2\Id_{m_2})$, so (\ref{eq:MP})
and (\ref{eq:TTheta}) yield
\begin{align*}
z&=-\frac{1}{m_0(z)}+\frac{m_1}{N}\frac{a_1}{1+a_1m_0(z)}
+\frac{m_2}{N}\frac{a_2}{1+a_2m_0(z)},\\
t_r(z)&=\frac{m_r}{N}\frac{a_r}{1+a_rm_0(z)} \quad \text{ for } r=1,2
\end{align*}
as the equations defining $m_0,t_1,t_2$.
On the other hand, when $V_1'V_2=0$, the matrices $\hSigma_1$
and $\hSigma_2$ are asymptotically free.
The outlier eigenvalues of the individual matrix $\hSigma_1$ are predicted by
\[\check{\Lambda}_0=\bigcup_{i=1}^{l_1}
[\lambda:\check{m}_0(\lambda)\check{t}_1(\lambda)=-1/\theta_{1,i}]\]
where $\check{m}_0,\check{t}_1$ are defined by
\[z=-\frac{1}{\check{m}_0(z)}+\frac{m_1}{N}\frac{a_1}{1+a_1\check{m}_0(z)},
\qquad \check{t}_1(z)=\frac{m_1}{N}\frac{a_1}{1+a_1\check{m}_0(z)}\]
and $\check{m}_0(z)$ is the Stieltjes transform modeling the
spectral distribution of $a_1X_1'X_1$.
Then $m_0(z)=\check{m}_0(\omega_1(z))$ and
$t_1(z)=\check{t}_1(\omega_1(z))$, where
\[\omega_1(z)=z-\frac{m_2}{N}\frac{a_2}{1+a_2m_0(z)}\]
is the first subordination function with respect to the free additive
convolution. Then the first set on the right
of (\ref{eq:Lambda0subordination}) is
simply $\omega_1^{-1}(\check{\Lambda}_0)$. A similar statement holds for the
second set of (\ref{eq:Lambda0subordination}), the second subordination
function, and the outlier eigenvalues of $\hSigma_2$,
and our results coincide with the prediction
of \cite[Theorem 2.1]{belinschietal}.
\end{example}

\section{Estimation of principal eigenvalues and
eigenvectors}\label{sec:estimation}

The results of the preceding section indicate that each outlier
eigenvalue/eigenvector of $\hSigma$ may be interpreted as estimating
an eigenvalue/eigenvector of a surrogate matrix
(\ref{eq:surrogate}). When there is no high-dimensional noise,
$\sigma_1^2=\ldots=\sigma_k^2=0$, we may verify that $t_r(\lambda)=\Tr U_r'BU_r$
for each $r=1,\ldots,k$ and any $\lambda$. In this setting,
if $\hSigma$ is an unbiased MANOVA estimate of a single component
$\Sigma_r$, then (\ref{eq:unbiasedestimation}) implies that the surrogate matrix
is also simply $\Sigma_r$.

In the presence of high-dimensional noise, this is no longer true. Even for the
MANOVA estimate $\hSigma$ of $\Sigma_r$, the surrogate matrix
may depend on multiple variance components $\Sigma_1,\ldots,\Sigma_k$.
We propose an alternative algorithm for estimating
eigenvalues and eigenvectors of $\Sigma_r$, based on the idea of searching
for matrices $\hSigma=Y'BY$ where this surrogate depends only on $\Sigma_r$.
Figure \ref{fig:estimation} depicts
differences between the MANOVA eigenvector and our estimated eigenvector in
several examples for the one-way model.

\begin{figure}
\includegraphics[width=0.33\textwidth,page=1]{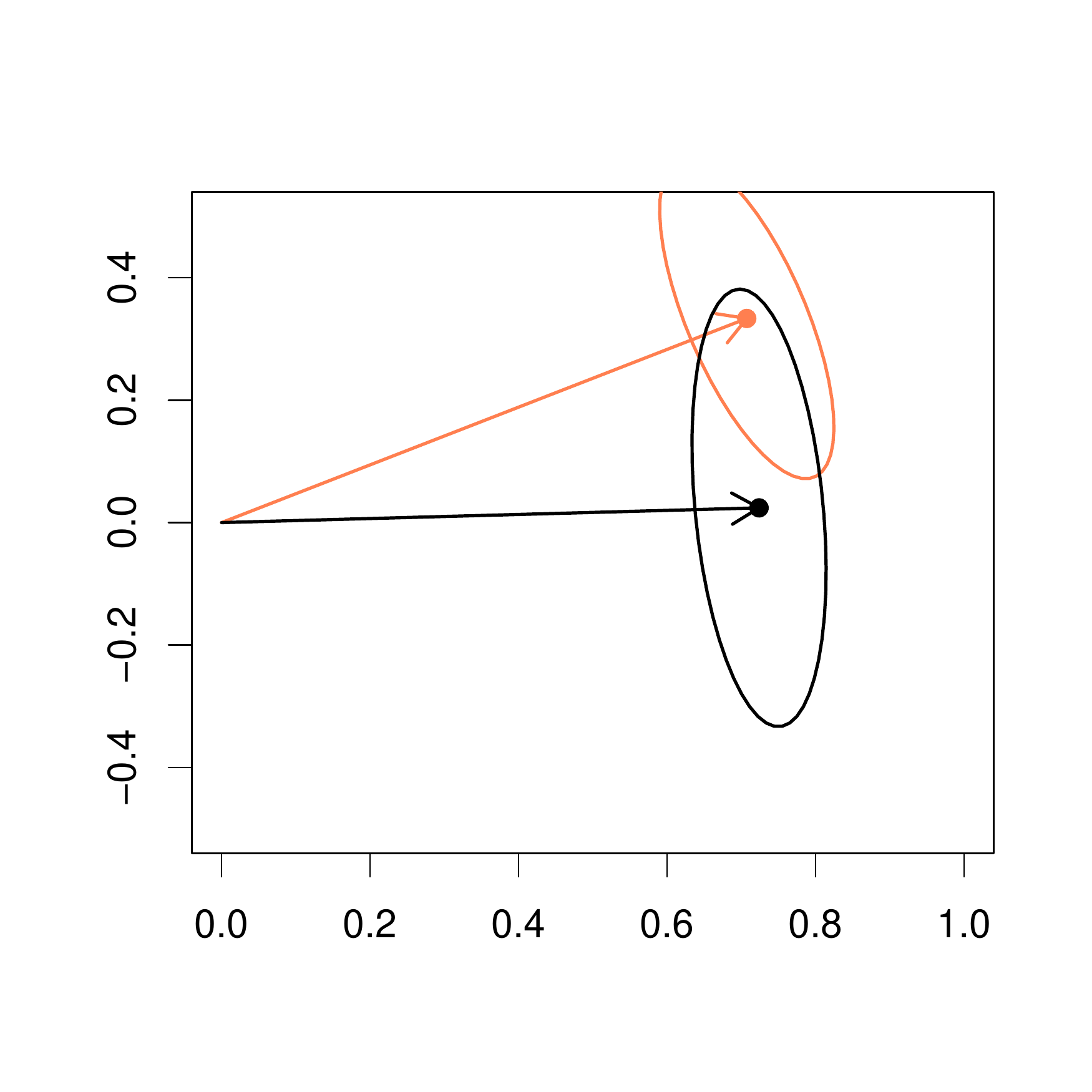}%
\includegraphics[width=0.33\textwidth,page=2]{estimation.pdf}%
\includegraphics[width=0.33\textwidth,page=3]{estimation.pdf}
\caption{Estimates of the principal eigenvector of $\Sigma_1$ in a one-way
design. The population covariances are $\Sigma_1=\mu \e_1\e_1'$ and
$\Sigma_2=29\v\v'+\Id$, where $\v=\frac{1}{2}\e_1+\frac{\sqrt{3}}{2}\e_2$
and (left) $\mu=6$, (middle) $\mu=8$, or (right) $\mu=10$. Means
and 90\% ellipsoids across 100 simulations are shown for the first two
coordinates of the unit-norm leading MANOVA eigenvector (red) and of the
unit-norm estimate of
Algorithm \ref{alg:spikes} (black). The design is $I=150$ groups of size $J=2$
with $p=600$ traits.}\label{fig:estimation}
\end{figure}

We implement this algorithmic idea as follows: Fix $k$ symmetric matrices
$B_1,\ldots,B_k \in \R^{n \times n}$ satisfying Assumption \ref{assump:main}(c).
For $\ba=(a_1,\ldots,a_k) \in \R^k$, denote
\[B(\ba)=\sum_{r=1}^k a_rB_r.\]
Let $F(\ba)$ be the matrix defined in (\ref{eq:F}) for $B=B(\ba)$, let
$\hSigma(\ba)=Y'B(\ba)Y$, and let $\mu_0(\ba)$, $m_0(z,\ba)$,
and $t_r(z,\ba)$ be the law $\mu_0$ and the functions
$m_0(z)$ and $t_r(z)$ defined
with $F=F(\ba)$. We search for coefficients $\ba \in \R^k$
where $\hSigma(\ba)$ has an outlier
eigenvalue $\hl$ satisfying $t_s(\hl,\ba)=0$ for all
$s \neq r$. At any such pair $(\hl,\ba)$, the surrogate matrix $t(\hl) \cdot
\Sigma$ depends only on $\Sigma_r$, and we have
$T(\hl,\ba)=\hl\Id-t_r(\hl,\ba)\Sigma_r$ by (\ref{eq:TSigma}). By Theorem
\ref{thm:spikemapping}, we expect $\hl$ to be close to a value $\lambda$ where
\begin{equation}\label{eq:algorithmmotivation}
0=\det T(\lambda,\ba) \approx \det (\hl\Id-t_r(\hl,\ba)\Sigma_r).
\end{equation}
Thus, we estimate an eigenvalue $\mu$ of $\Sigma_r$ by
$\hmu=\hl/t_r(\hl,\ba)$. Furthermore, by Theorem \ref{thm:eigenvectors}, we
expect the eigenvector $\hv$ of $\hSigma(\ba)$ corresponding to $\hl$ to satisfy
\[P_\cS \hv \approx (\w'\partial_\lambda T(\lambda,\ba)\w)^{-1/2}\w,\]
where $\w$ is the null vector of $T(\lambda,\ba)$. By
(\ref{eq:algorithmmotivation}), we expect $\w \approx \v$ where $\v$
is the eigenvector of $\Sigma_r$ corresponding to $\mu$. Thus, we
estimate $\v$ by $\hv$.

\begin{algorithm}[t]
\begin{algorithmic}
\STATE{Initialize $\cM=\emptyset$. Fix $\delta>0$ a small constant.}
\FOR{each $\ba \in S^{k-1}$ and each $\hl \in \spec(\hSigma(\ba)) \cap
\I_\delta(\ba)$}
  \IF{$t_s(\hl,\ba)=0$ for all $s \in \{1,\ldots,k\} \setminus \{r\}$}
    \STATE{Add $(\hmu,\hv)$ to $\cM$, where $\hmu=\hl/t_r(\hl,\ba)$ and
        $\hv$ is the unit eigenvector such that $\hSigma(\ba)\hv=\hl\hv$.}
  \ENDIF
\ENDFOR
\STATE{Return $\cM$}
\end{algorithmic}
\caption{Algorithm for estimating eigenvalues and eigenvectors of $\Sigma_r$}
\label{alg:spikes}
\end{algorithm}

The procedure is summarized in Algorithm \ref{alg:spikes}. We note that the
combinations $\ba$ where $t_s(\lambda,\ba) \approx 0$ for $s \neq r$ are not
known a priori---in particular, they depend on the unknown spike
eigenvalues and eigenvectors to be estimated. Hence we search for such values
$\ba \in \R^k$. By scale invariance, we restrict to $\ba$ on the unit sphere
\[S^{k-1}=\{\ba \in \R^k:\|\ba\|=1\}.\]
We further restrict to outlier eigenvalues $\hl \in
\spec(\hSigma(\ba))$ which fall above $\supp(\mu_0)$, belonging to
\[\I_\delta(\ba)=\{x \in \R:x \geq y+\delta \text{ for all } y \in
\supp(\mu_0(\ba))\}.\]
We note that outliers falling below $\supp(\mu_0)$ 
will be identified as corresponding to $-\ba \in S^{k-1}$, and for simplicity
of the procedure, we ignore any outliers that fall between intervals of
$\supp(\mu_0(\ba))$.

\begin{figure}
\centering
\includegraphics[width=0.5\textwidth,page=19]{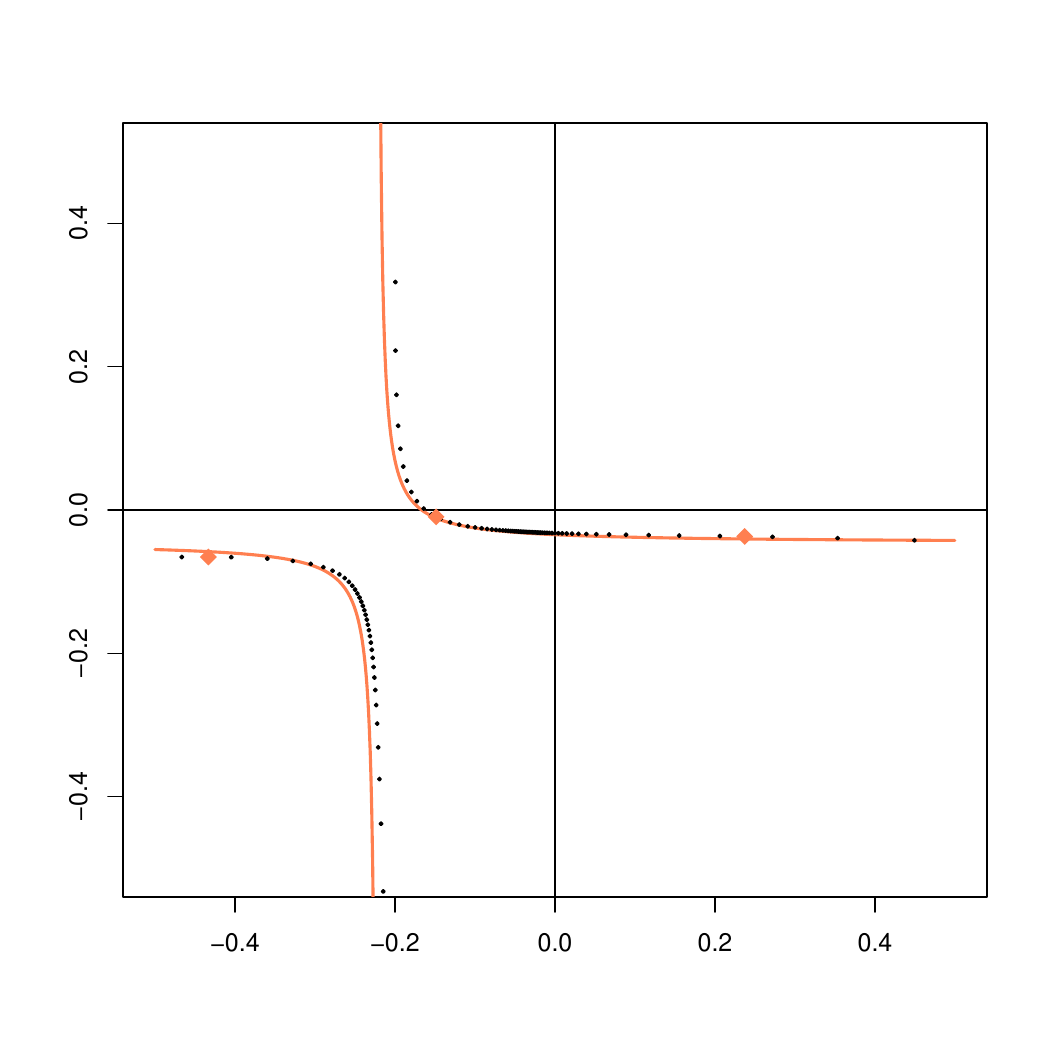}
\caption{Illustration of Algorithm \ref{alg:spikes} for the one-way design,
where $k=2$. The setting is the same as in Figure \ref{fig:outliers}. The
red curve depicts the locus $\mathcal{L}$ from (\ref{eq:Ltrue}) on the
$(s_1,s_2)$ plane, which has one $s_1$-intercept at $(-1/6,0)$ and one
$s_2$-intercept at $(0,-1/29)$. Black points show values in
$\widehat{\mathcal{L}}$ corresponding to $\ba=(a_1,a_2)$ belonging to a grid of
100 equispaced points on the unit circle. The three points of
$\widehat{\mathcal{L}}$ corresponding to the three outliers of the MANOVA
estimate $\hSigma_1$, where $(a_1,a_2)=\pm (1/J,-1/J)$, are
depicted in red.}\label{fig:Tplot}
\end{figure}

One may understand the behavior of Algorithm \ref{alg:spikes} by plotting the
values
\begin{equation}\label{eq:Lobserved}
\widehat{\mathcal{L}}=\left\{
m_0(\hl,\ba) \cdot (t_1(\hl,\ba),\ldots,t_k(\hl,\ba)):\;
\ba \in S^{k-1}, \;\hl \in \spec(\hSigma(\ba)) \cap \I_\delta(\ba)\right\}.
\end{equation}
This is illustrated for an example
of the one-way design in Figure \ref{fig:Tplot}. By Theorem
\ref{thm:spikemapping}, we expect these values to fall close to
\[m_0(\lambda,\ba) \cdot (t_1(\lambda,\ba),\ldots,t_k(\lambda,\ba)),\]
where $\lambda$ is the deterministic prediction for the location of $\hl$,
satisfying $0=\det T(\lambda,\ba)$. By this condition and the
form (\ref{eq:TTheta}) for $T$, these values belong to the locus
\begin{equation}\label{eq:Ltrue}
\mathcal{L}=\left\{(s_1,\ldots,s_k) \in \R^k:\,0=\det\left(\Id+\sum_{r=1}^k
s_rV_r\Theta_rV_r'\right)\right\},
\end{equation}
which does not depend on $\ba$ and is defined solely by the spike parameters
$\Theta_1,\ldots,\Theta_k$ and $V_1,\ldots,V_k$. This is depicted also in Figure
\ref{fig:Tplot}. (We have picked a simulation to display in Figure
\ref{fig:Tplot} where $\widehat{\mathcal{L}}$ and $\mathcal{L}$ are
particularly close, for purposes of illustration.)
The spike values $\theta$ on the diagonal of $\Theta_r$
are in 1-to-1 correspondence with the points
$(0,\ldots,0,-1/\theta,0,\ldots,0) \in \mathcal{L}$
which fall on the $r^{\text{th}}$
coordinate axis. Algorithm \ref{alg:spikes} may be
understood as estimating these intercepts by the intercepts of the
observed locus $\widehat{\mathcal{L}}$.

We have written Algorithm \ref{alg:spikes} in the idealized setting
where we search over all $\ba \in S^{k-1}$. In practice, we discretize
$S^{k-1}$ as in Figure \ref{fig:Tplot} and search over this discretization for
pairs $(\hl,\ba)$ where $t_s(\hl,\ba) \approx 0$ for all $s \neq r$. We then
numerically refine each located pair $(\hl,\ba)$. Computing the values
$t_r(\hl,\ba)$ and the lower endpoint of $\I_\delta(\ba)$ requires knowledge of
the noise variances $\sigma_1^2,\ldots,\sigma_k^2$. These computations are
particularly simple in balanced classification designs, and we discuss this
in Section \ref{sec:balanced}. If $\sigma_1^2,\ldots,\sigma_k^2$
are unknown, they may be replaced by $1/n$-consistent
estimates, for example $\hat{\sigma}_r^2=p^{-1}\Tr \hSigma_r$ where $\hSigma_r$
is the unbiased MANOVA estimate for $\Sigma_r$. (See
\cite[Proposition 2.13]{fanjohnstoneedges} for a proof. In practice, large
outliers of $\hSigma_r$ may be removed before computing the trace.) The unknown
quantity $N=p-L$ may be replaced by the dimension $p$.

We prove the following theoretical guarantee for this procedure, for simplicity 
in the setting where $\Sigma_r$ has separated eigenvalues. Define $s:\R^k \to
\R^k$ by
\begin{equation}\label{eq:s}
s(\ba)=(s_1(\ba),\ldots,s_k(\ba)),\qquad s_r(\ba)=\frac{1}{N\sigma_r^2}
\Tr_r\Big(F(\ba)[\Id+F(\ba)]^{-1}\Big).
\end{equation}
As $F(m_0 \cdot \ba)=m_0 \cdot F(\ba)$, this
function satisfies $s(m_0(\lambda,\ba) \cdot \ba)=m_0(\lambda,\ba) \cdot
(t_1(\lambda,\ba),\ldots,t_k(\lambda,\ba))$.
To guarantee that the algorithm does not make duplicate estimates, we require
$B_1,\ldots,B_k$ to be chosen such that $s$ is injective in the following
quantitative sense.
\begin{assumption}\label{assump:estimation}
There exists a constant $c>0$ such that for any $\ba_1,\ba_2 \in \R^k$
where $\Id+F(\ba_1)$ and $\Id+F(\ba_2)$ are invertible,
\[\|s(\ba_1)-s(\ba_2)\| \geq c\,\frac{\|\ba_1-\ba_2\|}{(1+\|\ba_1\|)
(1+\|\ba_2\|)}.\]
\end{assumption}
We will verify in Section \ref{sec:balanced} that this condition holds for
balanced classification designs, where $B_1,\ldots,B_k$ are the projections
corresponding to the canonical mean-squares.

\begin{theorem}[Spike estimation]\label{thm:estimation}
Fix $\delta,\tau>0$ and $r \in \{1,\ldots,k\}$. Suppose
Assumptions \ref{assump:main} and \ref{assump:estimation} hold for
$B_1,\ldots,B_k$. Suppose furthermore that
the diagonal values $\theta_i$ of $\Theta_r$
satisfy $\theta_i \geq \tau$ and $|\theta_i-\theta_j| \geq \tau$
for all $1 \leq i \neq j \leq l_r$.
Then there exists a constant $c_0>0$ (not depending on
$\bar{C}$ in Assumption \ref{assump:main}) such that the following holds:

Let $\cM$ be the output of Algorithm
\ref{alg:spikes} with parameter $\delta$ for estimating the spikes of
$\Sigma_r$. Let
$\hcE=[\hmu:(\hmu,\hv) \in \cM]$ and $\hcV=[\hv:(\hmu,\hv) \in \cM]$
be the estimated eigenvalues and eigenvectors.
Then, for any $\eps,D>0$ and all $n \geq n_0(\delta,\tau,\eps,D)$,
\begin{enumerate}[(a)]
\item With probability at least $1-n^{-D}$,
there is a subset $\cE \subset \spec(\Sigma_r)$ containing all
eigenvalues greater than $c_0$ such that
\[\ordereddist(\hcE,\cE)<n^{-1/2+\eps}.\]
\item On the event of part (a), for any $\mu \in \cE$, let $\v$ be the unit
eigenvector where $\Sigma_r\v=\mu\v$, and let $(\hmu,\hv) \in \cM$ be such that
$\hmu-\mu<n^{-1/2+\eps}$. Then for some scalar value $\alpha \in (0,1]$
and choice of sign for $\v$,
\[\|P_\cS \hv-\alpha \v\|<n^{-1/2+\eps}.\]
\item For each $\hv \in \hcV$,
$P_{\cS^\perp}\hv/\|P_{\cS^\perp}\hv\|$ is independent of $P_\cS
\hv$ and uniformly distributed over unit vectors in $\cS^\perp$.
\end{enumerate}
\end{theorem}

In the presence of high-dimensional noise, the eigenvector estimate $\hv$
remains inconsistent for $\v$. However, asymptotically as
$n,p \to \infty$, parts (b) and (c) indicate that $\hv$ is not biased in a
particular direction away from $\v$.
Note that in part (a), some lower bound $c_0$ for the size of the population
spike eigenvalue is necessary to guarantee estimation of this spike,
as otherwise it might not produce an outlier in 
any matrix $\hSigma(\ba)$. (In this case, a portion of the true locus
$\mathcal{L}$ in (\ref{eq:Ltrue}) may not be tracked by the observed locus
$\widehat{\mathcal{L}}$.)

\begin{example}
We explore in simulations the accuracy of this procedure for estimating
eigenvalues and eigenvectors of $\Sigma_1$ in two finite-sample
settings of the one-way model (\ref{eq:oneway}), corresponding to the designs
\begin{align*}
D_1&:\,n=600,\,p=300,\,I=300,\,J=2\\
D_2&:\,n=300,\,p=600,\,I=150,\,J=2
\end{align*}
In all simulations, we take $\sigma_1^2=0$ and $\sigma_2^2=1$. In particular,
$\Sigma_1$ is low-rank, as hypothesized for genetic covariances of
high-dimensional trait sets \cite{walshblows,blowsetal}.
For both designs, we fix the tuning parameter $\delta=0.5$.

We first consider a rank-one matrix $\Sigma_1=\mu \e_1\e_1'$ for various
settings of $\mu$ between 2 and 10, and
$\Sigma_2=\Id$ with no spike. The following tables display the mean and standard
error of $\hmu$ estimated by Algorithm \ref{alg:spikes}, and of the 
alignment $\hv'\e_1$ of the estimated eigenvector.
Displayed also are the corresponding quantities for the leading
eigenvalue/eigenvector of the MANOVA estimate $\hSigma_1$.
We observe in all cases that Algorithm
\ref{alg:spikes} corrects a bias in the MANOVA eigenvalue, and the alignment
$\hv'\e_1$ is approximately the same as for the MANOVA eigenvector.
Algorithm \ref{alg:spikes} never estimates more than one spike for
$\Sigma_1$ in this setting; however, if $\mu$ is small, it may sometimes
estimate 0 spikes. We display also the percentage of simulations in which a
spike was estimated. For $\mu=2$ under Design $D_2$, the
predicted outlier is less than $\delta=0.5$ away from the edge of the
spectrum, and Algorithm \ref{alg:spikes} never estimated this spike.
\begin{center}
{\bf Design $D_1$}\\*
\begin{tabular}{l|rrrrr}
  \hline
 & $\mu=2$ & $\mu=4$ & $\mu=6$ & $\mu=8$ & $\mu=10$ \\ 
  \hline
Eigenvalue, MANOVA & 2.70 (0.19) & 4.60 (0.36) & 6.56 (0.52) & 8.53 (0.69) & 10.51 (0.85) \\ 
  Alignment $\e_1$, MANOVA & 0.85 (0.02) & 0.93 (0.01) & 0.96 (0.01) & 0.97 (0.00) & 0.97 (0.00) \\ 
\hline
  Eigenvalue, estimated & 2.00 (0.20) & 3.98 (0.37) & 5.98 (0.53) & 7.97 (0.69) & 9.96 (0.85) \\ 
  Alignment $\e_1$, estimated & 0.84 (0.02) & 0.93 (0.01) & 0.95 (0.01) & 0.97 (0.00) & 0.97 (0.00) \\ 
  Percent estimated & 98 & 100 & 100 & 100 & 100 \\ 
   \hline
\end{tabular}\\
\vspace{0.1in}
{\bf Design $D_2$}\\*
\begin{tabular}{l|rrrrr}
  \hline
 & $\mu=2$ & $\mu=4$ & $\mu=6$ & $\mu=8$ & $\mu=10$ \\ 
  \hline
Eigenvalue, MANOVA & 4.65 (0.23) & 6.31 (0.49) & 8.18 (0.72) & 10.10 (0.95) & 12.04 (1.19) \\ 
  Alignment $\e_1$, MANOVA & 0.58 (0.07) & 0.78 (0.03) & 0.85 (0.02) & 0.88 (0.02) & 0.90 (0.01) \\ 
\hline
  Eigenvalue, estimated & NA & 4.02 (0.46) & 5.89 (0.75) & 7.87 (0.98) & 9.84 (1.20) \\ 
  Alignment $\e_1$, estimated & NA & 0.76 (0.03) & 0.84 (0.02) & 0.88 (0.02) & 0.90 (0.01) \\ 
  Percent estimated & 0 & 87 & 100 & 100 & 100 \\ 
   \hline
\end{tabular}
\end{center}
\vspace{0.1in}

Next, we consider $\Sigma_1=0$ and $\Sigma_2=\theta \v\v'+\Id$ for a unit
vector $\v$ and for $\mu=\theta+1 \in \{10,20,30\}$. In both designs
$D_1$ and $D_2$, this produces one positive and one negative outlier eigenvalue
in the MANOVA estimate $\hSigma_1$. The tables below show the percentages of
simulations in which a spurious spike eigenvalue is estimated by Algorithm
\ref{alg:spikes} for $\Sigma_1$. In such cases, there is enough deviation of
the observed locus $\widehat{\mathcal{L}}$ from the true locus $\mathcal{L}$
(which is the horizontal line $s_2=-1/\theta$) to produce a spurious
intercept where $t_2(\hl,\ba)=0$, and the algorithm interprets this
as an alignment of the spike in $\Sigma_2$ with a small
spike in $\Sigma_1$. We find that the spurious points $(\hl,\ba)$ where
$t_2(\hl,\ba)=0$ occur for $\hl$ close to the edges of $\supp(\mu_0(\ba))$, and
this error percentage may be reduced in finite samples by setting a more
conservative choice of $\delta$, if desired.

\begin{center}
{\bf Design $D_1$}\\*
\begin{tabular}{l|rrr}
  \hline
 & $\mu=10$ & $\mu=20$ & $\mu=30$ \\ 
  \hline
  Percent spurious & 2 & 8 & 18 \\ 
   \hline
\end{tabular}\\
\vspace{0.1in}
{\bf Design $D_2$}\\*
\begin{tabular}{l|rrr}
  \hline
 & $\mu=10$ & $\mu=20$ & $\mu=30$ \\ 
  \hline
  Percent spurious & 0 & 8 & 15 \\ 
   \hline
\end{tabular}
\end{center}
\vspace{0.1in}

Next, we consider $\Sigma_1=\mu \e_1\e_1'$ and
$\Sigma_2=29\v\v'+\Id$ for $\v=\frac{1}{2}\e_1+\frac{\sqrt{3}}{2}\e_2$, which
forms a 60-degree alignment angle with $\e_1$. Displayed are the statistics for
the largest estimated eigenvalue/eigenvector and largest MANOVA
eigenvalue/eigenvector. Displayed also are the inner-product
alignments with the direction $\e_2$ (where signs are chosen so that the
estimated eigenvectors have positive $\e_1$ coordinate). The spike
in $\Sigma_2$ causes the MANOVA eigenvector to be biased towards
$\v$, and it also increases the bias and standard error of the MANOVA
eigenvalue. In settings of small $\mu$ when Algorithm \ref{alg:spikes} does
not always estimate a spike, the values $\hmu$ and $\hv'\e_2$ have a selection
bias among the simulations where estimation occurs. For the remaining settings,
$\hmu$ and $\hv'\e_2$ are nearly unbiased for the true values $\mu$ and 0,
and the alignments $\hv'\e_1$ are similar to those of the MANOVA eigenvectors.

\begin{center}
{\bf Design $D_1$}\\*
\begin{tabular}{l|rrrrr}
  \hline
 & $\mu=2$ & $\mu=4$ & $\mu=6$ & $\mu=8$ & $\mu=10$ \\ 
  \hline
Eigenvalue, MANOVA & 4.59 (1.14) & 5.70 (1.14) & 7.28 (1.15) & 9.07 (1.22) & 10.93 (1.33) \\ 
  Alignment $\e_1$, MANOVA & 0.57 (0.07) & 0.80 (0.06) & 0.89 (0.04) & 0.93 (0.02) & 0.95 (0.01) \\ 
  Alignment $\e_2$, MANOVA & 0.47 (0.11) & 0.26 (0.16) & 0.14 (0.15) & 0.09 (0.12) & 0.06 (0.10) \\ 
  \hline
  Eigenvalue, estimated & 2.67 (1.09) & 4.18 (1.01) & 6.11 (1.07) & 8.06 (1.17) & 10.03 (1.30) \\ 
  Alignment $\e_1$, estimated & 0.63 (0.10) & 0.83 (0.04) & 0.90 (0.02) & 0.93 (0.02) & 0.95 (0.01) \\ 
  Alignment $\e_2$, estimated & 0.10 (0.25) & 0.01 (0.19) & 0.01 (0.15) & 0.00 (0.12) & 0.00 (0.10) \\ 
  Percent estimated & 70 & 100 & 100 & 100 & 100 \\ 
   \hline
\end{tabular}\\
\vspace{0.1in}
{\bf Design $D_2$}\\*
\begin{tabular}{lrrrrr}
  \hline
 & $\mu=2$ & $\mu=4$ & $\mu=6$ & $\mu=8$ & $\mu=10$ \\ 
  \hline
Eigenvalue, MANOVA & 8.79 (1.52) & 9.49 (1.64) & 10.57 (1.74) & 11.98 (1.85) & 13.59 (1.99) \\ 
  Alignment $\e_1$, MANOVA & 0.44 (0.06) & 0.58 (0.06) & 0.71 (0.05) & 0.79 (0.04) & 0.84 (0.03) \\ 
  Alignment $\e_2$, MANOVA & 0.53 (0.07) & 0.44 (0.10) & 0.33 (0.12) & 0.24 (0.12) & 0.18 (0.12) \\ 
  \hline
  Eigenvalue, estimated & 5.15 (1.37) & 4.84 (1.41) & 6.28 (1.56) & 8.21 (1.72) & 10.15 (1.91) \\ 
  Alignment $\e_1$, estimated & 0.39 (0.05) & 0.60 (0.06) & 0.72 (0.04) & 0.80 (0.03) & 0.84 (0.03) \\ 
  Alignment $\e_2$, estimated & 0.34 (0.11) & 0.09 (0.17) & 0.02 (0.16) & 0.02 (0.14) & 0.01 (0.13) \\ 
  Percent estimated & 22 & 77 & 100 & 100 & 100 \\ 
   \hline
\end{tabular}
\end{center}
\vspace{0.1in}

Finally, we consider a setting with multiple spikes. We set $\Sigma_1$ to be
of rank 5, with eigenvalues $(10,8,6,4,2)$. We set $\Sigma_2$ to have 5
eigenvalues equal to 30 and remaining eigenvalues equal to 1, with the former
5-dimensional subspace having a 60-degree alignment angle with each spike
eigenvector of $\Sigma_1$. The tables below display statistics for the
five largest estimated and MANOVA eigenvalues in this setting. We observe that
Algorithm \ref{alg:spikes} reduces the bias of the MANOVA eigenvalues,
although a positive bias persists at these sample sizes.

\begin{center}
{\bf Design $D_1$}\\*
\begin{tabular}{l|rrrrr}
  \hline
 & $\mu=10$ & $\mu=8$ & $\mu=6$ & $\mu=4$ & $\mu=2$ \\ 
  \hline
Eigenvalue, MANOVA & 12.06 (1.10) & 9.70 (1.01) & 7.60 (0.96) & 5.87 (0.74) & 4.53 (0.55) \\ 
\hline
  Eigenvalue, estimated & 11.08 (1.12) & 8.65 (1.01) & 6.38 (0.95) & 4.36 (0.76) & 2.80 (0.57) \\ 
  Percent estimated & 100 & 100 & 100 & 100 & 97 \\ 
   \hline
\end{tabular}\\
\vspace{0.1in}
{\bf Design $D_2$}\\*
\begin{tabular}{l|rrrrr}
  \hline
 & $\mu=10$ & $\mu=8$ & $\mu=6$ & $\mu=4$ & $\mu=2$ \\ 
  \hline
Eigenvalue, MANOVA & 15.79 (1.61) & 12.94 (1.15) & 11.06 (1.00) & 9.21 (0.82) & 7.80 (0.73) \\ 
\hline
  Eigenvalue, estimated & 12.07 (1.68) & 8.95 (1.19) & 6.77 (1.03) & 4.74 (0.86) & 3.94 (0.53) \\ 
  Percent estimated & 100 & 100 & 100 & 98 & 37 \\ 
   \hline
\end{tabular}
\end{center}
\end{example}

\section{Balanced classification designs}\label{sec:balanced}
We consider the special example of model (\ref{eq:mixedmodel}) corresponding
to balanced classification designs. In these designs, there is a canonical
choice of matrices $B_1,\ldots,B_k$ for Algorithm \ref{alg:spikes} by
considerations of sufficiency, and Assumption \ref{assump:estimation}
may be explicitly verified for this choice. The quantities
$t_r(z)$ and $w_{rs}(z)$ also have explicit forms, which
we record here to facilitate numerical implementations. 

To motivate the general discussion, we first give several examples.
\begin{example}\label{example:oneway}
Consider the one-way model (\ref{eq:oneway}) in the balanced setting with 
$I$ groups of equal size $J$. We assume $J \geq 2$ is a fixed constant.
The canonical mean-square matrices of this model are defined by
\[\MS_1=\frac{1}{I-1}\sum_{i=1}^I \sum_{j=1}^J
(\bar{\y}_i-\bar{\y})(\bar{\y}_i-\bar{\y})', \qquad
\MS_2=\frac{1}{n-I}\sum_{i=1}^I \sum_{j=1}^J
(\y_{i,j}-\bar{\y}_i)(\y_{i,j}-\bar{\y}_i)',\]
where $\bar{\y}_i \in \R^p$ and $\bar{\y} \in \R^p$ denote the sample means
in group $i$ and across all groups. The MANOVA estimators are
\cite{searleetal}
\[\hSigma_1=\frac{1}{J}\MS_1-\frac{1}{J}\MS_2, \qquad \hSigma_2=\MS_2.\]
Recall that the one-way model may be written in the matrix form
\[Y=\one_n\bmu'+U_1\alpha_1+\alpha_2,\]
where $U_1$ is defined in (\ref{eq:onewaymatrix}).
Defining orthogonal projections $\pi_1$ and $\pi_2$ onto
$\col(U_1)\ominus \col(\one_n)$ and $\R^n \ominus \col(U_1)$, the above
mean-squares may be written as
\[\MS_1=Y'\frac{\pi_1}{I-1}Y, \qquad \MS_2=Y'\frac{\pi_2}{n-I}Y.\]
The MANOVA estimators then take the equivalent form of (\ref{eq:MANOVAoneway}).
\end{example}

\begin{example}\label{example:twoway}
Consider a nested two-way model with $I$ groups,
each group consisting of $J$ subgroups, and each subgroup consisting of $K$
individuals. Traits for individual $k$ in group $(i,j)$ are modeled as
\begin{equation}\label{eq:nestedtwoway}
\y_{i,j,k}=\bmu+\balpha_i+\bbeta_{i,j}+\beps_{i,j,k},
\quad \balpha_i \overset{iid}{\sim} \N(0,\Sigma_1),
\quad \bbeta_{i,j} \overset{iid}{\sim} \N(0,\Sigma_2),
\quad \beps_{i,j,k} \overset{iid}{\sim} \N(0,\Sigma_3).
\end{equation}
This corresponds to the North Carolina I design of \cite{comstockrobinson},
where individuals have a half-sibling relation within groups and a full-sibling
relation within subgroups. We assume $J,K \geq 2$ are fixed constants.

This model may be written in the matrix form
\[Y=\one_n\bmu'+U_1\alpha_1+U_2\alpha_2+\alpha_3,\]
where $\y_{i,j,k}$, $\balpha_i$, $\bbeta_{i,j}$, and $\beps_{i,j,k}$ are stacked
as the rows of $Y$, $\alpha_1$, $\alpha_2$, and $\alpha_3$, and the incidence
matrices are given by
\[U_1=\Id_I \otimes \one_{JK}, \qquad U_2=\Id_{IJ} \otimes \one_K.\]
Defining orthogonal projections $\pi_1$, $\pi_2$, and $\pi_3$ onto $\col(U_1)
\ominus \col(\one_n)$, $\col(U_2) \ominus \col(U_1)$, and $\R^n \ominus
\col(U_2)$, the canonical mean-squares are given by
\[\MS_1=Y'\frac{\pi_1}{I-1}Y, \qquad \MS_2=Y'\frac{\pi_2}{IJ-I}Y,
\qquad \MS_3=Y'\frac{\pi_3}{n-IJ}Y.\]
The MANOVA estimators are defined as \cite{searleetal}
\[\hSigma_1=\frac{1}{JK}\MS_1-\frac{1}{JK}\MS_2, \qquad
\hSigma_2=\frac{1}{K}\MS_2-\frac{1}{K}\MS_3, \qquad \hSigma_3=\MS_3.\]
\end{example}

\begin{example}\label{example:crossed}
Consider a replicated crossed two-way design
\[\y_{i,j,k,l}=\bmu+\balpha_i+\bbeta_{i,j}+\bgamma_{i,k}+\bdelta_{i,j,k}
+\beps_{i,j,k,l},\]
where
\[\balpha_i \overset{iid}{\sim} \N(0,\Sigma_1),
\quad \bbeta_{i,j} \overset{iid}{\sim} \N(0,\Sigma_2),
\quad \bgamma_{i,k} \overset{iid}{\sim} \N(0,\Sigma_3),
\quad \bdelta_{i,j,k} \overset{iid}{\sim} \N(0,\Sigma_4),
\quad \bgamma_{i,j,k,l} \overset{iid}{\sim} \N(0,\Sigma_5).\]
This corresponds to the North Carolina II design of \cite{comstockrobinson},
where $I$ replicates of a cross-breeding experiment are performed, each
experiment breeding $J$ distinct fathers with $K$ distinct mothers and
measuring traits in $L$ offspring for each $(i,j,k)$. We assume $J,K,L
\geq 2$ are fixed constants.

This model may be written in the matrix form
\[Y=\one_n\bmu'+\sum_{r=1}^4 U_r\alpha_r+\alpha_5,\]
where the incidence matrices are
\[U_1=\Id_I \otimes \one_{JKL}, \qquad U_2=\Id_{IJ} \otimes \one_{KL},
\qquad U_3=\Id_I \otimes \one_J \otimes \Id_K \otimes \one_L,
\qquad U_4=\Id_{IJK} \otimes \one_L.\]
Defining orthogonal projections $\pi_1,\pi_2,\pi_3,\pi_4,\pi_5$ onto
$\oS_1=\col(U_1) \ominus \col(\one_n)$, $\oS_2=\col(U_2) \ominus \col(U_1)$,
$\oS_3=\col(U_3) \ominus \col(U_1)$,
$\oS_4=\col(U_4) \ominus (\col(U_1) \oplus \oS_2
\oplus \oS_3)$, and $\oS_5=\R^n \ominus \col(U_4)$, the canonical mean-squares
are
\[\MS_r=Y'\frac{\pi_r}{d_r}Y \qquad \text{ for } r=1,\ldots,5,\]
where $d_r=\dim(\oS_r)$. The forms of the classical MANOVA estimators may be
deduced from the general discussion below.
\end{example}

To encompass these examples, we consider a general balanced classification
design defined by the following properties, as discussed also in
\cite[Appendix A]{fanjohnstonebulk}:
\begin{enumerate}
\item For each $r$, let $c_r=n/m_r$. Then $U_r'U_r=c_r\Id_{m_r}$, and
$\Pi_r=c_r^{-1}U_rU_r'$ is an orthogonal projection onto a subspace
$S_r \subset \R^n$ of dimension $m_r$.
\item Define $S_0=\col(X)$. Then $S_0 \subset S_r \subset S_k=\R^n$
for each $r=1,\ldots,k-1$.
\item Partially
order the subspaces $S_r$ by inclusion: $s \preceq r$ if $S_s \subseteq S_r$.
Let $\oS_0=S_0$, and for $r=1,\ldots,k$ 
let $\oS_r$ denote the orthogonal complement in $S_r$ of all $S_s$
properly contained in $S_r$. Then for each $r$,
\begin{equation}\label{eq:orthosum}
S_r=\bigoplus_{s \preceq r} \oS_s.
\end{equation}
In particular, $\R^n=S_k=\oplus_{r=0}^k \oS_r$.
\end{enumerate}

\begin{figure}
\centering
\includegraphics{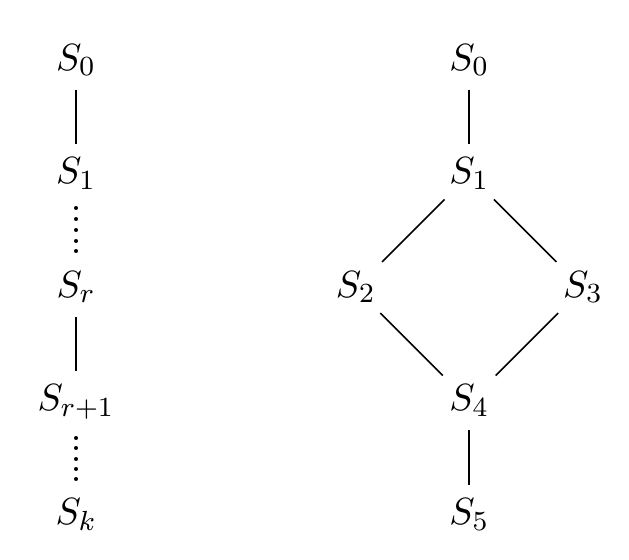}
\caption{Inclusion lattices for the subspaces $\{S_r\}$ determined by the nested
(left) and crossed (right) examples.}\label{fig:inclusionlattice}
\end{figure}

The subspace inclusion lattices for the nested designs of Examples
\ref{example:oneway} and \ref{example:twoway} and the crossed design of Example
\ref{example:crossed} are depicted in Figure \ref{fig:inclusionlattice}.

For each $r=0,\ldots,k$, let $d_r=\dim(\oS_r)$, let
$V_r \in \R^{n \times d_r}$ have orthonormal columns spanning $\oS_r$, and let
$\pi_r=V_rV_r'$
be the orthogonal projection onto $\oS_r$. (In particular, $d_0=\dim \col(X)$
is the dimensionality of fixed effects.) Then $\pi_0,\ldots,\pi_k$ are
mutually orthogonal projections summing to $\Id_n$. Note that the condition
(\ref{eq:orthosum}) implies
\[U_rU_r'=c_r\Pi_r=\sum_{s \preceq r} c_r\pi_s.\]
Then the likelihood of $Y$ in (\ref{eq:multivariatenormal}) may be written in
the form
\[f(Y) \propto \exp\left(-\frac{1}{2}\sum_{s=0}^k \Tr \left[
\left(\sum_{r \geq 1:\,r \succeq s} c_r\Sigma_r\right)^{-1}(Y-X\beta)'\pi_s
(Y-X\beta)\right]\right)\]
where $\pi_sX=0$ for $s \geq 1$. Hence the quantities
\[\pi_0Y,\quad \MS_1=Y'(\pi_1/d_1)Y, \quad \ldots, \quad \MS_k=Y'(\pi_k/d_k)Y\]
form sufficient statistics for this model.

In this setting, we restrict attention to matrices of the form
\begin{equation}\label{eq:hatSigmabalanced}
\hSigma=a_1\MS_1+\ldots+a_k\MS_k=Y'BY,
\qquad B=B(\ba)=a_1\frac{\pi_1}{d_1}+\ldots+a_k\frac{\pi_k}{d_k},
\end{equation}
and we suggest the choices $B_r=\pi_r/d_r$ for use in Algorithm
\ref{alg:spikes}. In particular, the classical MANOVA estimators are of this
form: From (\ref{eq:unbiasedestimation}), we have
\[\E[\MS_s]=\sum_{r=1}^k d_s^{-1}\Tr(U_r'\pi_sU_r)\Sigma_r
=\sum_{r \succeq s} c_r\Sigma_r, \qquad
\E[\hSigma]=\sum_{s=1}^k \sum_{r \succeq s} a_sc_r\Sigma_r.\]
The MANOVA estimate of $\Sigma_r$ is obtained by choosing
$\ba=(a_1,\ldots,a_k)$ so that $\E[\hSigma]=\Sigma_r$. Denoting
\begin{equation}\label{eq:MANOVAH}
H_{rs}=\1\{s \preceq r\}c_r, \qquad H=(H_{rs})_{r,s=1}^k \in \R^{k \times k},
\end{equation}
this is satisfied by letting $\ba$ be the $r^{\text{th}}$ column of $H^{-1}$.
(This corresponds to the procedure of M\"obius inversion over the subspace
inclusion lattice, discussed in greater detail in \cite{speed}.)

We record the following result for this class of balanced models.
In particular, (\ref{eq:MPbalanced}) below implies that $m_0(\lambda)$ may be
computed by solving a polynomial equation of degree $k+1$. For $\lambda \in
\R \setminus \supp(\mu_0)$, it may be shown that $m_0(\lambda)$ is the unique
root of this polynomial which satisfies
$z_0'(m_0(\lambda))>0$. From this, the quantities $t_r(\lambda)$ and
$w_{rs}(\lambda)$ are easily computed. The edges
of $\supp(\mu_0)$ are given by the values $z_0(m_*)$ where $m_* \in \R$
solves the equation $0=z_0'(m_*)$
(see \cite[Proposition 2.2]{fanjohnstoneedges}). By (\ref{eq:MPbalanced}), this
equation may be written as a polynomial of degree $2k$ in $m_*$.

\begin{proposition}\label{prop:balanced}
Consider a balanced classification design as
defined above. Suppose Assumptions \ref{assump:main}(a,b,d) hold for this
design, and in addition, $d_r>cn$ for each $r=1,\ldots,k$ and a constant $c>0$.
Then:
\begin{enumerate}[(a)]
\item Assumption \ref{assump:estimation} holds for the matrices $B_r=\pi_r/d_r$.
\item Let $\hSigma$ and $B$ be defined by (\ref{eq:hatSigmabalanced}),
where $\ba \in \R^k$ satisfies $\|\ba\|<C$ for a constant $C>0$. Then
the Marcenko-Pastur equation (\ref{eq:MP}) corresponding to $\hSigma$
takes the form
\begin{equation}\label{eq:MPbalanced}
z=-\frac{1}{m_0(z)}+\sum_{s=1}^k C_sb_s(z), \qquad b_s(z)=\frac{a_s}
{1+(N/d_s)a_sC_sm_0(z)}, \qquad C_s=\sum_{r \succeq s} c_r\sigma_r^2.
\end{equation}
The functions $t_r(z)$ and $w_{rs}(z)$ from (\ref{eq:TSigma}) and
(\ref{eq:wrs}) take the forms
\[t_r(z)=c_r\sum_{s \geq 1:\,s \preceq r} b_s(z), \qquad
w_{rs}(z)=c_rc_s\sum_{t \geq 1:\,t \preceq r,\,
t \preceq s} (N/d_t)b_t(z)^2.\]
\end{enumerate}
\end{proposition}
\begin{proof}
We rotate coordinates. Fix $r \in \{1,\ldots,k\}$ and write
$\{s:s \preceq r\}=\{s_0,\ldots,s_j\}$ where $s_0=0$.
We may write the singular value decomposition of $U_r$ as
\[U_r=\sqrt{c_r}\sum_{s \preceq r} V_sW_{r,s}',\]
where the columns of $V_s$ form an orthonormal basis for $\oS_s$ and
where $W_r=[W_{r,s_0} \mid \ldots \mid W_{r,s_j}]$ is orthogonal
$m_r \times m_r$. Denote $\vn=n-d_0$, $\vm_r=m_r-d_0$,
$V=[V_1 \mid \ldots \mid V_k] \in \R^{n \times \vn}$, and
\[\valpha_r=\begin{pmatrix} W_{r,s_1}'\alpha_r \\ \vdots \\
W_{r,s_j}'\alpha_r \end{pmatrix} \in \R^{\vm_r \times p}, \qquad
\vU_r=\sqrt{c_r}V'\begin{pmatrix} V_{s_1} & \cdots &
V_{s_j} \end{pmatrix} \in \R^{\vn \times \vm_r}.\]
By rotational invariance, $\valpha_r$ still has independent rows with
distribution $\N(0,\Sigma_r)$. Also,
$\vU_r$ has a simple form---each $V'V_{s_i} \in \R^{\vn \times
d_{s_i}}$ has a single block equal to $\Id_{d_{s_i}}$ and remaining blocks 0.
For any matrices $\hSigma$ and $B$ given by (\ref{eq:hatSigmabalanced}),
defining
\[\vY=V'Y \in \R^{\vn \times p}, \qquad \vB=V'BV=
\begin{pmatrix} (a_1/d_1)\Id_{d_1} & & \\
& \ddots & \\ & & (a_k/d_k)\Id_{d_k} \end{pmatrix} \in \R^{\vn \times \vn}\]
and applying $V'$ to (\ref{eq:mixedmodel}), we obtain the rotated model
\begin{equation}\label{eq:mixedmodelrotated}
\vY=\sum_{r=1}^k \vU_r\valpha_r, \qquad \hSigma=\vY'\vB\vY.
\end{equation}

Let $F$ be the matrix (\ref{eq:F}) in the model (\ref{eq:mixedmodel}).
Let $\vM=\vm_1+\ldots+\vm_k$, and 
denote by $\vF \in \R^{\vM \times \vM}$ this matrix in the rotated model
(\ref{eq:mixedmodelrotated}), with $(r,s)$ block equal to
$N\sigma_r\sigma_s \vU_r'\vB\vU_s$.
Let $Q=\diag(W_1,\ldots,W_k)$, where $W_r$ is the matrix of right singular
vectors of $U_r$ as above. Then observe that $\vF$ is the matrix
$Q'FQ$ with $d_0$ rows and $d_0$ columns of 0's removed from each block. Thus,
the law $\mu_0$ and the functions $m_0(z)$, $s_r(\ba)$,
$t_r(z)$, and $w_{rs}(z)$ do
not change upon replacing $F$ by $\vF$ in their definitions.

Let us further decompose $\vm_r=\sum_{s \geq 1:\,s \preceq r} d_s$, and consider
$\vF$ in the expanded block decomposition corresponding to
\[\vM=\sum_{r=1}^k \sum_{s \geq 1:\,s\preceq r} d_s.\]
Index a row or column of this decomposition by the pair $(r,s)$ where
$s \preceq r$. Then from the forms of $\vU_r$ and $\vB$, we have
\[\vF_{(r,s),(r',s')}=\1\{s=s'\}N\sqrt{c_rc_{r'}}\sigma_r\sigma_{r'}
\frac{a_s}{d_s}\Id_{d_s}.\]
For each $s \in \{1,\ldots,k\}$, let $E_s$ be the submatrix formed by the
blocks $((r,s),(r',s))$ where $r \succeq s$ and $r' \succeq s$.
Note that $\vF$ is (upon permuting rows and columns) block-diagonal with blocks
$E_1,\ldots,E_k$. We may write
$E_s=N(a_s/d_s)R_s'R_s$ where $R_s=(\sqrt{c_r}\sigma_r \Id_{d_s}\;:\;r
\succeq s)$. Then $E_s$ has rank $d_s$, with $d_s$ identical non-zero
eigenvalues equal to $N(a_s/d_s)C_s$ where $C_s$ is defined in
(\ref{eq:MPbalanced}). As the eigenvalues of $\vF$ are the union of
those of $E_1,\ldots,E_k$, writing (\ref{eq:MP}) in spectral form
establishes (\ref{eq:MPbalanced}).

To verify Assumption \ref{assump:estimation},
note that $R_sR_s'=C_s\Id_{d_s}$, so the Woodbury identity yields
\[E_s(\Id+E_s)^{-1}=\frac{Na_s}{d_s}R_s'R_s\left(\Id
-\frac{Na_s}{d_s(1+N(a_s/d_s)C_s)}R_s'R_s\right)
=\frac{N(a_s/d_s)}{1+N(a_s/d_s)C_s}R_s'R_s.\]
Then for all $s \preceq r$ and $s' \preceq r'$,
\begin{equation}\label{eq:Fbalanced}
\Big(F(\Id+F)^{-1}\Big)_{(r,s),(r',s')}=\1\{s=s'\}
\sqrt{c_rc_r'}\sigma_r\sigma_{r'}\frac{N(a_s/d_s)}{1+N(a_s/d_s)C_s}\Id_{d_s}.
\end{equation}
The $r^{\text{th}}$ diagonal block trace in the collapsed decomposition
$\vM=\vm_1+\ldots+\vm_k$ is the sum of the trace of the above over
$s \preceq r$, $s=s'$, and $r=r'$. Thus
\[s_r(\ba)=c_r\sum_{s \preceq r} \frac{a_s}{1+(N/d_s)a_sC_s}
=Hf(\ba),\]
where $H$ is defined in (\ref{eq:MANOVAH}) and
\[f(\ba)=\left(\frac{a_1}{1+(N/d_1)a_1C_1},
\ldots, \frac{a_k}{1+(N/d_k)a_kC_k}\right).\]
As $C_1,\ldots,C_k$ and $N/d_1,\ldots,N/d_k$ are bounded above by a constant,
we have
\[\|f(\ba_1)-f(\ba_2)\| \geq
\frac{c\,\|\ba_1-\ba_2\|}{(1+\|\ba_1\|)(1+\|\ba_2\|)}\]
for a constant $c>0$. Under a suitable permutation of $1,\ldots,k$,
the matrix $H$ is lower-triangular, with all entries bounded above, and with
all diagonal entries $c_r$ bounded away from 0. Thus the
least singular value of $H$ is bounded away from 0, so
Assumption \ref{assump:estimation} holds.

Finally, substituting $m_0a_s$ for $a_s$ in (\ref{eq:Fbalanced}), we also have
\[\Big(F(\Id+m_0F)^{-1}\Big)_{(r,s),(r',s')}=\1\{s=s'\}
\sqrt{c_rc_r'}\sigma_r\sigma_{r'}\frac{N(a_s/d_s)}{1+N(a_s/d_s)C_sm_0}
\Id_{d_s}.\]
Taking block traces and Hilbert-Schmidt norms in the collapsed decomposition
$\vM=\vm_1+\ldots+\vm_k$ yields the expressions for $t_r$ and $w_{rs}$.
\end{proof}

\section{Proofs}\label{sec:proofs}
This section contains the proofs of our main results.
Section \ref{subsec:LLN} proves Theorems \ref{thm:spikemapping} and
\ref{thm:eigenvectors}, which describe the first-order behavior of outlier
eigenvalues and eigenvectors. Section \ref{subsec:CLT} proves Theorem
\ref{thm:CLT} on Gaussian fluctuations. Section \ref{subsec:estimation}
proves Theorem \ref{thm:estimation} providing
theoretical guarantees for Algorithm \ref{alg:spikes}.

The proofs of Theorems \ref{thm:spikemapping}--\ref{thm:CLT} will apply a
matrix perturbation approach developed in \cite{paul}.
Without loss of generality, we may rotate coordinates in $\R^p$ so that $\cS$
corresponds to the first $L$ coordinates. Hence for every $r=1,\ldots,k$,
\[V_r=\begin{pmatrix} \oV_r \\ 0 \end{pmatrix}\]
where $\oV_r \in \R^{L \times l_r}$. Recalling $N=p-L$ and assuming momentarily
that $\sigma_r^2>0$, we may write
\begin{equation}\label{eq:Gammar}
\Sigma_r=\sigma_r^2 \begin{pmatrix} \Gamma_r & 0 \\ 0 & \Id_N
\end{pmatrix},\qquad \Gamma_r=\Id_L+\sigma_r^{-2} \oV_r \Theta_r
\oV_r'.
\end{equation}
Letting $\oX_r \in \R^{m_r \times L}$ and $X_r \in \R^{m_r \times N}$
be independent with i.i.d.\ $\N(0,1/N)$ entries,
and setting $\Xi_r=\oX_r \Gamma_r^{1/2}$, we may represent $\alpha_r$ as
\[\alpha_r=\sqrt{N}\begin{pmatrix} \oX_r & X_r \end{pmatrix}\Sigma_r^{1/2}
=\sqrt{N} \sigma_r \begin{pmatrix} \Xi_r & X_r \end{pmatrix}.\]
Recalling $F_{rs}=N\sigma_r\sigma_sU_r'BU_s$ from (\ref{eq:F}), we then have
\begin{equation}\label{eq:hSigmablock}
\hSigma=Y'BY=\sum_{r,s=1}^k \alpha_r'U_r'BU_s\alpha_s
=\sum_{r,s=1}^k \begin{pmatrix} \Xi_r' \\ X_r' \end{pmatrix}
F_{rs} \begin{pmatrix} \Xi_s & X_s \end{pmatrix}
=\begin{pmatrix} S_{11} & S_{12} \\ S_{21} & S_{22} \end{pmatrix},
\end{equation}
where
\[\begin{pmatrix} S_{11} & S_{12} \\ S_{21} & S_{22} \end{pmatrix}
=\begin{pmatrix} \Xi' F \Xi & \Xi'FX \\ X'F\Xi & X'FX \end{pmatrix}, \qquad
\begin{pmatrix} \Xi & X \end{pmatrix}=\begin{pmatrix} \Xi_1 & X_1 \\
\vdots & \vdots \\ \Xi_k & X_k \end{pmatrix}.\]
(Throughout this section, $X$ refers to this matrix and not the design
matrix of (\ref{eq:mixedmodel}).)
Note that $\sigma_r^2\Gamma_r$, $\sigma_r\Xi_r$, and $F_{rs}/(\sigma_r\sigma_s)$
are well-defined by continuity even when $\sigma_r^2=0$ and/or $\sigma_s^2=0$.
The above definitions are understood in this sense if $\sigma_r^2=0$ for any
$r$.

For any $z \notin \spec(X'FX)$, the Schur complement of the
lower-right block of $\hSigma-z\Id$ is
\begin{align}
\hK(z)&=(S_{11}-z\Id)-S_{12}(S_{22}-z\Id)^{-1}S_{21}
=-\Xi'G_M(z)\Xi-z\Id_L\label{eq:hK}
\end{align}
where
\begin{equation}\label{eq:GM}
G_M(z)=FXG_N(z)X'F-F, \qquad G_N(z)=(X'FX-z\Id_N)^{-1}.
\end{equation}
If $\hl$ is an eigenvalue of $\hSigma$ separated from $\supp(\mu_0)$, then
we expect from Theorem \ref{thm:sticktobulk} that $\hl \notin \spec(X'FX)$, so
we should have $0=\det \hK(\hl)$. Defining the complex spectral domain
\[U_\delta=\{z \in \C:\dist(z,\supp(\mu_0)) \geq \delta\},\]
we will show that on $U_\delta$, the matrix $\hK(z)$ is close to
the deterministic approximation
\begin{equation}\label{eq:K}
K(z)=\sum_{r=1}^k t_r(z)(\sigma_r^2\Gamma_r)-z\Id_L.
\end{equation}
Recalling (\ref{eq:Gammar}) and comparing (\ref{eq:K}) with (\ref{eq:TSigma}),
we observe that $K(z)$ is the upper $L \times L$ submatrix of $-T(z)$.
This will yield Theorems \ref{thm:spikemapping} and \ref{thm:eigenvectors}.
Studying further the fluctuations of $\hK(z)$ about $K(z)$, we will establish
Theorem \ref{thm:CLT}.

We show in Appendix \ref{sec:resolventapprox} that $G_M(z)$ and $G_N(z)$ are
blocks of a linearized resolvent matrix for $X'FX$.
Our proof establishes deterministic approximations for linear and quadratic
functions of the entries of $G_M(z)$, which we may state as follows:
Recall (\ref{eq:Fidentity}), and define a deterministic approximation of
$G_M$ as
\[\Pi_M=-F(\Id+m_0F)^{-1}=m_0F(\Id+m_0F)^{-1}F-F.\]
Define
\begin{equation}\label{eq:Delta}
\Delta(z)=XG_N(z)X'-m_0(z)(\Id+m_0(z)F)^{-1}.
\end{equation}
Then, omitting the spectral argument $z$ for brevity, we have
\begin{equation}\label{eq:GMPiM}
G_M=\Pi_M+F\Delta F, \qquad t_r=(N\sigma_r^2)^{-1}\Tr_r(-G_M+F\Delta F).
\end{equation}
We prove the following lemmas in Appendix \ref{sec:resolventapprox}.
\begin{lemma}\label{lemma:matrixlocallaw}
Fix $\delta,\eps,D>0$.
For any $z \in U_\delta$ and any deterministic matrix $V \in \C^{M \times M}$,
\[\P\Big[|\Tr \Delta V|>n^{-1/2+\eps}\|V\|_\HS\Big]<n^{-D}.\]
\end{lemma}
\begin{lemma}\label{lemma:secondorderapprox}
Fix $\delta,\eps,D>0$. For any $z \in U_\delta$ and any deterministic
matrices $V,W \in \C^{M \times M}$,
\[\P\Big[\big|\Tr \Delta V\Delta W-N^{-1}(\partial_z m_0)
\Tr\big[V(\Id+m_0 F)^{-2}\big]
\Tr\big[W(\Id+m_0F)^{-2}\big]\big|>n^{1/2+\eps}\|V\|\,\|W\|\Big]<n^{-D}.\]
\end{lemma}
We will use Lemma \ref{lemma:matrixlocallaw} to approximate linear functions in
$G_M(z)$, and then use Lemma \ref{lemma:secondorderapprox} to approximate
quadratic functions in $G_M(z)$.
Note that if $V=\w\v'$ is of rank one, then Lemma \ref{lemma:matrixlocallaw}
is an anisotropic local law of the form established in \cite{knowlesyin} for
spectral arguments $z$ separated from $\supp(\mu_0)$. For
general $V$, the statement above is stronger than that obtained by expressing
$V$ as a sum of rank-one matrices and applying the triangle inequality to the
Hilbert-Schmidt norm. We will
require this stronger form for the proof of Theorem \ref{thm:CLT}.

We record here also the following basic results regarding $\supp(\mu_0)$ and
the Stieltjes transform $m_0(z)$ for spectral arguments $z$ separated
from this support, proven in \cite[Propositions A.3 and B.1]{fanjohnstoneedges}.
\begin{proposition}\label{prop:boundedsupportrestate}
Suppose Assumption \ref{assump:main} holds, and let $\mu_0$ be the law
defined by Theorem \ref{thm:MP}.
For a constant $C>0$, $\supp(\mu_0) \subset [-C,C]$.
\end{proposition}
\begin{proposition}\label{prop:m0regularoutsiderestate}
Suppose Assumption \ref{assump:main} holds, and let $m_0(z)$ be the Stieltjes
transform of the law $\mu_0$.
Fix any constant $\delta>0$. Then for some constant $c>0$, all $z \in U_\delta$,
and each eigenvalue $t_\a$ of $F$,
\[|1+t_\a m_0(z)|>c.\]
\end{proposition}

\noindent {\bf Notation:} Throughout, $\delta>0$ is a fixed constant. $C,c>0$
denote $\delta$-dependent
constants that may change from instance to instance. For random
(or deterministic) scalars $\xi$ and $\zeta$, we write
\[\xi \prec \zeta \text{ and } \xi=\O(\zeta)\]
if, for any constants $\eps,D>0$, we have
\[\P[|\xi|>n^\eps |\zeta|]<n^{-D}\]
for all $n \geq n_0$, where $n_0$ may depend only on $\delta,\eps,D$ and the
constants of Assumptions \ref{assump:main}.

\subsection{First-order behavior}\label{subsec:LLN}

We prove Theorems \ref{thm:spikemapping} and \ref{thm:eigenvectors}.
Let us first establish the approximations
\begin{align}
\sup_{z \in U_\delta} \|\hK(z)-K(z)\| &\prec n^{-1/2},\label{eq:hKbound}\\
\sup_{z \in U_\delta} \|\partial_z \hK(z)-\partial_z K(z)\| &\prec n^{-1/2}.
\label{eq:dhKbound}
\end{align}

Denote by $(G_M)_{rs}$ and $(\Pi_M)_{rs}$ the $(r,s)$ blocks of $G_M$ and
$\Pi_M$. We record a basic lemma which bounds $G_M$, $\Pi_M$, and the
derivatives of $K$ and $\hK$. Quantities such as $F_{rs}/(\sigma_r\sigma_s)$
are defined by continuity at $\sigma_r^2=0$ and/or $\sigma_s^2=0$.

\begin{lemma}\label{lemma:Kbound}
There is a constant $C>0$ such that
\begin{enumerate}[(a)]
\item For all $z \in U_\delta$ and $r,s=1,\ldots,k$,
\[\|F_{rs}/(\sigma_r\sigma_s)\|/<C,
\qquad \|(\Pi_M)_{rs}/(\sigma_r\sigma_s)\|<C,
\qquad \|K(z)+z\Id\|<C,\qquad \|\partial_z K(z)\|<C.\]
\item For any $D>0$ and all $n \geq n_0(\delta,D)$,
with probability at least $1-n^{-D}$, for all
$z \in U_\delta$ and $r,s=1,\ldots,k$
\[\|(G_M)_{rs}/(\sigma_r\sigma_s)\|<C, \qquad \|\hK(z)+z\Id\|<C, \qquad
\|\partial_z \hK(z)\|<C, \qquad \|\partial_z^2 \hK(z)\|<C.\]
\end{enumerate}
\end{lemma}
\begin{proof}
For (a), $\|F_{rs}/(\sigma_r\sigma_s)\|<C$ by Assumption \ref{assump:main}.
From Proposition \ref{prop:m0regularoutsiderestate}, we have $\|(\Id+m_0F)^{-1}\|<C$.
Furthermore, $|m_0(z)| \leq 1/\delta$ for $z \in
U_\delta$ by (\ref{eq:m0stieltjes}). Then, denoting by $P_r$ the projection onto
block $r$ and applying
\[(\Pi_M)_{rs}=m_0P_rF(\Id+m_0F)^{-1}FP_s-F_{rs},\]
we obtain
\[\|(\Pi_M)_{rs}/(\sigma_r\sigma_s)\|<C\|P_rF/\sigma_r\|\|FP_s/\sigma_s\|
+\|F_{rs}/(\sigma_r\sigma_s)\|<C.\]
This implies also $|t_r(z)|<C$, which together with
$\|\sigma_r^2\Gamma_r\|<C$ yields the bound on $K$. The bound for
$\partial_z K$ follows similarly.

For (b), applying Theorem \ref{thm:sticktobulk} and a standard spectral norm
bound for Gaussian matrices, on an event of probability $1-n^{-D}$
we have $\spec(X'FX) \subset \supp(\mu_0)_{\delta/2}$,
$\|X_r\|<C$, and $\|\oX_r\|<C$ for all $r=1,\ldots,k$.
From the spectral decomposition of $G_N$, on this event, we have
$\|G_N\|<C$, $\|\partial_z G_N\|<C$, and $\|\partial_z^2 G_N\|<C$
for all $z \in U_\delta$. Then
\[\|(G_M)_{rs}/(\sigma_r\sigma_s)\|,\,
\|\partial_z (G_M)_{rs}/(\sigma_r\sigma_s)\|,\,
\|\partial_z^2 (G_M)_{rs}/(\sigma_r\sigma_s)\|<C.\]
As $\hK=-\sum_{r,s} \Xi_r'G_M\Xi_s-z\Id$ and
$\|\sigma_r\Xi_r\|<C\|\sigma_r^2\Gamma_r\|^{1/2}<C$, this yields the bounds
on $\hK$ and its derivatives.
\end{proof}

We recall also the following bound for Gaussian quadratic forms.
\begin{lemma}[Gaussian quadratic forms]\label{lemma:hansonwright}
Let $\x$ and $\y$ be independent vectors of any dimensions,
with i.i.d.\ $\N(0,1/N)$ entries. Then for any complex deterministic matrices
$A$ and $B$ of the corresponding sizes,
\[\x'A\x-N^{-1}\Tr A \prec N^{-1}\|A\|_\HS, \qquad
\x'B\y \prec N^{-1}\|B\|_\HS.\]
\end{lemma}
\begin{proof}
The first statement follows from the Hanson-Wright inequality,
see e.g.\ \cite{rudelsonvershynin}. The second follows from the first applied to
$(\x,\y)$, with $A$ a $2 \times 2$ block matrix having blocks 0, $B$, $B'$, 0.
\end{proof}

Applying (\ref{eq:GMPiM}), we may write $\hK(z)-K(z)=E_1(z)+E_2(z)$ where
\begin{align}
E_1(z)&=-\Xi'G_M\Xi+\sum_{r=1}^k \Big(N^{-1} \Tr_r G_M\Big)
\Gamma_r,\label{eq:E1}\\
E_2(z)&=-\sum_{r=1}^k \Big(N^{-1}\Tr_r F\Delta F\Big)\Gamma_r.
\label{eq:E2}
\end{align}
Writing $P_r$ for the projection onto block $r$,
Lemma \ref{lemma:matrixlocallaw} yields
\begin{equation}\label{eq:E2bound}
(N\sigma_r^2)^{-1}\Tr_r F\Delta F=(N\sigma_r^2)^{-1}
\Tr \Delta FP_rF \prec n^{-3/2}\|FP_rF/\sigma_r^2\|_\HS
\prec n^{-1}.
\end{equation}
Hence $\|E_2(z)\| \prec n^{-1}$. For $E_1$, we write
\[\Xi'G_M\Xi=\sum_{r,s=1}^k \Xi_r'(G_M)_{rs}\Xi_s
=\sum_{r,s=1}^k (\sigma_r^2\Gamma_r)^{1/2}\oX_r'
\frac{(G_M)_{rs}}{\sigma_r\sigma_s} \oX_s (\sigma_s^2\Gamma_s)^{1/2}.\]
Recall that the matrices $\oX_r$ are independent of each other and of $G_M$.
Applying Lemma \ref{lemma:hansonwright} conditional on $G_M$
and taking a union bound over the columns of $\oX_r$ and $\oX_s$,
for all $r,s$,
\[\left\|\oX_r' \frac{(G_M)_{rs}}{\sigma_r\sigma_s}
\oX_s-\1\{r=s\}((N\sigma_r^2)^{-1}\Tr_r G_M)\Id_L\right\|_\infty \prec
n^{-1}\|(G_M)_{rs}/(\sigma_r\sigma_s)\|_\HS,\]
where $\|A\|_\infty=\max_{i,j} |A_{ij}|$. As $L$ is at most a constant,
this norm is equivalent to the operator norm. By Lemma \ref{lemma:Kbound},
$\|(G_M)_{rs}/(\sigma_r\sigma_s)\|_\HS \prec n^{1/2}$, so
$\|E_1(z)\|\prec n^{-1/2}$. Then
\begin{equation}\label{eq:hKboundz}
\|\hK(z)-K(z)\| \prec n^{-1/2}.
\end{equation}

Lipschitz continuity allows us to take a union bound over $z \in U_\delta$:
On the event where the conclusions of Lemma \ref{lemma:Kbound} hold,
for any $z,z' \in U_\delta$,
\[\|\hK(z)-\hK(z')\|<C|z-z'|, \qquad \|K(z)-K(z')\|<C|z-z'|.\]
Then, taking a union bound of (\ref{eq:hKboundz}) over a grid of values in
$U_\delta \cap \{|z| \leq n^{1/2}\}$ with spacing $n^{-1/2}$, we obtain
\begin{equation}\label{eq:hKboundsmallz}
\sup_{z \in U_\delta:\,|z| \leq n^{1/2}}
\|\hK(z)-K(z)\| \prec n^{-1/2}.
\end{equation}
For $|z|>n^{1/2}$, we apply a direct argument: By Proposition
\ref{prop:boundedsupportrestate} and (\ref{eq:m0stieltjes}), we have
$|m_0(z)|<Cn^{-1/2}$. Then $|t_r(z)-(N\sigma_r^2)^{-1}\Tr_r F|<Cn^{-1/2}$.
Furthermore, on the high-probability event where
$\|X'FX\|<C$ and $\|\oX_r\|<C$ for each $r=1,\ldots,k$,
we have $\|G_N\|<Cn^{-1/2}$,
$\|[(G_M)_{rs}-F_{rs}]/(\sigma_r\sigma_s)\|<Cn^{-1/2}$,
and $\|\sigma_r\Xi_r\|<C$. Then, on this event,
\[\sup_{|z|>n^{1/2}} \|\hK(z)-K(z)\|
\leq \left\|\Xi'F\Xi-\sum_{r=1}^k (N^{-1}\Tr_r
F)\Gamma_r\right\|+Cn^{-1/2}.\]
Applying Lemma \ref{lemma:hansonwright} again yields $\sup_{|z|>n^{1/2}}
\|\hK(z)-K(z)\| \prec n^{-1/2}$. Combining with (\ref{eq:hKboundsmallz})
yields (\ref{eq:hKbound}). Note that $D_{ij}(z) \equiv \hK(z)_{ij}-K(z)_{ij}$ 
is analytic over $U_{\delta/2}$. Letting $\gamma$ be the circle around $z$ with 
radius $\delta/2$, the Cauchy integral formula implies
\[|\partial_z D_{ij}(z)| \leq \frac{1}{2\pi} \int_\gamma
\frac{|D_{ij}(w)|}{|z-w|^2}dw \leq \frac{4}{\delta}\max_{w \in \gamma}
|D_{ij}(w)|.\]
Then applying (\ref{eq:hKbound}) with $\delta/2$ in place of $\delta$, we obtain
also the derivative bound (\ref{eq:dhKbound}).

\begin{proof}[Proof of Theorem \ref{thm:spikemapping}]
Let $\cE$ be the event where
\[\spec(X'FX) \subset \supp(\mu_0)_{\delta/2},\]
\[\sup_{z \in U_{\delta/2}} \|\hK(z)-K(z)\|<n^{-1/2+\eps/2},
\qquad \sup_{z \in U_{\delta/2}} \|\partial_z \hK(z)-\partial_z K(z)\|
<n^{-1/2+\eps/2},\]
which holds with probability $1-n^{-D}$ for all $n \geq n_0(\delta,\eps,D)$.
On $\mathcal{E}$, by the Schur complement identity
\[\det(\hSigma-\lambda\Id)=\det(X'FX-\lambda \Id)\det(\hK(\lambda)),\]
the eigenvalues of $\hSigma$ outside $\supp(\mu_0)_{\delta/2}$
are the roots $\hl \in U_{\delta/2} \cap \R$ of $\det(\hK(\hl))$,
counting multiplicity. As $T(z)$ is block diagonal with upper $L \times L$ block
equal to $-K(z)$ and lower $N \times N$ block equal to $-m_0(z)^{-1}\Id$,
the elements of $\Lambda_0$ are the roots
$\lambda \in \R \setminus \supp(\mu_0)$ of $\det(K(\lambda))$, counting
multiplicity.

Let $\hmu_1(\lambda) \leq \ldots \leq \hmu_L(\lambda)$ be
the eigenvalues of $\hK(\lambda)$, and let $\mu_1(\lambda) \leq
\ldots \leq \mu_L(\lambda)$ be those of $K(\lambda)$. Proposition
\ref{prop:Tproperties} implies that $\partial_\lambda K(\lambda)$ has maximum
eigenvalue at most -1, so for any interval
$I$ of $\R \setminus \supp(\mu_0)$, any $\lambda,\lambda' \in I$ with
$\lambda<\lambda'$, and any $\ell \in \{1,\ldots,L\}$,
\[\mu_\ell(\lambda)-\mu_\ell(\lambda') \geq \lambda'-\lambda.\]
On $\cE$, for $\lambda \in I \cap
U_{\delta/2}$, we may bound the largest eigenvalue of
$\partial_\lambda \hK(\lambda)$ by $-1/2$. Then similarly
\[\hmu_\ell(\lambda)-\hmu_\ell(\lambda') \geq (\lambda'-\lambda)/2.\]

For each $(\lambda,\ell)$ with
$\lambda \in I \cap U_\delta$ and $\mu_\ell(\lambda)=0$, we have
\[\mu_\ell(\lambda-n^{-1/2+\eps}) \geq n^{-1/2+\eps},\qquad
\mu_\ell(\lambda+n^{-1/2+\eps}) \leq -n^{-1/2+\eps},\]
and hence on $\cE$
\[\hmu_\ell(\lambda-n^{-1/2+\eps})>0,\qquad
\hmu_\ell(\lambda+n^{-1/2+\eps})<0.\]
So there is some $\hl$ where
$\hmu_\ell(\hl)=0$ and $|\hl-\lambda|<n^{-1/2+\eps}$.
Conversely, for each $(\hl,\ell)$ with $\hl \in I \cap U_\delta$ and
$\hmu_\ell(\hl)=0$, there is some $\lambda$ with
$\mu_\ell(\lambda)=0$ and $|\lambda-\hl|<n^{-1/2+\eps}$. Taking
$\Lambda_\delta$ and $\hLambda_\delta$ to be the roots of $\det(K(\lambda))$
and $\det(\hK(\hl))$ corresponding to these
pairs $(\lambda,\ell)$ and $(\hl,\ell)$ for each interval $I$ of $\R \setminus
\supp(\mu_0)$, we obtain Theorem \ref{thm:spikemapping}.
\end{proof}

\begin{proof}[Proof of Theorem \ref{thm:eigenvectors}]
For the given $\lambda$ and $\hl$, Theorem \ref{thm:spikemapping} implies
$\lambda-\hl \prec n^{-1/2}$. Let us write
\[\hK(\hl)-K(\lambda)=\big(\hK(\hl)-\hK(\lambda)\big)
+\big(\hK(\lambda)-K(\lambda)\big).\]
The first term on the right has norm $\O(n^{-1/2})$, by the
bound on $\partial_\lambda \hK(\lambda)$ from Lemma \ref{lemma:Kbound}.
The second term also has norm $\O(n^{-1/2})$, by (\ref{eq:hKbound}).
Hence
\begin{equation}\label{eq:hKhlbound}
\|\hK(\hl)-K(\lambda)\| \prec n^{-1/2}.
\end{equation}
Similarly,
\begin{equation}\label{eq:dhKhlbound}
\|\partial_\lambda \hK(\hl)-\partial_\lambda K(\lambda)\| \prec n^{-1/2}.
\end{equation}

For the given $\hv$, let us write $\hv=(\hv_1,\hv_2)$ where $\hv_1$ consists of
the first $L$ coordinates. Then, in the block decomposition of $\hSigma$ from
(\ref{eq:hSigmablock}), the equation $\hSigma\hv=\hl\hv$ yields
\[S_{11}\hv_1+S_{12}\hv_2=\hl\hv_1,\qquad S_{21}\hv_1+S_{22}\hv_2=\hl\hv_2.\]
The second equation yields $\hv_2=-(S_{22}-\hl\Id)^{-1}S_{21}\hv_1$.
Substituting this into the first yields $\hK(\hl)\hv_1=0$,
while substituting it into $1=\|\hv\|^2=\|\hv_1\|^2+\|\hv_2\|^2$ yields
\[1=\hv_1'(\Id+S_{12}(S_{22}-\hl\Id)^{-2}S_{21})\hv_1=-\hv_1'(\partial_\lambda
\hK(\hl))\hv_1.\]
Applying (\ref{eq:dhKhlbound}), we obtain
\begin{equation}\label{eq:hv1form}
-\hv_1'(\partial_\lambda K(\lambda))\hv_1=1+\O(n^{-1/2}).
\end{equation}
In particular, $\|\hv_1\| \geq c$ for a
constant $c>0$. Hence $\hv_1/\|\hv_1\|$ is a well-defined unit
vector in $\ker \hK(\hl)$.

For the given $\v$, let us also write $\v=(\v_1,\v_2)$. As $\v \in \cS$ by
Proposition \ref{prop:Tproperties}, we have $\v_2=0$, $\|\v_1\|=1$, and
$\v_1 \in \ker K(\lambda)$. We apply the Davis-Kahan theorem to bound
$\|\hv_1/\|\hv_1\|-\v_1\|$:
Let $\mu_1(\lambda) \leq \ldots \leq \mu_L(\lambda)$ be
the eigenvalues of $K(\lambda)$, with $\mu_\ell(\lambda)=0$.
By Proposition \ref{prop:Tproperties}, $\partial_\lambda K(\lambda)$ has
maximum eigenvalue at most -1. Thus, if $|\mu_{\ell'}(\lambda)|<\delta$ for
another $\ell' \neq \ell$, then
$\mu_{\ell'}(\lambda-\delta)>0$ and $\mu_{\ell'}(\lambda+\delta)<0$, so
$\mu_{\ell'}(\lambda')=0$
for some $\lambda' \in (\lambda-\delta,\lambda+\delta)$.
This contradicts the given condition that $\lambda$ is separated from other
elements of $\Lambda_0$ by $\delta$. Hence
$|\mu_{\ell'}(\lambda)| \geq \delta$ for all $\ell' \neq \ell$, so
the Davis-Kahan Theorem and (\ref{eq:hKhlbound}) imply
\begin{equation}\label{eq:hv1bound}
\big\|\hv_1-\|\hv_1\|\v_1\big\| \prec n^{-1/2}
\end{equation}
for an appropriate choice of sign of $\v_1$.
Substituting into (\ref{eq:hv1form}),
$-\|\hv_1\|^2\v_1'\partial_\lambda K(\lambda) \v_1=1+\O(n^{-1/2})$.
As $\v_1'\partial_\lambda K(\lambda) \v_1 \leq -1$, this yields
\[\|\hv_1\|=(-\v_1'\partial_\lambda K(\lambda) \v_1)^{-1/2}+\O(n^{-1/2}).\]
Substituting back into (\ref{eq:hv1bound}),
\[\big\|\hv_1-(-\v_1'\partial_\lambda K(\lambda)\v_1)^{-1/2}\v_1\big\|
=\big\|P_{\cS}\hv-(\v'\partial_\lambda T(\lambda)\v)^{-1/2}\v\big\| \prec
n^{-1/2},\]
where the equality uses $\v=(\v_1,0)$, $P_{\cS}\hv=(\hv_1,0)$, and that
$K$ is the upper-left block of $-T$. This proves (a).

For (b), note simply that for any $O \in \R^{N \times N}$, the rotation
$X \mapsto XO$
induces the mapping $(\hv_1,\hv_2) \mapsto (\hv_1,O'\hv_2)$. As $X$ and
$\hSigma$ are invariant in law under such a rotation, $\hv_2$ must be
rotationally invariant in law conditional on $\hv_1$.
\end{proof}

\subsection{Fluctuations of outlier eigenvalues}\label{subsec:CLT}
Next, we prove Theorem \ref{thm:CLT}. We establish asymptotic normality using
the following elementary lemma.
\begin{lemma}\label{lemma:lyapunov}
Suppose $\z \in \R^n$ has law $\N(0,VV')$ where $V \in \R^{n \times m}$ for any
dimension $m$. Let $A \in \R^{n \times n}$. If
$\|V'AV\|/\|V'AV\|_\HS \to 0$ as $n \to \infty$, then
\[\|V'AV\|_\HS^{-1}(\z'A\z-\E[\z'A\z]) \to \N(0,2).\]
\end{lemma}
\begin{proof}
Denote the spectral decomposition of $V'AV$ as $O'DO$ where
$D=\diag(d_1,\ldots,d_m)$, and let
$\xi \in \R^m$ have i.i.d.\ $\N(0,1)$ entries. Then $\z'A\z-\E[\z' A\z]$
is equal in law to
\[\xi'D\xi-\E[\xi'D\xi]=\sum_{i=1}^m d_i(\xi_i^2-1).\]
As
\[\sum_{i=1}^m \E\Big[d_i^2(\xi_i^2-1)^2\Big]=2\|D\|_\HS^2,
\qquad \sum_{i=1}^m \E\Big[|d_i(\xi_i^2-1)|^3\Big] \leq
C\|D\|_\HS^2 \cdot \|D\|,\]
and $\|D\|/\|D\|_\HS \to 0$, the result follows from the
Lyapunov central limit theorem.
\end{proof}

\begin{proof}[Proof of Theorem \ref{thm:CLT}]
For the given $\lambda$ and $\v$, we have $\v=(\v_1,0)$, where
$\v_1 \in \R^L$ and $K(\lambda)\v_1=0$. Furthermore, recall from the proof of
Theorem \ref{thm:eigenvectors} that for the given $\hl$,
there is a unit vector $\hv_1$ where $\hK(\hl)\hv_1=0$ and
$\|\hv_1-\v_1\| \prec n^{-1/2}$. Lemma \ref{lemma:Kbound} implies
$\|\hK(\hl)\|<C$ with probability $1-n^{-D}$, so
\[\v_1'\hK(\hl)\v_1=(\hv_1-\v_1)'\hK(\hl)(\hv_1-\v_1) \prec n^{-1}.\]
Applying this and $\v_1'K(\lambda)\v_1=0$, we obtain
\begin{equation}\label{eq:vKv}
\v_1'\big(\hK(\hl)-\hK(\lambda)\big)\v_1+
\v_1'\big(\hK(\lambda)-K(\lambda)\big)\v_1 \prec n^{-1}.
\end{equation}
Recall that Theorem \ref{thm:spikemapping} implies $\lambda-\hl \prec n^{-1/2}$.
Applying a second-order Taylor expansion for the first term of (\ref{eq:vKv}),
approximating $\partial_\lambda \hK(\lambda)$ by $\partial_\lambda K(\lambda)$
using (\ref{eq:dhKbound}), and bounding $\partial_\lambda^2 \hK(\lambda)$ using
Lemma \ref{lemma:Kbound}, we get
\begin{equation}\label{eq:hKhlsecondorder}
\v_1'\big(\hK(\hl)-\hK(\lambda)\big)\v_1=
(\hl-\lambda)\v_1'\partial_\lambda K(\lambda)\v_1+\O(n^{-1}).
\end{equation}
For the second term of (\ref{eq:vKv}), recall
$\hK(\lambda)-K(\lambda)=E_1(\lambda)+E_2(\lambda)$ with
$E_1$ and $E_2$ as in (\ref{eq:E1}--\ref{eq:E2}). Recall also from
(\ref{eq:E2bound}) that
$\|E_2(\lambda)\| \prec n^{-1}$.
Then (\ref{eq:vKv}) becomes
\begin{equation}\label{eq:hllambdadiff}
(\hl-\lambda)\v_1'\partial_\lambda K(\lambda)\v_1+\v_1'E_1(\lambda)\v_1
\prec n^{-1}.
\end{equation}

Observe that $\Xi$ is independent of $X$, and $\z=\Xi \v_1 \in \R^M$ has
independent Gaussian entries. The covariance matrix of $\z$ is $VV'$ for the
diagonal matrix
\[V=V'=N^{-1/2}\sum_{r=1}^k (\v_1'\Gamma_r \v_1)^{1/2}P_r.\]
Then $\v_1'E_1(\lambda)\v_1=\E[-\z'G_M(\lambda)\z \mid X]$,
and we may apply Lemma \ref{lemma:lyapunov} conditional on $X$:
Lemma \ref{lemma:Kbound} implies, with high probability,
$\|(G_M)_{rs}/(\sigma_r\sigma_s)\|<C$ for each $r,s$, so
$\|V'G_M(\lambda)V\|<C/n$.
On the other hand, since $\v'T(\lambda)\v=0$,
we have from (\ref{eq:TTheta}) and (\ref{eq:m0stieltjes})
\[\left|\sum_{r=1}^k t_r(\lambda)\v_1'\oV_r\Theta_r\oV_r'\v_1\right|
=\left|\frac{1}{m_0(\lambda)}\right| \geq \delta.\]
Then for some constant $c>0$ and
some $r \in \{1,\ldots,k\}$ we must have
\[|t_r(\lambda)|>c,\qquad |\v_1'\oV_r\Theta_r\oV_r'\v_1|>c.\]
The latter implies $\v_1'(\sigma_r^2\Gamma_r)\v_1>c$. The former implies
$(N\sigma_r^2)^{-1}|\Tr_r G_M(\lambda)|>c$ on an event of
probability $1-n^{-D}$, by (\ref{eq:GMPiM}) and (\ref{eq:E2bound}). Then
$\|(G_M)_{rr}/\sigma_r^2\|_\HS>c\sqrt{n}$, and for this $r$
\begin{equation}\label{eq:variancelowerbound}
\|V'G_M(\lambda)V\|_\HS \geq N^{-1}\v_1'(\sigma_r^2\Gamma_r)\v_1
\|(G_M)_{rr}/\sigma_r^2\|_\HS>cn^{-1/2}.
\end{equation}
Thus, on this high probability event, we have
$\|V'G_M(\lambda)V\|/\|V'G_M(\lambda)V\|_\HS<Cn^{-1/2}$. Applying Lemma
\ref{lemma:lyapunov} conditional on $X$ and this event,
\[\|V'G_M(\lambda)V\|_\HS^{-1}(\v_1'E_1(\lambda)\v_1) \to \N(0,2).\]
As the limit does not depend on $X$, this convergence holds
unconditionally as well. Then, applying this,
$\v_1'\partial_\lambda K(\lambda) \v_1=-\v'\partial_\lambda T(\lambda) \v$,
and (\ref{eq:variancelowerbound}) to (\ref{eq:hllambdadiff}), we have
\begin{equation}\label{eq:CLTconditional}
\frac{(\v'\partial_\lambda T(\lambda) \v)}{\sqrt{2}\|V'G_M(\lambda)V\|_\HS}
(\hl-\lambda) \to \N(0,1).
\end{equation}

Finally, let us approximate $\|V'G_M(\lambda)V\|_\HS$: We have
\[\|V'G_MV\|_\HS^2=\Tr G_MVV'G_MVV'
=\sum_{r,s=1}^k N^{-2}(\v_1'\Gamma_r\v_1)(\v_1'\Gamma_s\v_1)\Tr G_MP_rG_MP_s.\]
We apply $G_M=\Pi_M+F\Delta F$ from (\ref{eq:GMPiM}) and expand the above.
Note that Lemma \ref{lemma:matrixlocallaw} implies
\[\frac{\Tr \Pi_MP_rF\Delta FP_s}{\sigma_r^2\sigma_s^2}
\prec n^{-1/2}\frac{\|FP_s\Pi_MP_rF\|_\HS}{\sigma_r^2\sigma_s^2} \prec
\|FP_s/\sigma_s\| \cdot \|(\Pi_M)_{sr}/(\sigma_s\sigma_r)\| \cdot
\|P_rF/\sigma_r\| \prec 1,\]
so the cross terms of the expansion are negligible, and we have
\[\frac{\Tr G_MP_rG_MP_s}{\sigma_r^2\sigma_s^2}=
\frac{\Tr \Pi_MP_r\Pi_MP_s}{\sigma_r^2\sigma_s^2}+\frac{\Tr F\Delta FP_rF\Delta
FP_s}{\sigma_r^2\sigma_s^2}+\O(1).\]
The first term on the right may be written as $\Tr
(P_s\Pi_MP_r)(P_r\Pi_MP_s)/(\sigma_r^2\sigma_s^2)
=\|\Pi_M/(\sigma_r\sigma_s)\|_{rs}^2$.
For the second term, applying Lemma \ref{lemma:secondorderapprox},
\begin{align*}
\frac{\Tr \Delta FP_rF\Delta FP_sF}{\sigma_r^2\sigma_s^2}
&=(N\sigma_r^2\sigma_s^2)^{-1}
(\partial_\lambda m_0)\Tr_r \big[F(\Id+m_0F)^{-2}F\big]
\Tr_s\big[F(\Id+m_0F)^{-2}F\big]+\O(n^{1/2})\\
&=N(\partial_\lambda t_r)(\partial_\lambda
t_s)(\partial_\lambda m_0)^{-1}+\O(n^{1/2}).
\end{align*}
Then, recalling $w_{rs}$ from (\ref{eq:wrs}) and applying
$\v_1'(\sigma_r^2\Gamma_r)\v_1=\v'\Sigma_r\v$ by (\ref{eq:Gammar}), we obtain
\begin{align*}
\|V'G_MV\|_\HS^2&=N^{-1}\sum_{r,s=1}^k
w_{rs}(\lambda)(\v'\Sigma_r\v)(\v'\Sigma_s\v)
+(N\partial_\lambda m_0)^{-1}\left(\sum_{r=1}^k (\partial_\lambda t_r)
\v'\Sigma_r\v\right)^2+\O(n^{-3/2})\\
&=N^{-1}\sum_{r,s=1}^k w_{rs}(\lambda)(\v'\Sigma_r\v)(\v'\Sigma_s\v)
+(N\partial_\lambda m_0)^{-1}(\v'\partial_\lambda T \v-1)^2
+\O(n^{-3/2}),
\end{align*}
where the second line applies (\ref{eq:TSigma}).
By (\ref{eq:variancelowerbound}), the $\O(n^{-3/2})$ error above is
negligible. Then Theorem \ref{thm:CLT} follows from
this and (\ref{eq:CLTconditional}).
\end{proof}

\subsection{Guarantees for spike estimation}\label{subsec:estimation}
Finally, we prove Theorem \ref{thm:estimation}. For notational
convenience, we assume $r=1$. Part (c) of Theorem \ref{thm:estimation} follows
immediately from the observation that Algorithm \ref{alg:spikes} uses only the
eigenvalues/eigenvectors of $\hSigma(\ba)$, so each estimated eigenvector $\hv$
is equivariant under orthogonal rotations on $\cS^\perp$.

For parts (a) and (b),
we may decompose their content into the following three claims.
\begin{enumerate}[1.]
\item With probability at least $1-n^{-D}$, for each $(\hmu,\hv) \in \cM$, 
there exists a spike eigenvalue and
eigenvector $(\mu,\v)$ of $\Sigma_1$ and a scalar $\alpha \in
(0,1]$ such that $|\hmu-\mu|<n^{-1/2+\eps}$ and
$\|P_\cS \hv-\alpha\v\|<n^{-1/2+\eps}$.
\item
For each spike eigenvalue $\mu$ of $\Sigma_1$ and a sufficiently small constant $\eps>0$, with probability at least $1-n^{-D}$, there is at most
one pair $(\hmu,\hv) \in \cM$ where $|\hmu-\mu|<\eps$.
\item
For a constant $c_0>0$ independent of $\bar{C}$ in Assumption \ref{assump:main},
and for each spike eigenvalue $\mu$ of $\Sigma_1$ with $\mu>c_0$,
with probability at least $1-n^{-D}$,
there exists $(\hmu,\hv) \in \cM$ such that $|\hmu-\mu|<n^{-1/2+\eps}$.
\end{enumerate}

The first claim will be straightforward to show from the preceding
probabilistic results. The second and third claims require a certain
injectivity and surjectivity property of the map
$(\hl,\ba) \mapsto (t_1(\hl,\ba),\ldots,t_k(\hl,\ba))$ for $\ba \in S^{k-1}$
and $\hl \in \spec(\hSigma(\ba))$. For this, we will use Assumption
\ref{assump:estimation}.

Denote by $m_0(\lambda,\ba)$, $T(\lambda,\ba)$, etc.\ these functions
defined for $B=B(\ba)=a_1B_1+\ldots+a_kB_k$.
We record the following basic bounds.

\begin{lemma}\label{lemma:continuityina}
There is a constant $C>0$ such that
\begin{enumerate}[(a)]
\item For all $\ba \in S^{k-1}$,
$\lambda \in \R \setminus \supp(\mu_0(\ba))_\delta$, and $r=1,\ldots,k$,
\[|m_0(\lambda,\ba)|<C, \quad |\partial_\lambda m_0(\lambda,\ba)|<C,
\quad |\partial_\lambda^2 m_0(\lambda,\ba)|<C,
 \quad |m_0(\lambda,\ba)|^{-1}<C(|\lambda| \vee 1),\]
\[\|\partial_\ba m_0(\lambda,\ba)\|<C,
\quad \|\partial_\lambda \partial_\ba m_0(\lambda,\ba)\|<C,
\quad \|\partial_\ba^2 m_0(\lambda,\ba)\|<C,\]
\[|t_r(\lambda,\ba)|<C, \quad |\partial_\lambda t_r(\lambda,\ba)|<C,
\quad |\partial_\lambda^2 t_r(\lambda,\ba)|<C,\]
\[\|\partial_\ba t_r(\lambda,\ba)\|<C,
\quad \|\partial_\lambda \partial_\ba t_r(\lambda,\ba)\|<C,
\quad \|\partial_\ba^2 t_r(\lambda,\ba)\|<C.\]
\item For all $\ba \in S^{k-1}$, the roots $\lambda$
of $0=\det T(\lambda,\ba)$ are contained in $[-C,C]$.
\item For any $D>0$ and all $n \geq n_0(\delta,D)$,
with probability at least $1-n^{-D}$,
\[\sup_{\ba \in S^{k-1}} \|\hSigma(\ba)\|<C, \qquad
\sup_{\ba \in \R^k} \max_{r=1}^k \|\partial_{a_r} \hSigma(\ba)\|<C.\]
\end{enumerate}
\end{lemma}
\begin{proof}
For (a), the upper bounds on $m_0$, $\partial_\lambda m_0$, and
$\partial_\lambda^2 m_0$ follow from (\ref{eq:m0stieltjes}) and the condition
$|x-\lambda| \geq \delta$ for all $x \in \supp(\mu_0(\ba))_\delta$.
The upper bound on
$m_0^{-1}$ follows from (\ref{eq:MP}) and the bounds $\|F(\ba)\|<C$ and
$\|(\Id+m_0F(\ba))^{-1}\|<C$, the latter holding by
Proposition \ref{prop:m0regularoutsiderestate}. For the derivatives in $\ba$,
fix $r$ and denote $m_0=m_0(\lambda,\ba)$, $F=F(\ba)$,
and $\partial=\partial_{a_r}$. We have
\begin{align}
\partial \Big(F(\Id+m_0F)^{-1}\Big)
&=(\partial F)(\Id+m_0F)^{-1}-F(\Id+m_0F)^{-1}\Big((\partial m_0)F+
m_0(\partial F)\Big)(\Id+m_0F)^{-1}\nonumber\\
&=(\Id+m_0F)^{-1}(\partial F)(\Id+m_0F)^{-1}-(\partial m_0)F^2(\Id+m_0F)^{-2}.
\label{eq:derFJ}
\end{align}
Then, differentiating (\ref{eq:MP}) with respect to $a_r$ and also with respect
to $z=\lambda$, we obtain the equations
\begin{align*}
0&=(\partial m_0)(m_0^{-2}-N^{-1}\Tr[F^2(\Id+m_0F)^{-2}])
+N^{-1}\Tr[(\partial F)(\Id+m_0F)^{-2}],\\
1&=(\partial_\lambda m_0)(m_0^{-2}-N^{-1}\Tr[F^2(\Id+m_0F)^{-2}]).
\end{align*}
Applying the second equation to the first,
\begin{equation}\label{eq:darm0}
\partial m_0=-(\partial_\lambda m_0)N^{-1}\Tr[(\partial F)(\Id+m_0F)^{-2}].
\end{equation}
The bound $\|\partial_\ba m_0\|<C$ then follows from
$|\partial_\lambda m_0|<C$, $\|F_r\|<C$, and $\|(\Id+m_0F)^{-1}\|<C$.
The bounds $\|\partial_\lambda \partial_\ba m_0\|<C$ and
$\|\partial_\ba^2 m_0\|<C$ follow from the chain rule. For $t_r(\lambda,\ba)$,
recall from (\ref{eq:GMPiM}) that
\[-t_r=(N\sigma_r^2)^{-1} \Tr_r \Pi_M=
(N\sigma_r^2)^{-1} \Tr_r (m_0F(\Id+m_0F)^{-1}F-F).\]
The bound $|t_r|<C$ then follows from $\|P_rFP_r\|<C\sigma_r^2$ and
$\|P_rF\|<C\sigma_r$, where $P_r$ is the projection onto block $r$.
The bounds on the derivatives of $t_r$ follow from the
chain rule and those on $m_0$.

Part (a) implies $\|T(\lambda,\ba)\|<C$ for all $\lambda \in \R \setminus
\supp(\mu_0(\ba))_\delta$. As Proposition \ref{prop:Tproperties}(c)
shows $\partial_\lambda T(\lambda,\ba)$ has smallest
eigenvalue at least 1, $T(\lambda,\ba)$ must be non-singular for all $\lambda$
outside $[-C,C]$ for some $C>0$, implying part (b). Finally, part (c) follows
from $\hSigma(\ba)=\sum_r a_r Y'B_rY$ and the observation that
$\|Y'B_rY\|<C$ for all $r=1,\ldots,k$ with probability $1-n^{-D}$.
\end{proof}

Next, we verify that for the conclusions of Theorems
\ref{thm:spikemapping} and \ref{thm:eigenvectors}(a), we may take a union bound
over $\ba \in S^{k-1}$.

\begin{lemma}\label{lemma:unionbounda}
Under the conditions of Theorem \ref{thm:estimation}, for all $n \geq
n_0(\delta,\eps,D)$, with probability $1-n^{-D}$ the conclusions of Theorems
\ref{thm:spikemapping} and \ref{thm:eigenvectors}(a) hold simultaneously for all
$\ba \in S^{k-1}$.
\end{lemma}
\begin{proof}
Consider a covering net $\N \subset S^{k-1}$ with $|\N| \leq n^C$ for some
$C=C(k)>0$, such that for all $\ba \in S^{k-1}$ there exists $\ba_0 \in \N$
where $\|\ba_0-\ba\|<n^{-1/2}$. With probability
$1-n^{-D}$, the conclusions of Theorems \ref{thm:spikemapping}
and \ref{thm:eigenvectors} hold with constants $\delta/2$ and $\eps/2$
simultaneously over $\ba_0 \in \N$ by a union bound. Furthermore, by Lemma
\ref{lemma:continuityina}, with probability at least $1-n^{-D}$ we have
$\|\hSigma(\ba)-\hSigma(\ba_0)\|<Cn^{-1/2}$ for all $\ba \in S^{k-1}$, where
$\ba_0$ is the closest point to $\ba$ in $\N$. Note that by
Theorem \ref{thm:sticktobulk}, this implies also
$\supp(\mu_0(\ba)) \subseteq \supp(\mu_0(\ba_0))_{\delta/4}$ and
$\supp(\mu_0(\ba_0)) \subseteq \supp(\mu_0(\ba))_{\delta/4}$ for all large $n$.

On the above event, consider any $\ba \in S^{k-1}$ and nearest point
$\ba_0 \in \N$. Let $\Lambda_{\delta/2}(\ba_0)$ and $\hLambda_{\delta/2}(\ba_0)$
be the sets guaranteed by Theorem \ref{thm:spikemapping} at $\ba_0$, so
\[\ordereddist(\Lambda_{\delta/2}(\ba_0),\hLambda_{\delta/2}(\ba_0))
<n^{-1/2+\eps/2}.\]
The condition $\|\hSigma(\ba)-\hSigma(\ba_0)\|<Cn^{-1/2}$
implies there is $\hLambda_\delta(\ba) \subset \spec(\hSigma(\ba))$ such that
\[\ordereddist(\hLambda_\delta(\ba),\hLambda_{\delta/2}(\ba_0))
<Cn^{-1/2}.\]
Since $\hLambda_{\delta/2}(\ba_0)$ contains all eigenvalues of $\hSigma(\ba_0)$
outside $\supp(\mu_0(\ba_0))_{\delta/2}$, we have that $\hLambda_\delta(\ba)$
contains all eigenvalues of $\hSigma(\ba)$ outside $\supp(\mu_0(\ba))_\delta$.
On the other hand, $\Lambda_{\delta/2}(\ba_0)$ is a subset of roots of
$0=\det(K(\lambda,\ba_0))$, where $K(\lambda,\ba_0)$ is defined
by (\ref{eq:K}) at $B=B(\ba_0)$. Letting $\mu_1(\lambda,\ba_0) \leq \ldots \leq
\mu_L(\lambda,\ba_0)$ denote the eigenvalues of $K(\lambda,\ba_0)$, the multiset
$\Lambda_0(\ba_0)$ is in 1-to-1 correspondence with
pairs $(\ell,\lambda_0)$ where $\mu_\ell(\lambda_0,\ba_0)=0$. For each such
$(\ell,\lambda_0)$, Lemma \ref{lemma:continuityina} implies
$\|K(\lambda_0,\ba)-K(\lambda_0,\ba_0)\|<Cn^{-1/2}$, so
$|\mu_\ell(\lambda_0,\ba)|<Cn^{-1/2}$. As $-K$ is the upper $L \times L$
submatrix of $T$, Proposition \ref{prop:Tproperties}(c)
implies $\mu_\ell$ decreases in $\lambda$ at a rate of at least 1, so
 $\mu_\ell(\lambda,\ba)=0$ for some $\lambda$ with
$|\lambda-\lambda_0|<Cn^{-1/2}$. Thus there exists $\Lambda_\delta(\ba)
\subseteq \Lambda_0(\ba)$ where
\[\ordereddist(\Lambda_\delta(\ba),\Lambda_{\delta/2}(\ba_0))<Cn^{-1/2},\]
and similarly $\Lambda_\delta(\ba)$ contains all elements of $\Lambda_0(\ba)$
outside $\supp(\mu_0(\ba))_\delta$. Then
\[\ordereddist(\Lambda_\delta(\ba),\hLambda_\delta(\ba))<2Cn^{-1/2}+n^{-1/2+\eps/2}<n^{-1/2+\eps},\]
so the conclusion of Theorem \ref{thm:spikemapping} holds at
each $\ba \in S^{k-1}$.

For Theorem \ref{thm:eigenvectors}(a), let $\lambda \in \Lambda_0(\ba)$ be
separated from other elements of $\Lambda_0(\ba)$
by $\delta$. Then Proposition \ref{prop:Tproperties}(c) implies 0 is separated
from other eigenvalues of $T(\lambda,\ba)$ by $\delta$.
Letting $\lambda_0 \in \Lambda_0(\ba_0)$
be such that $|\lambda_0-\lambda|<Cn^{-1/2}$, as identified above,
Lemma \ref{lemma:continuityina} implies
$\|T(\lambda_0,\ba_0)-T(\lambda,\ba)\|<Cn^{-1/2}$. Thus
if $\v$ and $\v_0$ are the null unit eigenvectors of
$T(\lambda,\ba)$ and $T(\lambda_0,\ba_0)$, then $\|\v-\v_0\|<Cn^{-1/2}$ for an
appropriate choice of sign.
Similarly, if $\hl \in \spec(\hSigma(\ba))$ and $\hl_0 \in
\spec(\hSigma(\ba_0))$ are such that
$|\hl-\lambda|<n^{-1/2+\eps}$ and $|\hl_0-\lambda_0|<n^{-1/2+\eps}$, then
$\hl$ is separated from other eigenvalues of $\hSigma(\ba)$ by
$\delta-Cn^{-1/2+\eps}$, and the bound
$\|\hSigma(\ba)-\hSigma(\ba_0)\|<Cn^{-1/2}$ implies that
the corresponding eigenvectors $\hv$ and $\hv_0$ satisfy
$\|\hv-\hv_0\|<Cn^{-1/2}$. Lemma \ref{lemma:continuityina} finally implies
$\|\partial_\lambda T(\lambda,\ba)-\partial_\lambda
T(\lambda_0,\ba_0)\|<Cn^{-1/2}$, so the conclusion of Theorem
\ref{thm:eigenvectors}(a) at $\ba$ follows from that at $\ba_0$.
\end{proof}

We may now establish the first of the above three claims
for Theorem \ref{thm:estimation}.

\begin{proof}[Proof of Claim 1]
Consider the event of probability $1-n^{-D}$ on which the conclusions of
Theorems \ref{thm:spikemapping}(a) and \ref{thm:eigenvectors} hold
simultaneously over $\ba \in S^{k-1}$.

For each $(\hmu,\hv) \in \cM$, there are $\ba \in S^{k-1}$ and $\hl \in
\spec(\hSigma(\ba)) \cap \I_\delta(\ba)$ with $\hmu=\hl/t_1(\hl,\ba)$ and
$t_2(\hl,\ba)=\ldots=t_k(\hl,\ba)=0$. On the above event, for each
such $(\hl,\ba)$, there exists $\lambda$ with $|\hl-\lambda|<n^{-1/2+\eps}$
and $0=\det T(\lambda,\ba)$. Then Lemma \ref{lemma:continuityina} implies
\begin{equation}\label{eq:ThTbound}
\|T(\hl,\ba)-T(\lambda,\ba)\|<Cn^{-1/2+\eps}.
\end{equation}
An eigenvalue of $T(\lambda,\ba)$ is 0, so an eigenvalue of $T(\hl,\ba)$ has
magnitude at most $Cn^{-1/2+\eps}$. From the two
equivalent forms (\ref{eq:TSigma}) and (\ref{eq:TTheta}) of $T$
and the condition
$t_2(\hl,\ba)=\ldots=t_k(\hl,\ba)=0$,
\begin{equation}\label{eq:Tform}
T(\hl,\ba)=\hl\Id-t_1(\hl,\ba)\Sigma_1
=-\frac{1}{m_0(\hl,\ba)}\Id-t_1(\hl,\ba)V_1'\Theta_1V_1.
\end{equation}
Since $|m_0(\lambda,\ba)|<C$, the second form above implies that the
$O(n^{-1/2+\eps})$ eigenvalue of $T(\hl,\ba)$ must be
$\hl-t_1(\hl,\ba)\mu=-1/m_0(\hl,\ba)-t_1(\hl,\ba)\theta$ for a spike
eigenvalue $\mu=\theta+\sigma_1^2$ of $\Sigma_1$. As $\theta$ is bounded, the
condition $|-1/m_0(\hl,\ba)-t_1(\hl,\ba)\theta|<Cn^{-1/2+\eps}$ implies in
particular that $|t_1(\hl,\ba)|>c$ for a constant $c>0$.
Then dividing $|\hl-t_1(\hl,\ba)\mu|<Cn^{-1/2+\eps}$ by $t_1(\hl,\ba)$,
$|\hmu-\mu|<Cn^{-1/2+\eps}$ for a different constant $C>0$.
Furthermore, on the above event, $\|P_\cS\hv-\alpha\w\|<n^{-1/2+\eps}$ for
the null vector $\w$ of $T(\lambda,\ba)$ and for
$\alpha=(\w'\partial_\lambda T(\lambda,\ba)\w)^{-1/2}$. By the second form in
(\ref{eq:Tform}), the separation of values of $\Theta_1$ by $\tau$, and the
above lower bound on $t_1(\hl,\ba)$, the null eigenvalue of $T(\lambda,\ba)$ is
separated from other eigenvalues by a constant $c>0$. Then
(\ref{eq:ThTbound}) implies $\|\w-\v\|<Cn^{-1/2+\eps}$
where $\v$ is the (appropriately signed) eigenvector of $T(\hl,\ba)$
corresponding to the eigenvalue $\hl-t_1(\hl,\ba)\mu$. This is exactly the
eigenvector of $\Sigma_1$ corresponding to $\mu$, and thus
$\|P_\cS\hv-\alpha\v\|<Cn^{-1/2+\eps}$.
\end{proof}

For the remaining two claims, let us first sketch the argument at a high level:
Suppose $\mu$ is a spike eigenvalue of
$\Sigma_1$, and $\ba_0 \in S^{k-1}$ and $\hl_0 \in
\spec(\hSigma(\ba_0))$ are such that
\[\hl_0/t_1(\hl_0,\ba_0) \approx \mu, \qquad t_r(\hl_0,\ba_0)
\approx 0 \text{ for all } r=2,\ldots,k.\]
We will show that under Assumption \ref{assump:estimation},
this holds for some $(\hl_0,\ba_0)$ whenever $\mu$ is
sufficiently large. The separation of $\mu$ from other eigenvalues of
$\Sigma_1$ will
imply that $\hl_0$ is separated from other eigenvalues of $\hSigma(\ba_0)$.
Then for all $\ba \in S^{k-1}$ in a neighborhood of
$\ba_0$, we may identify an eigenvalue $\hl(\ba)$ of $\hSigma(\ba)$ such that
$\hl(\ba_0)=\hl_0$ and $\hl(\ba)$ varies analytically in $\ba$. Applying a
version of the inverse function theorem, we will show that the mapping
\[\ba \mapsto (t_2(\hl(\ba),\ba),\ldots,t_k(\hl(\ba),\ba))\]
is injective in this neighborhood of $\ba_0$, and its image contains 0. This
local injectivity, together with Assumption \ref{assump:estimation}, will
imply Claim 2. The image containing 0 will imply Claim 3.

We use the following quantitative version of the inverse function theorem to
carry out this argument.
\begin{lemma}\label{lemma:inversefunction}
Fix constants $C,c_0,c_1,m>0$. Let $\x_0 \in \R^m$, let
$U=\{\x \in \R^m:\|\x-\x_0\|<c_0\}$, and let $f:U \to \R^m$ be twice
continuously differentiable. Denote by $\der f \in \R^{m \times m}$ the
derivative of $f$, and suppose for all $\v \in \R^m$,
$i,j,k \in \{1,\ldots,m\}$, and $\x \in U$ that
\[\|(\der f(\x_0))\v\| \geq c_1\|\v\|,
\qquad |\partial_{x_i}\partial_{x_j} f_k(\x)|<C.\]
Then there are constants $\eps_0,\eps_1,c>0$ such that
$f$ is injective on $U_0=\{\x \in \R^m:\|\x-\x_0\|<\eps_0\}$,
$\|f(\x_1)-f(\x_2)\| \geq c\|\x_1-\x_2\|$ for all $\x_1,\x_2 \in U_0$,
and the image $f(U_0)$ contains $\{\y \in \R^m:\|\y-f(\x_0)\|<\eps_1\}$.
\end{lemma}
\begin{proof}
Assume without loss of generality $\x_0=0$ and $f(\x_0)=0$.
By Taylor's theorem and the given second derivative bound, for all $\x \in U$
and a constant $C>0$, $\|\der f(\x)-\der f(0)\| \leq C\|\x\|$.
Then for sufficiently small $\eps_0>0$, all
$\x \in U_0=\{\x \in \R^m:\|\x\|<\eps_0\}$, and all $\v \in \R^m$,
\[\|(\der f(\x))\v\| \geq \|(\der f(0))\v\|-C\|\x\|\|\v\|
\geq (c_1-C\eps_0)\|\v\| \geq (c_1/2)\|\v\|.\]
Furthermore, for all $\x_1,\x_2 \in U$,
\[\|f(\x_2)-f(\x_1)-(\der f(\x_1))(\x_2-\x_1)\| \leq C\|\x_2-\x_1\|^2.\]
Then for sufficiently small $\eps_0>0$ and all $\x_1,\x_2 \in U_0$,
\begin{equation}\label{eq:fx1x2bound}
\|f(\x_2)-f(\x_1)\| \geq (c_1/2)\|\x_2-\x_1\|-C\|\x_2-\x_1\|^2
\geq c\|\x_2-\x_1\|
\end{equation}
for a constant $c>0$. In particular, $f$ is injective on $U_0$.

To prove the surjectivity claim, let $K=\{\x \in \R^m:\|\x\| \leq
\eps_0/2\} \subset U_0$. For a sufficiently small constant $\eps_1>0$, the
above applied with $\x_2=\x$ on the boundary of $K$ and $\x_1=0$ implies
\[\|f(\x)\|>2\eps_1 \text{ for all } \x  \text{ on the boundary of } K.\]
Fix any $\y \in \R^m$ with $\|\y\|<\eps_1$, and define
$h(\x)=\|f(\x)-\y\|^2$ over $\x \in K$. As $K$ is compact, there is $\x_* \in K$
that minimizes $h$. Since $h(0)=\|\y\|^2<\eps_1^2$ while
$h(\x)>\eps_1^2$ for $\x$ on the boundary of $K$ by the above, $\x_*$ is in
the interior of $K$. Then
\[0=\der h(\x_*)=2(f(\x_*)-\y)'(\der f(\x_*)).\]
Since $\der f(\x_*)$ is invertible by (\ref{eq:fx1x2bound}), this implies
$f(\x_*)=\y$. So $f(U_0)$ contains any such $\y$.
\end{proof}

We now make the above proof sketch for Claims 2 and 3 precise.
\begin{lemma}\label{lemma:tplusinjective}
Let $\mu=\theta+\sigma_1^2$ be the $\ell^{\text{th}}$ largest
spike eigenvalue of $\Sigma_1$. Define
\[t_+(\lambda,\ba)=(t_2(\lambda,\ba),\ldots,t_k(\lambda,\ba)).\]
Then there exist constants $c,\eps_0,\eps_1>0$ such that for any $D>0$ and all
$n \geq n_0(\delta,D)$, under the conditions of Theorem \ref{thm:estimation},
the following holds with probability at least $1-n^{-D}$: For all
$\ba_0 \in S^{k-1}$, if there exists
$\hl_0 \in \spec(\hSigma(\ba_0)) \cap \I_\delta(\ba_0)$ which satisfies
\begin{equation}\label{eq:tnearaxis}
\left|-\frac{1}{m_0(\hl_0,\ba_0)}-t_1(\hl_0,\ba_0)\,\theta\right|<\eps_1,
\qquad \|t_+(\hl_0,\ba_0)\|<\eps_1,
\end{equation}
then:
\begin{itemize}
\item $\hl_0$ is the $\ell^{\text{th}}$ largest eigenvalue of $\hSigma(\ba_0)$.
\item The $\ell^{\text{th}}$ largest eigenvalue $\hl(\ba)$ of
$\hSigma(\ba)$ is simple over $O=\{\ba \in \R^k:\|\ba-\ba_0\|<\eps_0\}$.
\item The map $\hf(\ba)=t_+(\hl(\ba),\ba)$ is injective on $U=O \cap S^{k-1}$
and satisfies $\|\hf(\ba_1)-\hf(\ba_2)\| \geq c\|\ba_1-\ba_2\|$
for all $\ba_1,\ba_2 \in U$. Furthermore,
its image $\hf(U)$ contains $\{\bt \in \R^{k-1}:\|\bt\|<\eps_1\}$.
\end{itemize}
\end{lemma}
\begin{proof}
Throughout the proof, we use the convention that constants $C,c>0$ do not
depend on $\eps_0,\eps_1$. 

Let $\N \subset S^{k-1}$ be a covering net with $|\N| \leq n^C$, such that for
each $\ba \in S^{k-1}$ there is $\ba_0 \in \N$ with $\|\ba-\ba_0\|<n^{-1/2}$.
It suffices to establish the result for each fixed
$\ba_0 \in \N$ with probability $1-n^{-D}$. The result then
holds simultaneously for all $\ba_0 \in S^{k-1}$ by a union bound over $\N$ and
the Lipschitz continuity of $m_0^{-1}$, $t_1$, and $t_+$ as established in
Lemma \ref{lemma:continuityina}.

Thus, let us fix $\ba_0 \in \N$. Consider the good
event where the conclusion of Theorem \ref{thm:spikemapping} holds for
$B=B(\ba_0)$, and also $\|\hSigma(\ba)-\hSigma(\ba_0)\| \leq
C\|\ba-\ba_0\|$ and $\|\partial_{a_r} \hSigma(\ba)\|<C$ for all $r=1,\ldots,k$
and $\ba \in \R^k$. Consider $m_0$, $t_r$, $T$ defined at $\ba_0$,
and (for notational convenience) suppress their dependence on $\ba_0$.
On this good event, for each $\hl_0$ satisfying (\ref{eq:tnearaxis}),
there exists $\lambda_0$ with $|\lambda_0-\hl_0|<n^{-1/2+\eps}$
and $0=\det T(\lambda_0)$. Lemma \ref{lemma:continuityina} implies
$|m_0(\lambda_0)^{-1}-m_0(\hl_0)^{-1}|<Cn^{-1/2+\eps}$ and
$|t_r(\lambda_0)-t_r(\hl_0)|<Cn^{-1/2+\eps}$ for
each $r$. Then (\ref{eq:TTheta}), (\ref{eq:tnearaxis}), and the condition
$\theta \geq \tau$ in Theorem \ref{thm:estimation} imply
\begin{equation}\label{eq:TapproxSigma1}
\|T(\lambda_0)+m_0(\lambda_0)^{-1}(\Id-\theta^{-1}
V_1\Theta_1V_1')\|<C\eps_1.
\end{equation}

Since $\lambda_0 \in \I_\delta(\ba_0)$ is greater than $\supp(\mu_0)$, we have
$m_0(\lambda_0)<0$ by (\ref{eq:m0stieltjes}). As $\theta$ is the
$\ell^{\text{th}}$ largest value of $\Theta_1$, this implies the
$\ell^{\text{th}}$
smallest eigenvalue of $-m_0(\lambda_0)^{-1}(\Id-\theta^{-1}V_1\Theta_1V_1)$ is
0. Then, denoting by $\mu_1(\lambda) \leq \ldots \leq \mu_p(\lambda)$
the eigenvalues of $T(\lambda)$, (\ref{eq:TapproxSigma1}) yields
$|\mu_\ell(\lambda_0)|<C\eps_1$.
The separation of values of $\Theta_1$ by $\tau$ further implies
$\mu_{\ell-1}(\lambda_0)<-|m_0(\lambda_0)\theta|^{-1}
\tau+C\eps_1$ and
$\mu_{\ell+1}(\lambda_0)>|m_0(\lambda_0)\theta|^{-1}\tau-C\eps_1$.
As $\theta<C$ and $|m_0(\lambda_0)|<C$, for sufficiently small $\eps_1$ this
yields
\[\mu_{\ell+1}(\lambda_0)>c, \qquad \mu_{\ell-1}(\lambda_0)<-c, \qquad
\mu_\ell(\lambda_0)=0\]
for a constant $c>0$, where the third statement must hold
because $0=\det T(\lambda_0)$.
For each $j=1,\ldots,p$ and all $\lambda<\lambda'$ in $\I_\delta(\ba_0)$, note
that 
\[\lambda'-\lambda \leq \mu_j(\lambda')-\mu_j(\lambda) <C(\lambda'-\lambda),\]
where the lower bound follows from Proposition \ref{prop:Tproperties}(c) and the
upper bound follows from $\|\partial_\lambda T(\lambda)\|<C$.
Then $\lambda_0$ is separated from all other roots of $0=\det T(\lambda)$ by a
constant $c>0$. Furthermore, there are exactly $\ell-1$ roots of $0=\det
T(\lambda)$ which are greater than $\lambda_0$,
one corresponding to each $\mu_j$ for $j=1,\ldots,\ell-1$. Then on the above
good event, there can only be one such $\hl_0$ satisfying (\ref{eq:tnearaxis}),
which is the $\ell^{\text{th}}$ largest eigenvalue
of $\hSigma(\ba_0)$. Furthermore, it is separated from all
other eigenvalues of $\hSigma(\ba_0)$ by a constant $c>0$, and for a
sufficiently small constant $\eps_0>0$, the $\ell^{\text{th}}$ largest
eigenvalue $\hl(\ba)$ is simple and analytic on $O=\{\ba \in
\R^k:\|\ba-\ba_0\|<\eps_0\}$. This verifies the first two statements.

To verify the third statement, consider a chart
$(V,\varphi)$ where $V=\{\v \in \R^{k-1}:\|\v\|<\eps_0\}$,
$\varphi:V \to U$ is a smooth, bijective map 
with bounded first- and second-order derivatives,
$\varphi(0)=\ba_0$, and $\|\varphi(\v_1)-\varphi(\v_2)\| \geq \|\v_1-\v_2\|/2$
for all $\v_1,\v_2 \in V$.
We apply Lemma \ref{lemma:inversefunction} to the map $\hg=\hf \circ \varphi$.
To verify the second-derivative bounds for $\hg$,
note that for $\ba \in O$, letting $\hv(\ba)$
be the unit eigenvector where $\hSigma(\ba)\hv(\ba)=\hl(\ba)\hv(\ba)$, we have
\begin{align}
\partial_{a_r} \hl(\ba)&=\hv(\ba)'(\partial_{a_r} \hSigma(\ba))\hv(\ba),
\label{eq:derhl}\\
\partial_{a_r}\partial_{a_s} \hl(\ba)
&=(\partial_{a_s} \hv(\ba))'(\partial_{a_r} \hSigma(\ba))\hv(\ba)
+\hv(\ba)'(\partial_{a_r} \hSigma(\ba))(\partial_{a_s}\hv(\ba))\nonumber\\
&=2\hv(\ba)'(\partial_{a_r} \hSigma(\ba))
(\hl(\ba)\Id-\hSigma(\ba))^\dagger (\partial_{a_s} \hSigma(\ba))\hv(\ba),
\nonumber
\end{align}
where $(\hl(\ba)\Id-\hSigma(\ba))^\dagger$ is the Moore-Penrose pseudo-inverse.
Since $\hl(\ba)$ is separated from other eigenvalues of $\hSigma(\ba)$ by a
constant, $\|(\hl(\ba)\Id-\hSigma(\ba))^\dagger\|<C$.
Then Lemma \ref{lemma:continuityina} and the chain rule imply that on the above
good event, $\hg$ has
all second-order derivatives bounded on $V$. It remains to check the condition
$\|(\der \hg(0))\v\| \geq c\|\v\|$ for a constant $c>0$ and
all $\v \in \R^{k-1}$. Since $\der \hg(0)=\der \hf(\ba_0) \cdot
\der \varphi(0)$, and $\der \varphi(0)\v$ is orthogonal to $\ba_0$ with $\|\der
\varphi(0)\v\| \geq \|\v\|/2$, we must check
\begin{equation}\label{eq:dhfwcondition}
\|(\der \hf(\ba_0))\w\| \geq c\|\w\|
\end{equation}
for a constant $c>0$ and all $\w$ orthogonal to $\ba_0$, where $\der \hf$ is the
derivative of $\hf:O \to \R^{k-1}$.

For this, let $\lambda_0=\hl_0+O(n^{-1/2+\eps})$ be the root of
$0=\det T(\lambda,\ba_0)$, and let $\v_0 \in \ker T(\lambda_0,\ba_0)$.
As $\lambda_0$ is a simple root, the implicit
function theorem implies we may define
$\lambda(\ba)$ analytically on a neighborhood of $\ba_0$ such that
$\lambda(\ba_0)=\lambda_0$ and $0=\det T(\lambda(\ba),\ba)$.
As $T(\lambda(\ba),\ba)$ is analytic in $\ba$ and
0 is a simple eigenvalue of
this matrix at $\ba_0$, we may also define the null eigenvector
$\v(\ba)$ analytically on a neighborhood of $\ba_0$, so that $\v(\ba_0)=\v_0$,
$T(\lambda(\ba),\ba)\v(\ba)=0$, and $\|\v(\ba)\|^2=1$. We show in Lemma
\ref{lemma:dlambbound} below that on an event of probability $1-n^{-D}$, we
have
\begin{equation}\label{eq:derlambbound}
\|\der \lambda(\ba_0)-\der \hl(\ba_0)\|<n^{-1/2+\eps}.
\end{equation}

Assuming (\ref{eq:derlambbound}) holds, let us first show that the analogue
of (\ref{eq:dhfwcondition}) holds for the function
\[f(\ba)=t_+(\lambda(\ba),\ba).\]
Denote $m(\ba)=m_0(\lambda(\ba),\ba)$, $\bb(\ba)=m(\ba)\ba$, and
$s_+(\bb)=(s_2(\bb),\ldots,s_k(\bb))$ where $s$ is as in (\ref{eq:s}). Then
$m(\ba)f(\ba)=s_+(\bb(\ba))$.
Denote $\bb_0=\bb(\ba_0)$ and differentiate this with respect to
$\ba$ at $\ba_0$ to get
\[f(\ba_0)(\der m(\ba_0))'+m(\ba_0)\der f(\ba_0)=
\der s_+(\bb_0)\der \bb(\ba_0).\]
Hence for any $\w \in \R^k$,
\[\der f(\ba_0)\w=\frac{1}{m(\ba_0)}\Big(\der s_+(\bb_0)\der \bb(\ba_0)
\w-f(\ba_0)(\der m(\ba_0))'\w\Big).\]
Applying $\|f(\ba_0)\|=\|t_+(\lambda_0,\ba_0)\|<\eps_1$ from
(\ref{eq:tnearaxis}), and $\|\der m(\ba_0)\|<C$ and $c<|m(\ba_0)|<C$ from
the chain rule, (\ref{eq:derlambbound}), (\ref{eq:derhl}),
and Lemma \ref{lemma:continuityina}, we have
\begin{equation}\label{eq:derfwbound}
\|\der f(\ba_0)\w\| \geq c\|\der s_+(\bb_0)\der \bb(\ba_0)\w\|
-C\eps_1\|\w\|.
\end{equation}
To bound the first term on the right, recall (\ref{eq:TTheta}) and
multiply the condition $0=\v(\ba)'T(\lambda(\ba),\ba)\v(\ba)$ by $m(\ba)$ to get
\[0=\v(\ba)'\left(-\Id+\sum_{r=1}^k s_r(\bb(\ba))V_r\Theta_rV_r'\right)
\v(\ba).\]
Differentiate this with respect to $\ba$ at $\ba_0$, and set
$y_r=\v_0'V_r\Theta_rV_r'\v_0$ and $\y=(y_1,\ldots,y_k)$, to get
\[0=\sum_{r=1}^k \der s_r(\bb_0)\der \bb(\ba_0)
\cdot \v_0'V_r\Theta_rV_r'\v_0
=\y'\der s(\bb_0)\der \bb(\ba_0).\]
For any $\w \in \R^k$, letting $\y_+=(y_2,\ldots,y_k)$, this yields
\begin{align*}
\|\y_+\| \cdot \|\der s_+(\bb_0)\der \bb(\ba_0)\w\| &\geq
|\y_+'\der s_+(\bb_0)\der \bb(\ba_0)\w|\\
&=|y_1 \der s_1(\bb_0)\der \bb(\ba_0)\w|\\
&\geq |y_1| \cdot \|\der s(\bb_0)\der \bb(\ba_0)\w\|
-|y_1| \cdot \|\der s_+(\bb_0)\der \bb(\ba_0)\w\|.
\end{align*}
So
\begin{equation}\label{eq:dersplusbound}
\|\der s_+(\bb_0)\der \bb(\ba_0)\w\|
\geq \frac{|y_1| \cdot \|\der s(\bb_0)\der
\bb(\ba_0)\w\|}{|y_1|+\|\y_+\|}.
\end{equation}
Note that $|y_1|+\|\y_+\|<C$. Applying $\v_0'T(\lambda_0,\ba_0)\v_0=0$ to
(\ref{eq:TapproxSigma1}), we have also $|y_1-\theta|<C\eps_1$, so
$|y_1|>\theta-C\eps_1>c$ for sufficiently small $\eps_1>0$.
Finally, recall $\bb(\ba)=m(\ba)\ba$, so
$\der \bb(\ba_0)=m(\ba_0)\Id+\ba_0(\der m(\ba_0))'$.
If $\w$ is orthogonal to $\ba_0$, then
\[\|\der \bb(\ba_0)\w\|
=\|m(\ba_0)\w+\ba_0(\der m(\ba_0))'\w\|
\geq \|m(\ba_0)\w\| \geq c\|\w\|.\]
As $\|\bb_0\|<C$, Assumption \ref{assump:estimation} implies
$\|\der s(\bb_0)\v\| \geq c\|\v\|$ for any $\v \in \R^k$, so combining
these observations with (\ref{eq:dersplusbound}) and (\ref{eq:derfwbound})
yields finally $\|(\der f(\ba_0))\w\| \geq c\|\w\|$ for $\w$ orthogonal to
$\ba_0$.

To conclude the proof, recall $f(\ba)=t_+(\lambda(\ba),\ba)$ while
$\hf(\ba)=t_+(\hl(\ba),\ba)$. Applying (\ref{eq:derlambbound}),
Lemma \ref{lemma:continuityina}, and the chain
rule, we obtain $\|\der f(\ba_0)-\der \hf(\ba_0)\|<Cn^{-1/2+\eps}$. Hence
(\ref{eq:dhfwcondition}) holds, and we may apply
Lemma \ref{lemma:inversefunction} to the function $\hg=\hf \circ \varphi$.
This shows, for some constants $c,\tilde{\eps}_0,\tilde{\eps}_1>0$, that $\hf$
is injective on $\tilde{U}=\{\ba \in S^{k-1}:\|\ba-\ba_0\|<\tilde{\eps}_0\}$,
$\hf(\tilde{U})$ contains
$\{\bt \in \R^{k-1}:\|\bt-\hf(\ba_0)\|<\tilde{\eps}_1\}$,
and $\|\hf(\ba_1)-\hf(\ba_2)\|\geq c\|\ba_1-\ba_2\|$ for $\ba_1,\ba_2 \in
\tilde{U}$. Observe that if $\|\bt\|<\eps_1$, then
$\|\bt-\hf(\ba_0)\|<2\eps_1$ by (\ref{eq:tnearaxis}). Reducing $\eps_0$ and
$\eps_1$ to $\tilde{\eps}_0$ and $\tilde{\eps}_1/2$ concludes the proof.
\end{proof}

\begin{lemma}\label{lemma:dlambbound}
Let $\ba_0 \in S^{k-1}$, let $U \subset \R^k$ be a neighborhood of $\ba_0$,
and let $\lambda(\ba)$ and $\hl(\ba)$ be analytic functions on $U$ such that
$0=\det T(\lambda(\ba),\ba)$ and $\hl(\ba) \in \spec(\hSigma(\ba))$ for each
$\ba \in U$. Suppose $\lambda(\ba_0)-\hl(\ba_0) \prec n^{-1/2}$, and 
$\lambda(\ba_0)$ is separated from all other roots of $0=\det T(\lambda,\ba_0)$
by a constant $c>0$. Then
\[\|\der \lambda(\ba_0)-\der \hl(\ba_0)\| \prec n^{-1/2}.\]
\end{lemma}
\begin{proof}
Let $\lambda_0=\lambda(\ba_0)$ and $\hl_0=\hl(\ba_0)$.
Denote by $\hK(\lambda,\ba)$ and $K(\lambda,\ba)$ the functions (\ref{eq:hK})
and (\ref{eq:K}) for $F=F(\ba)$.
Let us first establish, for each $r=1,\ldots,k$,
\begin{equation}\label{eq:daKhKbound}
\|\partial_{a_r} K(\lambda_0,\ba_0)-\partial_{a_r} \hK(\lambda_0,\ba_0)\|
\prec n^{-1/2}.
\end{equation}
The proof is similar to that of (\ref{eq:hKbound}), and we will be brief.
For notational convenience, we omit all arguments $(\lambda_0,\ba_0)$ and denote
$\partial=\partial_{a_r}$. Recalling $G_M$ from (\ref{eq:GM}),
\begin{align}
\partial G_M&=(\partial F)XG_NX'F+FX(\partial G_N)X'F+FXG_NX'(\partial
F)-\partial F\nonumber\\
&=(\partial F)XG_NX'F-FXG_NX'(\partial F)XG_NX'F+FXG_NX'(\partial F)
-\partial F\nonumber\\
&=-(FXG_NX'-\Id)(\partial F)(XG_NX'F-\Id).\label{eq:dGM}
\end{align}
Denoting by $(\partial G_M)_{rs}$ the $(r,s)$ block, (\ref{eq:dGM}) and
Lemma \ref{lemma:Kbound} imply
$\|(\partial G_M)_{rs}/(\sigma_r\sigma_s)\|_\HS \prec n^{1/2}$, so Lemma
\ref{lemma:hansonwright} applied conditionally on $X$ yields
\begin{equation}\label{eq:dhK}
\left\|\partial \hK+\sum_{r=1}^k \Big(N^{-1}\Tr_r(\partial
G_M)\Big)\Gamma_r\right\| \prec n^{-1/2}.
\end{equation}

Now recall $XG_NX'=\Delta+m_0(\Id+m_0F)^{-1}$ from (\ref{eq:Delta}), so
$XG_NX'F-\Id=\Delta F-(\Id+m_0F)^{-1}$. Substituting this into (\ref{eq:dGM})
and applying Lemmas \ref{lemma:matrixlocallaw} and
\ref{lemma:secondorderapprox}, we obtain after some simplification
\begin{align*}
\sigma_r^{-2}\Tr_r (\partial G_M)&=-\sigma_r^{-2}
\Tr_r \big[(\Id+m_0F)^{-1}(\partial F)(\Id+m_0F)^{-1}\big]\\
&\hspace{0.5in}-N^{-1}(\partial_\lambda m_0)
\Tr\big[(\partial F)(\Id+m_0F)^{-1}\big]
\sigma_r^{-2}\Tr_r\big[F^2(\Id+m_0F)^{-2}\big]+\O(n^{1/2}).
\end{align*}
Applying (\ref{eq:derFJ}) and (\ref{eq:darm0}),
\[(N\sigma_r^2)^{-1}\Tr_r (\partial G_M)=-(N\sigma_r^2)^{-1}
\Tr_r \Big[\partial\Big(F(\Id+m_0F)^{-1}\Big)\Big]
+\O(n^{-1/2})=-\partial t_r+\O(n^{-1/2}).\]
Applying this to (\ref{eq:dhK}) and recalling the definition (\ref{eq:K}) of
$K$, we obtain (\ref{eq:daKhKbound}) as desired.

Note that (\ref{eq:hKbound}), (\ref{eq:dhKbound}), $\lambda_0-\hl_0 \prec
n^{-1/2}$, and Lemma \ref{lemma:Kbound} imply
\begin{align}
\|K(\lambda_0,\ba_0)-\hK(\hl_0,\ba_0)\| \prec n^{-1/2},\label{eq:K0hK0}\\
\|\partial_\lambda K(\lambda_0,\ba_0)-\partial_\lambda \hK(\hl_0,\ba_0)\|
\prec n^{-1/2}.\label{eq:dlK0hK0}
\end{align}
From (\ref{eq:dhK}), we verify that on the high-probability event
where $\spec(X'F(\ba_0)X) \subset \supp(\mu_0(\ba_0))_{\delta/2}$,
$\|X_r\|<C$, and $\|\oX_r\|<C$ for all $r=1,\ldots,k$, we have
\[\sup_{\lambda \in \R \setminus \supp(\mu_0(\ba_0))_\delta}
\|\partial_{a_r} \hK(\lambda,\ba_0)\|<C, \qquad
\sup_{\lambda \in \R \setminus \supp(\mu_0(\ba_0))_\delta}
\|\partial_\lambda \partial_{a_r} \hK(\lambda,\ba_0)\|<C.\]
Then this and (\ref{eq:daKhKbound}) yield similarly
\begin{equation}
\|\partial_{a_r} K(\lambda_0,\ba_0)-\partial_{a_r} \hK(\hl_0,\ba_0)\|
\prec n^{-1/2}.\label{eq:daK0hK0}
\end{equation}

Let $\mu_1(\lambda,\ba) \leq \ldots \leq \mu_L(\lambda,\ba)$ and
$\hmu_1(\lambda,\ba) \leq \ldots \leq \hmu_L(\lambda,\ba)$ be the eigenvalues
of $K(\lambda,\ba)$ and $\hK(\lambda,\ba)$. Then (\ref{eq:K0hK0}) implies
$\mu_\ell(\lambda_0,\ba_0)-\hmu_\ell(\hl_0,\ba_0) \prec n^{-1/2}$
for each $\ell$. Note that $0=\det K(\lambda(\ba),\ba)$ and $0=\det
\hK(\hl(\ba),\ba)$ for all $\ba \in U$. In particular,
$\mu_\ell(\lambda_0,\ba_0)=0$ for some $\ell$. As $\lambda_0$ is separated
from other roots of $0=\det T(\lambda,\ba_0)$ by $c>0$, Proposition
\ref{prop:Tproperties}(c) implies 0 is separated from other eigenvalues of
$T(\lambda_0,\ba_0)$ by $c$. Assuming $U$ is sufficiently small, this implies
that $\mu_\ell(\lambda(\ba),\ba)=0$ and
$\hmu_\ell(\lambda(\ba),\ba)=0$ for the same $\ell$ and
all $\ba \in U$. Differentiating these identities in $\ba$ at $\ba_0$, we obtain
\begin{equation}\label{eq:derlamb}
\der \lambda(\ba_0)=-(\partial_\lambda \mu_\ell(\lambda_0,\ba_0))^{-1}
\partial_\ba \mu_\ell(\lambda_0,\ba_0), \qquad
\der \hl(\ba_0)=-(\partial_\lambda \hmu_\ell(\hl_0,\ba_0))^{-1}
\partial_\ba \hmu_\ell(\hl_0,\ba_0).
\end{equation}
Letting $\v_0 \in \ker K(\lambda_0,\ba_0)$ and $\hv_0 \in \ker \hK(\hl_0,\ba_0)$
be the unit eigenvectors, we have for both $\partial=\partial_\lambda$ and
$\partial=\partial_{a_r}$ that
\[\partial \mu_\ell(\lambda_0,\ba_0)=\v_0'\partial K(\lambda_0,\ba_0) \v_0,
\qquad \partial \hmu_\ell(\hl_0,\ba_0)=\hv_0'\partial \hK(\hl_0,\ba_0) \hv_0.\]
The Davis-Kahan theorem yields
$\|\v_0-\hv_0\| \prec n^{-1/2}$, so (\ref{eq:dlK0hK0}), (\ref{eq:daK0hK0}),
and the bounds $\|\partial K\|,\|\partial \hK\| \prec 1$ imply
\[\partial \mu_\ell(\lambda_0,\ba_0)-\partial
\hmu_\ell(\hl_0,\ba_0) \prec n^{-1/2}, \qquad
\partial \mu_\ell(\lambda_0,\ba_0) \prec 1, \qquad
\partial \hmu_\ell(\hl_0,\ba_0) \prec 1.\]
Applying this and $\partial_\lambda \mu_\ell(\lambda_0,\ba_0) \leq -1$
to (\ref{eq:derlamb}),
we obtain $\|\der\lambda(\ba_0)-\der \hl(\ba_0)\| \prec n^{-1/2}$.
\end{proof}

We now conclude the proofs of the remaining two claims for Theorem
\ref{thm:estimation}.

\begin{proof}[Proof of Claim 2]
Suppose $\mu=\theta+\sigma_1^2$ is a spike eigenvalue of $\Sigma_1$.
Each estimated $\hat{\mu}$ where $|\hat{\mu}-\mu|<\eps$ corresponds to
a pair $(\hl,\ba)$ where $\ba \in S^{k-1}$,
$\hl \in \spec(\hSigma(\ba)) \cap \I_\delta(\ba)$, and
\[|\hl/t_1(\hl,\ba)-\mu|<\eps, \qquad
t_2(\hl,\ba)=\ldots=t_k(\hl,\ba)=0.\]
Then $\hl=-1/m_0(\hl,\ba)+\sigma_1^2t_1(\hl,\ba)$ by (\ref{eq:MP}). Applying
this and $|t_1(\hl,\ba)|<C$ to the above, $(\hl,\ba)$ satisfies
\begin{equation}\label{eq:estimationcondition}
\left|-\frac{1}{m_0(\hl,\ba)}-t_1(\hl,\ba)\theta\right|<C\eps,
\qquad t_2(\hl,\ba)=\ldots=t_k(\hl,\ba)=0.
\end{equation}

By Lemma \ref{lemma:tplusinjective}, there exist constants $\eps_0,\eps_1>0$
such that if $C\eps<\eps_1$, then with probability $1-n^{-D}$,
(\ref{eq:estimationcondition}) cannot
hold for two different pairs $(\hl_0,\ba_0)$ and $(\hl_1,\ba_1)$ with
$\|\ba_1-\ba_0\|<\eps_0$. On the other hand, on the event where the conclusion
of Theorem \ref{thm:spikemapping} holds for all $\ba \in S^{k-1}$, we have
$-C<m_0(\hl_0,\ba_0)<-c$ and $-C<m_0(\hl_1,\ba_1)<-c$ for constants $C,c>0$
by Lemma \ref{lemma:continuityina}. On this event, if
(\ref{eq:estimationcondition}) holds 
for $(\hl_0,\ba_0)$ and $(\hl_1,\ba_1)$ with $\|\ba_1-\ba_0\| \geq \eps_0$,
then $\|m_0(\hl_0,\ba_0)\ba_0-m_0(\hl_1,\ba_1)\ba_1\|>c\eps_0$
for some $c>0$ because both $\ba_0$ and $\ba_1$ belong to the sphere. Recalling
$s:\R^k \to \R^k$ from (\ref{eq:s}), note that
$m_0(\hl,\ba)t(\hl,\ba)=s(m_0(\hl,\ba)\ba)$.
Assumption \ref{assump:estimation} then implies
$\|m_0(\hl_0,\ba_0)t(\hl_0,\ba_0)-m_0(\hl_1,\ba_1)t(\hl_1,\ba_1)\|>c\eps_0$
for a different $c>0$. But the first condition of (\ref{eq:estimationcondition})
implies $|m_0(\hl_0,\ba_0)t(\hl_0,\ba_0)+1/\theta|<C\eps$ and similarly for
$(\hl_1,\ba_1)$, for some $C>0$. This is a contradiction for $\eps$
sufficiently small, so with probability $1-n^{-D}$,
at most one pair $(\hl,\ba)$ satisfies (\ref{eq:estimationcondition}).
\end{proof}

\begin{proof}[Proof of Claim 3]
We first show that for a constant $c_0>0$ (independent of $\bar{C}$)
and any value $\theta>c_0$,
there exist $\ba_0 \in S^{k-1}$ and $\lambda_0 \in \I_\delta(\ba_0)$ where
\begin{equation}\label{eq:theoreticalsolution}
-\frac{1}{m_0(\lambda_0,\ba_0)}-t_1(\lambda_0,\ba_0)\theta=0, \qquad
t_2(\lambda_0,\ba_0)=\ldots=t_k(\lambda_0,\ba_0)=0.
\end{equation}
Indeed, Proposition \ref{prop:boundedsupportrestate} shows $\supp(\mu_0(\ba)) \in
[-C_1,C_1]$ for a constant $C_1>0$ and all $\ba \in S^{k-1}$. Then for each $\ba
\in S^{k-1}$, at the left endpoint $\lambda_+$ of $\I_\delta(\ba)$ we have
\begin{equation}\label{eq:m0lambdaplus}
m_0(\lambda_+,\ba)=\int \frac{1}{x-\lambda_+}\,\mu_0(\ba)(dx) \leq 
-(2C_1+\delta)^{-1},
\end{equation}
and $m_0(\lambda,\ba)$ increases to 0 as $\lambda$ increases from $\lambda_+$ to
$\infty$. We apply Lemma \ref{lemma:inversefunction} to the map
$s$ from (\ref{eq:s}): Note that
$s(0)=0$, and Assumption \ref{assump:estimation} guarantees
$\|(\der s(0))\v\| \geq c\|\v\|$. Setting $U=\{\bb:\|\bb\|<\eps\}$ for a
sufficiently small constant $\eps>0$, we have
$|\partial_{b_i}\partial_{b_j} s_r(\bb)|<C$ for all $i,j,r$ and $\bb \in U$.
We may take $\eps<(2C_1+\delta)^{-1}$. Then
applying Lemma \ref{lemma:inversefunction}, for some constant
$c_0>0$ and any $\theta>c_0$, there
exists $\bb_0 \in U$ such that $s(\bb_0)=(-1/\theta,0,\ldots,0)$.
Now let $\ba_0=-\bb_0/\|\bb_0\| \in S^{k-1}$. As $\|\bb_0\|<(2C_1+\delta)^{-1}$,
(\ref{eq:m0lambdaplus}) implies there exists $\lambda \in
\I_\delta(\ba_0)$ with $m_0(\lambda,\ba_0)=-\|\bb_0\|$, and hence
$\bb_0=m_0(\lambda_0,\ba_0)\ba_0$. Noting that
$m_0(\lambda_0,\ba_0)t(\lambda_0,\ba_0)=s(\bb_0)=(-1/\theta,0,\ldots,0)$,
this yields (\ref{eq:theoreticalsolution}).

Now let $\mu=\theta+\sigma_1^2$ be a spike eigenvalue of $\Sigma_1$, where
$\theta>c_0$, and let $(\lambda_0,\ba_0)$ be as above.
By Theorem \ref{thm:spikemapping}, there exists $\hl_0 \in
\spec(\hSigma(\ba_0)) \cap \I_\delta(\ba_0)$ with $\hl_0-\lambda_0 \prec
n^{-1/2}$. Applying Lemma \ref{lemma:continuityina},
\[-\frac{1}{m_0(\hl_0,\ba_0)}-t_1(\hl_0,\ba_0)\theta \prec n^{-1/2},
\qquad t_r(\hl_0,\ba_0) \prec n^{-1/2} \text{ for all }
r=2,\ldots,k.\]
Lemma \ref{lemma:tplusinjective} implies there exist $\ba \in
S^{k-1}$ and $\hl \in \spec(\hSigma(\ba))$ with $t_+(\hl,\ba)=0$
and $c\|\ba-\ba_0\| \leq \|t_+(\hl_0,\ba_0)\|$. The latter
condition implies $\|\ba-\ba_0\| \prec n^{-1/2}$, so also
$\|\hSigma(\ba)-\hSigma(\ba_0)\| \prec n^{-1/2}$,
$\hl-\hl_0 \prec n^{-1/2}$, and $\hl \in \I_\delta(\ba)$ with probability
$1-n^{-D}$. Applying Lemma \ref{lemma:continuityina} again, we obtain
\[-\frac{1}{m_0(\hl,\ba)}-t_1(\hl,\ba)\theta \prec n^{-1/2},
\qquad t_2(\hl,\ba)=\ldots=t_k(\hl,\ba)=0.\]
This and (\ref{eq:MP}) imply $\hl/t_1(\hl,\ba)-\mu \prec n^{-1/2}$,
so with probability $1-n^{-D}$, there is an estimated eigenvalue $\hmu$ with
$|\hmu-\mu|<n^{-1/2+\eps}$.
\end{proof}

\appendix

\section{Resolvent approximations}\label{sec:resolventapprox}
We prove in this appendix Lemmas \ref{lemma:matrixlocallaw}
and \ref{lemma:secondorderapprox}.
Both statements rely on a ``fluctuation averaging'' idea, similar to that in
\cite{erdosbernoulli,erdosyauyin,erdosER,erdoslocalSC},
to control a weighted average of weakly dependent random variables. We introduce
a variant of this idea which controls the size of the weighted average
by the squared-sum of the weights, rather than the size of the largest weight,
and also develop it for sums over double-indexed and quadruple-indexed arrays.
We present this abstract result in Section
\ref{subsec:fluctuationavg}, and then apply it to combinations of resolvent
entries and their products in the remainder of the section.

\subsection{Fluctuation averaging}\label{subsec:fluctuationavg}
Let $\x_1,\ldots,\x_n$ be independent random variables in some
probability space. For $\cY$ a scalar-valued function of $\x_1,\ldots,\x_n$,
denote by $\E_i[\cY]$ its expectation with respect to only $\x_i$, i.e.\
\[\E_i[\cY]=\E[\cY \mid \x_1,\ldots,\x_{i-1},\x_{i+1},\ldots,\x_n].\]
Define
\[\Q_i[\cY]=\cY-\E_i[\cY].\]
Note that the operators $\{\E_i,\Q_i:i=1,\ldots,n\}$ all commute. For
$S \subset \{1,\ldots,n\}$, define
\[\E_S=\prod_{i \in S} \E_i, \qquad \Q_S=\prod_{i \in S} \Q_i\]
where the products denote operator composition.

We will consider subsets $S \subset \{1,\ldots,n\}$ of size at most a constant
$\ell>0$. For quantities $\xi$ and $\zeta$ possibly depending on $S$, we write
\[\xi \prec_\ell \zeta\]
to mean $\P[|\xi|>n^{\eps}|\zeta|]<n^{-D}$ for all $|S| \leq \ell$ and all
$n \geq n_0(\ell,\eps,D)$, where the constant $n_0$ is allowed to depend
on $\ell$ (in addition to $\eps$ and $D$).

We will require $\cY$ to satisfy the moment condition of the following lemma.
\begin{lemma}\label{lemma:expectation}
For constants $\tau,C_1,C_2,\ldots>0$, suppose $\cY \prec n^{-\tau}$ and
$\E[|\cY|^\ell] \leq n^{C_\ell}$ for each integer $\ell>0$.
Then for any sub-$\sigma$-algebra $\mathcal{G}$, $\E[\cY \mid
\mathcal{G}] \prec n^{-\tau}$.
\end{lemma}
\begin{proof}
See \cite[Lemma 3.2]{fanjohnstoneedges}.
\end{proof}

A variable $\cY_i$ is centered with respect to $\x_i$ if $\E_i[\cY_i]=0$.
If it is independent of $\x_j$, then $\Q_j[\cY_i]=0$. We quantify weak
dependence of $\cY_i$ on $\x_j$ by requiring $\Q_j[\cY_i]$ to be typically
smaller than $\cY_i$ by a factor of $n^{-1/2}$. The following is an abstract
fluctuation averaging result for variables that are weakly dependent in this
sense.
\begin{lemma}\label{lemma:fluctuationavg}
Let $\tau,C_1,C_2,\ldots>0$ be fixed constants, and let each $\cY_* \in
\{\cY_i,\cY_{ij},\cY_{ijkl}\}$ below be a
scalar-valued function of $\x_1,\ldots,\x_n$ that satisfies $\cY_*
\prec n^{-\tau}$ and $\E[|\cY_*|^\ell] \leq n^{C_\ell}$ for each $\ell>0$.
\begin{enumerate}[(a)]
\item Suppose $(\cY_i:i=1,\ldots,n)$ satisfy
$\E_i[\cY_i]=0$ and, for all $S \subset \{1,\ldots,n\}$ with $i
\notin S$ and $|S| \leq \ell$,
\begin{equation}\label{eq:QSYi}
Q_S[\cY_i] \prec_\ell n^{-\tau-|S|/2}.
\end{equation}
Then for any deterministic $(u_i \in \C:i=1,\ldots,n)$,
\[\sum_i u_i\cY_i \prec n^{-\tau}\left(\sum_i |u_i|^2\right)^{1/2}.\]
\item Suppose $(\cY_{ij}:i,j=1,\ldots,n,\;i \neq j)$ satisfy
$\E_i[\cY_{ij}]=\E_j[\cY_{ij}]=0$ and, for all
$S \subset \{1,\ldots,n\}$ with $i,j \notin S$ and $|S| \leq \ell$,
\[Q_S[\cY_{ij}] \prec_\ell n^{-\tau-|S|/2}.\]
Then for any deterministic $(u_{ij} \in \C:i,j=1,\ldots,n,\;i \neq j)$,
\[\sum_{i \neq j} u_{ij}\cY_{ij} \prec n^{-\tau}\left(\sum_{i \neq j}
|u_{ij}|^2\right)^{1/2}.\]
\item Suppose $(\cY_{ijkl}:i,j,k,l=1,\ldots,n \text{ all distinct})$ satisfy
$\E_i[\cY_{ijkl}]=\E_j[\cY_{ijkl}]=\E_k[\cY_{ijkl}]=\E_l[\cY_{ijkl}]=0$ and,
for all $S \subset
\{1,\ldots,n\}$ with $i,j,k,l \notin S$ and $|S| \leq \ell$,
\[Q_S[\cY_{ijkl}] \prec_\ell n^{-\tau-|S|/2}.\]
Then for any deterministic $(u_{ij} \in \C:i,j=1,\ldots,n,\;i \neq j)$ and
$(v_{kl} \in \C:k,l=1,\ldots,n,\;k \neq l)$,
\[\mathop{\sum_{i,j,k,l}}_{\text{all distinct}} u_{ij}v_{kl}\cY_{ijkl}
\prec n^{-\tau}\left(\sum_{i \neq j} |u_{ij}|^2\right)^{1/2}
\left(\sum_{k \neq l} |v_{kl}|^2\right)^{1/2}.\]
\end{enumerate}
\end{lemma}
\begin{proof}
The proof is similar to the ``Alternative proof of Theorem 4.7'' presented in
\cite[Appendix B]{erdoslocalSC}.
Fix any constants $\eps,D>0$, and choose an even integer $\ell$ such that
$(\ell-1)\eps>D$. For part (a), let us normalize so that $\sum_i |u_i|^2=1$.
We apply the moment method and bound the quantity
\begin{equation}\label{eq:fluctuationmoment}
\E\left[\left|\sum_i u_i \cY_i\right|^\ell\right]
=\sum_\i u_\i\E[\cY_\i],
\end{equation}
where we denote as shorthand
\[\i=(i_1,\ldots,i_\ell),\qquad
\sum_\i=\sum_{i_1,\ldots,i_\ell=1}^n,
\qquad u_\i=\prod_{a=1}^{\ell/2} u_{i_a} \prod_{a=\ell/2+1}^{\ell}
\overline{u_{i_a}},
\qquad \cY_\i=\prod_{a=1}^{\ell/2} \cY_{i_a}
\prod_{a=\ell/2+1}^\ell \overline{\cY_{i_a}}.\]
Fix $\i$, and let $\cT=\cT(\i) \subset \{1,\ldots,n\}$ be the indices
that appear exactly once in $\i$. Applying the identity
\[\cY=\left(\prod_{j \in \cT}(\E_j+\Q_j)\right)\cY
=\sum_{S \subseteq \cT} \E_{\cT \setminus S}\Q_S\cY\]
to each $\cY_{i_a}$ and expanding the product of the sums,
\[\cY_{\i}=\sum_{S_1,\ldots,S_\ell \subseteq \cT} \cY(S_1,\ldots,S_\ell),
\qquad
\cY(S_1,\ldots,S_\ell)=\prod_{a=1}^\ell
\E_{\cT \setminus S_a}\Q_{S_a} \tilde{\cY}_{i_a},\qquad
\tilde{\cY}_{i_a}=\begin{cases} \cY_{i_a} & a \leq \ell/2 \\
\overline{\cY_{i_a}} & a \geq \ell/2+1.\end{cases}\]
Note that $Q_{i_a}\tilde{\cY}_{i_a}=\tilde{\cY}_{i_a}$, so
(\ref{eq:QSYi}) and Lemma \ref{lemma:expectation} yield
\[\E_{\cT \setminus S_a}\Q_{S_a} \tilde{\cY}_{i_a} \prec_\ell
\Q_{S_a} \tilde{\cY}_{i_a} \prec_\ell
n^{-\tau-(|S_a \setminus \{i_a\}|)/2}.\]
Then, taking the product over all $a=1,\ldots,\ell$ and applying
$\sum_a |S_a \setminus \{i_a\}| \geq -|\cT|+\sum_a |S_a|$,
\begin{equation}\label{eq:YS1Slbound}
\cY(S_1,\ldots,S_\ell) \prec_\ell n^{-\ell \tau+\frac{|\cT|}{2}-\sum_{a=1}^\ell
\frac{|S_a|}{2}}.
\end{equation}
Next, note that if $i_a \in \cT$, then $\Q_{S_a}\cY_{i_a}=0$ and
$\cY(S_1,\ldots,S_\ell)=0$ unless $i_a \in S_a$.
Furthermore, if $i_a \in S_a$ but $i_a \notin S_b$ for all $b \neq a$, then
$\E_{\cT \setminus S_b}\Q_{S_b}\tilde{\cY}_{i_b}$ does not depend on
$\x_{i_a}$ for all $b \neq a$, so we have
\[\E_{i_a}[\cY(S_1,\ldots,S_\ell)]=\E_{i_a}\left[\E_{\cT \setminus S_a}
\Q_{S_a}\tilde{\cY}_{i_a}\right] \prod_{b:b \neq a} (\E_{\cT \setminus
S_b}\Q_{S_b}\tilde{\cY}_{i_b})=0.\]
Thus if $\E[\cY(S_1,\ldots,S_\ell)] \neq 0$,
then each $i_a \in \cT$ must belong to both $S_a$ and at least one other
$S_b$, so $\sum_a |S_a| \geq 2|\cT|$. Then
(\ref{eq:YS1Slbound}) and Lemma \ref{lemma:expectation} yield
$\E[\cY(S_1,\ldots,S_\ell)] \prec_\ell n^{-\ell\tau-|\cT|/2}$.
As the number of choices of subsets
$S_1,\ldots,S_\ell \subseteq \cT$ is an
$\ell$-dependent constant, we arrive at
\[\E[\cY_\i] \prec_\ell n^{-\ell\tau-|\cT|/2}.\]

Returning to (\ref{eq:fluctuationmoment}), we obtain
\[\E\left[\left|\sum_i u_i\cY_i\right|^\ell\right]
\prec_\ell \sum_{t=0}^\ell n^{-\ell\tau-t/2} \sum_{\i:|\cT(\i)|=t} |u_{\i}|.\]
We may separate the sum over $\{\i:|\cT(\i)|=t\}$ as a sum first over
groupings of the indices $i_1,\ldots,i_\ell$ that coincide, followed by a sum
over distinct values of those indices. Under our normalization,
\[\sum_i |u_i| \leq n^{1/2}, \qquad \sum_i |u_i|^k \leq 1
\text{ for all } k \geq 2.\]
Furthermore, the number of groupings is an $\ell$-dependent constant, so
\[\sum_{\i:|\cT(\i)|=t} |u_{\i}| \prec_\ell n^{t/2},\qquad
\E\left[\left|\sum_i u_i\cY_i\right|^\ell\right] \prec_\ell n^{-\ell\tau}.\]
The latter statement 
means that the expectation is (deterministically) at most $n^{-\ell\tau+\eps}$
for all $n \geq n_0(\ell,\eps)$. Then, as $\ell$ depends only on $\eps$ and
$D$, and we chose $(\ell-1)\eps>D$,
Markov's inequality yields for all $n \geq n_0(\eps,D)$
\[\P\left[\left|\sum_i u_i\cY_i\right|>n^{-\tau+\eps}\right]
\leq \frac{n^{-\ell\tau+\eps}}{(n^{-\tau+\eps})^\ell}<n^{-D}.\]
As $\eps,D>0$ were arbitrary, this concludes the proof of part (a).

Parts (b) and (c) are similar, except for an additional combinatorial argument
encapsulated in Lemma \ref{lemma:combinatorial} below: For (b), normalize so
that $\sum_{i \neq j} |u_{ij}|^2=1$ and write
\[\E\left[\left|\sum_{i \neq j} u_{ij} \cY_{ij}\right|^\ell\right]
=\sum_{\i,\j} u_{\i,\j}\E[\cY_{\i,\j}]\]
where
\[\sum_{\i,\j}=\sum_{i_1 \neq j_1} \ldots \sum_{i_\ell \neq j_\ell},
\qquad u_{\i,\j}=\prod_{a=1}^{\ell/2} u_{i_aj_a}
\prod_{a=\ell/2+1}^\ell \overline{u_{i_aj_a}},
\qquad \cY_{\i,\j}=\prod_{a=1}^{\ell/2} \cY_{i_aj_a}
\prod_{a=\ell/2+1}^\ell \overline{\cY_{i_aj_a}}.\]
Fixing $\i,\j$ and letting $\cT=\cT(\i,\j)$ be the 
indices that appear exactly once in the combined vector $(\i,\j)$, the same
argument yields $\E[\cY_{\i,\j}] \prec_\ell n^{-\ell\tau-|\cT|/2}$.
Applying Lemma \ref{lemma:combinatorial} with $B_a[i,j]=|u_{ij}|$
and $B_a[i,i]=0$ for all $a=1,\ldots,\ell$ and $i \neq j$, we get
\[\sum_{\i,\j:|\cT(\i,\j)|=t} |u_{\i,\j}| \prec_\ell n^{t/2},\]
which concludes the proof in the same way as part (a). For part (c),
normalize so that $\sum_{i \neq j} |u_{ij}|^2=\sum_{k \neq l} |v_{kl}|^2=1$,
and write analogously
\[\E\left[\left|\mathop{\sum_{i,j,k,l}}_{\text{distinct}}
u_{ij}v_{kl} \cY_{ijkl}\right|^\ell\right]
=\sum_{\i,\j,\mathbf{k},\mathbf{l}} u_{\i,\j}v_{\mathbf{k},\mathbf{l}}
\E[\cY_{\i,\j,\mathbf{k},\mathbf{l}}].\]
Letting $\cT$ be the indices appearing exactly
once in the combined vector $(\i,\j,\mathbf{k},\mathbf{l})$,
the bound $\E[\cY_{\i,\j,\mathbf{k},\mathbf{l}}] \prec
n^{-\ell\tau-|\cT|/2}$ follows as before, and the bound
\[\sum_{\i,\j,\mathbf{k},\mathbf{l}:|\cT|=t}
|u_{\i,\j}v_{\mathbf{k},\mathbf{l}}| \prec n^{t/2}\]
follows from Lemma \ref{lemma:combinatorial} applied with
$B_a[i,j]=|u_{ij}|$ and $B_a[i,i]=0$ for $a=1,\ldots,\ell$ and
$B_a[i,j]=|v_{ij}|$ and $B_a[i,i]=0$ for $a=\ell+1,\ldots,2\ell$.
\end{proof}

\begin{lemma}\label{lemma:combinatorial}
Fix $\ell \geq 1$. For each $a=1,\ldots,\ell$, let $B_a=(B_a[i,j]) \in \R^{n
\times n}$ satisfy
\[B_a[i,j] \geq 0, \qquad B_a[i,i]=0, \qquad \|B_a\|_\HS \leq 1\]
for all $i,j \in \{1,\ldots,n\}$.
For $(\i,\j)=(i_1,\ldots,i_\ell,j_1,\ldots,j_\ell) \in
\{1,\ldots,n\}^{2\ell}$, denote by $s(\i,\j)$ the number of elements of
$\{1,\ldots,n\}$ that appear exactly once in $(\i,\j)$. Then for a
constant $C_\ell>0$ and all $s \in \{0,\ldots,2\ell\}$,
\[\mathop{\sum_{\i,\j \in \{1,\ldots,n\}^{2\ell}}}_{s(\i,\j)=s}
\prod_{a=1}^\ell B_a[i_a,j_a] \leq C_\ell n^{s/2}.\]
\end{lemma}
\begin{proof}
Define an equivalence relation $(\i,\j) \sim (\i',\j')$ if a permutation of
$\{1,\ldots,n\}$ maps $(\i,\j)$ to $(\i',\j')$. For an equivalence class $E$,
define $s(E)=s(\i,\j)$ for any $(\i,\j) \in E$. Let $\cE$ be the set of
equivalence classes where $i_a \neq j_a$ for all $a=1,\ldots,\ell$.
Then, as $B_a$ has zero diagonal,
\[\mathop{\sum_{\i,\j \in \{1,\ldots,n\}^{2\ell}}}_{s(\i,\j)=s}
\prod_{a=1}^\ell B_a[i_a,j_a]
=\sum_{E \in \cE:\,s(E)=s} B(E),\qquad
B(E)=\sum_{(\i,\j) \in E} \prod_{a=1}^\ell B_a[i_a,j_a].\]

For $E \in \cE$, if $(\i,\j) \in E$ has $m$ distinct values, then
let $(\u,\v)=(u_1,\ldots,u_\ell,v_1,\ldots,v_\ell) \in \{1,\ldots,m\}^{2\ell}$
be the canonical element of $E$ where these values are
$\{1,\ldots,m\}$ in sequential order. Identify $E$ with the
directed multi-graph on the vertex set $\{1,\ldots,m\}$ with the $\ell$ edges
$\{(u_a,v_a):a=1,\ldots,\ell\}$. Writing the summation defining $B(E)$ as a
summation over the $m$ possible distinct index values,
\[B(E)=\mathop{\sum_{i(1),\ldots,i(m)=1}}_{\text{distinct}}^n \;
\prod_{a=1}^\ell B_a[i(u_a),i(v_a)].\]
As $B_a$ has nonnegative entries, we may drop the distinctness condition in the
sum to obtain the upper bound
\[B(E) \leq U(E)=\sum_{i(1),\ldots,i(m)=1}^n  \; \prod_{a=1}^\ell
B_a[i(u_a),i(v_a)].\]
The number of equivalence classes in $\cE$ is a constant $C_\ell>0$, so it
suffices to show for all $E \in \cE$
\begin{equation}\label{eq:UEbound}
U(E) \leq n^{s(E)/2}.
\end{equation}

Let the degree of a vertex be the total number of its in-edges and out-edges.
Then $s(E)$ is the number of degree-1 vertices.
Consider first a class $E$ where every
vertex has even degree. Then each connected component of
the multi-graph may be traversed as an Eulerian cycle, where each edge is
traversed exactly once in either its forward or backward direction.
Letting $\cC$ be the set of connected components, this yields
\[U(E) \leq \prod_{C \in \cC} \Tr\left(\prod_{a:(u_a,v_a) \in C}
\tilde{B}_a\right),\]
where the second product over edges is taken in the order of the Eulerian
cycle of $C$, and $\tilde{B}_a=B_a$ if $(u_a,v_a)$ is traversed in the forward
direction and
$\tilde{B}_a=B_a'$ if it is traversed in the backward direction. (This holds
because each term of $U(E)$ appears on the right upon expanding the traces,
and the extra terms on the right are nonnegative.)
Note that for any $k \geq 2$ and any matrices $A_1,\ldots,A_k$,
\[\Tr A_1\ldots A_k \leq \|A_1\|_\HS \cdot \|A_2\ldots A_k\|_\HS
\leq \|A_1\|_\HS \|A_2\|_\HS \ldots \|A_k\|_\HS.\]
The multi-graph has no self-loops, so each $C \in \cC$ has at least 2 edges.
Applying this and $\|\tilde{B}_a\|_\HS \leq 1$ for each $a$, we obtain $U(E)
\leq 1$.

Next, consider $E$ where every vertex has degree at least 2, and there
is some vertex $u$ of odd degree. Then there is another
vertex $v$ of odd degree in the same connected component as $u$, because the sum
of vertex degrees in a connected component is even. We may pick $v$ such that
there is a path $P$ from $u$ to $v$, traversing edges either forwards or
backwards, where every intermediary vertex between $u$ and $v$ has degree 2.
(Otherwise, replace $v$ by the first such vertex along any path from $u$ to
$v$.) Let us remove the path by summing over the intermediary vertex labels:
For notational convenience, suppose the intermediary vertices are
$p+1,\ldots,m$. Then, since only edges in the path $P$ touch the vertices
$p+1,\ldots,m$, we have
\[U(E)=\sum_{i(1),\ldots,i(p)} \prod_{a:(u_a,v_a) \notin P}
B_a[i(u_a),i(v_a)]
\left(\sum_{i(p+1),\ldots,i(m)} \prod_{a:(u_a,v_a) \in P}
B_a[i(u_a),i(v_a)]\right).\]
Note that the quantity in parentheses is element $[u,v]$ of the matrix
\[\prod_{a:(u_a,v_a) \in P} \tilde{B}_a,\]
where the product is taken in the order of traversal of $P$, and
$\tilde{B}_a=B_a$ or $B_a'$ depending on the direction of traversal of edge $a$.
As the Hilbert-Schmidt norm of this product is at most 1, we obtain
\[U(E) \leq U(E'), \qquad
U(E')=\sum_{i(1),\ldots,i(p)} \prod_{a:(u_a,v_a) \notin P} B_a[i(u_a),i(v_a)].\]
Here $E'$ corresponds to the multi-graph with
path $P$ and intermediary vertices $p+1,\ldots,m$ removed.
Each vertex of this new multi-graph still has degree at least
2---hence we may iteratively apply this procedure until the resulting graph has
no vertices of odd degree. Then $U(E) \leq 1$ follows from the first case
above.

Finally, consider $E$ where $s(E)=s>0$. For notational convenience, let
$1,\ldots,s$ be the vertices of degree 1. Then, applying the general inequality
$\sum_{i=1}^N w_i \leq N^{1/2}(\sum_{i=1}^N w_i^2)^{1/2}$ with $N=n^s$, we have
\[U(E) \leq n^{s/2}U(E')^{1/2},\qquad U(E')=
\sum_{i_1,\ldots,i_s=1}^n \left(\sum_{i_{s+1},\ldots,i_m=1}^n
\prod_{a=1}^\ell B_a[i(u_a),i(v_a)]\right)^2.\]
The quantity $U(E')$ corresponds to a multi-graph with $s+2(m-s)$ vertices and
$2\ell$ edges, where each vertex $s+1,\ldots,m$ is duplicated into two
copies. Each of the original vertices $1,\ldots,s$ now has degree 2, and each
copy of $s+1,\ldots,m$ continues to have degree at least 2. Then $U(E') \leq 1$
from the above, so $U(E) \leq n^{s/2}$. This establishes (\ref{eq:UEbound}) in
all cases, concluding the proof.
\end{proof}

\subsection{Preliminaries}

We first reduce the proofs of Lemmas \ref{lemma:matrixlocallaw} and
\ref{lemma:secondorderapprox} to the case where $F$ is diagonal and invertible,
and $z$ belongs to
\[U_\delta^\C=\{z \in U_\delta:\,|\Im z| \geq N^{-2}\}.\]
This latter reduction is for convenience of verifying the moment
condition of Lemma \ref{lemma:expectation}.

\begin{lemma}\label{lemma:reduction}
Suppose Lemmas \ref{lemma:matrixlocallaw} and \ref{lemma:secondorderapprox}
hold for $z \in U_\delta^\C$ and when $F$ is replaced by any invertible diagonal
matrix $T$ satisfying $\|T\| \leq C$. Then they hold also for the given
matrix $F$ and any $z \in U_\delta$.
\end{lemma}
\begin{proof}
Applying rotational invariance in law of $X$ and
the transformations $F \mapsto O'FO$, $X \mapsto O'X$, $V \mapsto O'VO$, and
$W \mapsto O'WO$, we may reduce from $F$ to the diagonal matrix $T$ of its
eigenvalues.

If $T$ is not invertible and/or $z \notin U_\delta^\C$, consider an invertible
matrix $\tilde{T}$ with $\|T-\tilde{T}\| \leq N^{-2}$ and
$\tilde{z} \in U_\delta^\C$ with $|z-\tilde{z}| \leq N^{-2}$.
Then, denoting by $\tilde{m}_0$ and $\tilde{\Delta}$
these quantities defined with $\tilde{T}$ and $\tilde{z}$,
on the high-probability event where
$\spec(X'TX) \subset \supp(\mu_0)_{\delta/2}$ and $\|X\|<C$, we have
\[|m_0-\tilde{m}_0|<CN^{-2},\qquad
|\partial_z m_0-\partial_z \tilde{m}_0|<CN^{-2},\qquad
\|\Delta-\tilde{\Delta}\|<CN^{-2}.\]
Here, the first two statements follow from (\ref{eq:m0stieltjes}) and the
condition $z \in U_\delta$, and the
third applies Proposition \ref{prop:m0regularoutsiderestate} and the identity
$A^{-1}-B^{-1}=A^{-1}(B-A)B^{-1}$. Bounding the trace by $M$ times the operator
norm, one may then verify $\Tr \Delta V-\Tr \tilde{\Delta}V
\prec N^{-1/2}\|V\|_\HS$. Similarly, the quantity on the left of Lemma
\ref{lemma:secondorderapprox} changes by $\O(N^{1/2}\|V\|\|W\|)$ upon
replacing $m_0$ and $\Delta$ by $\tilde{m}_0$ and $\tilde{\Delta}$. So
it suffices to establish Lemmas \ref{lemma:matrixlocallaw} and
\ref{lemma:secondorderapprox} for $\tilde{T}$ and $\tilde{z}$.
\end{proof}

Thus, in the remainder of this section, we consider a diagonal matrix 
\[T=\diag(t_1,\ldots,t_M) \in \R^{M \times M},
\qquad t_\a \neq 0 \text{ for all } t_\a.\]
Define the $(N+M) \times (N+M)$ linearized resolvent
\[G(z)=\begin{pmatrix} -z\Id_N & X' \\ X & -T^{-1} \end{pmatrix}^{-1}
=\begin{pmatrix} G_N(z) & G_o(z)' \\ G_o(z) & G_M(z) \end{pmatrix},\]
where the Schur complement identity yields
\begin{equation}\label{eq:G}
G_N(z)=(X'TX-z\Id_N)^{-1},\qquad G_o(z)=TXG_N(z),\qquad G_M(z)=TXG_N(z)X'T-T.
\end{equation}
From (\ref{eq:GMPiM}), we have
\begin{equation}\label{eq:Deltarewrite}
\Delta(z)=T^{-1}(G_M(z)-\Pi_M(z))T^{-1},\qquad
\Pi_M(z)=-T(\Id+m_0(z)T)^{-1}.
\end{equation}
Note that $G,G_N,G_M$ are symmetric, and $\Pi_M$ is diagonal. 
We omit the spectral argument $z$ when the meaning is clear.\\

\noindent {\bf Notation:}
Define $\I_N=\{1,\ldots,N\}$, $\I_M=\{1,\ldots,M\}$, and the disjoint union
$\I=\I_N \sqcup \I_M$. We index $G_N$ by
$\I_N$, $G_M$ by $\I_M$, and $G$ by $\I$. We use Roman letters
$i,j,k$ for indices in $\I_N$, Greek letters $\a,\b,\g$ for indices in $\I_M$,
and capital letters $A,B,C$ for general indices in $\I$.
We denote by $\x_i \in \R^M$ and $\x_\a \in \R^N$ the $i^{\text{th}}$ column and
$\a^{\text{th}}$ row of $X$, both regarded as column vectors.
For any subset $\mathcal{S} \subset \I$, $X^{(\mathcal{S})}$ denotes $X$
with rows in $\mathcal{S} \cap \I_M$ and columns in $\mathcal{S} \cap \I_N$
removed, $T^{(\mathcal{S})}$ denotes $T$ with rows and columns in $\mathcal{S}
\cap \I_M$ removed, and $G^{(\mathcal{S})},G_N^{(\mathcal{S})}$ etc.\ denote
these quantities defined with $X^{(\mathcal{S})}$ and $T^{(\mathcal{S})}$ in
place of $X$ and $T$. We index these matrices by
$\I_N \setminus \mathcal{S}$ and $\I_M \setminus \mathcal{S}$.

\begin{lemma}[Resolvent identities]\label{lemma:schurcomplement}
\hspace{1in}
\begin{enumerate}[(a)]
\item For all $i \in \I_N$ and $\a \in \I_M$,
\[G_{ii}=-\frac{1}{z+\x_i'G_M^{(i)}\x_i},\qquad
G_{\a\a}=-\frac{t_\a}{1+t_\a \x_\a'G_N^{(\a)}\x_\a}.\]
\item For all $i \neq j \in \I_N$ and $\a \neq \b \in \I_M$, denoting by
$\e_i \in \R^N$ and $\e_\a \in \R^M$ the standard basis vectors for
coordinates $i$ and $\a$,
\begin{align*}
G_{ij}&=-G_{ii}\x_i'G_o^{(i)}\e_j=G_{ii}G_{jj}^{(i)}\x_i'G_M^{(ij)}\x_j,\\
G_{i\a}&=-G_{ii}\x_i'G_M^{(i)}\e_\a=-G_{\a\a}\e_i'G_N^{(\a)}\x_\a,\\
G_{\a\b}&=-G_{\a\a}\e_\b'G_o^{(\a)}\x_\a=G_{\a\a}G_{\b\b}^{(\a)}\x_\a'G_N^{(\a\b)}\x_\b.
\end{align*}
\item For all $C \in \I$ and $A,B \in \I \setminus \{C\}$,
\[G_{AB}^{(C)}=G_{AB}-\frac{G_{AC}G_{CB}}{G_{CC}}.\]
\end{enumerate}
\end{lemma}
\begin{proof}
See \cite[Lemma 4.4]{knowlesyin}.
\end{proof}

For $i \neq j \in \I_N$ and $\a \neq \b \in \I_M$, define
\[\cZ_i=\x_i'G_M^{(i)}\x_i-N^{-1}\Tr G_M^{(i)},
\qquad \cZ_\a=\x_\a'G_N^{(\a)}\x_\a-N^{-1}\Tr G_N^{(\a)},\]
\[\cZ_{ij}=\x_i'G_M^{(ij)}\x_j,\qquad
\cZ_{\a\b}=\x_\a'G_N^{(\a\b)}\x_\b.\]
We will use the following bounds implicitly throughout the remainder of this
section. Note that $(|z| \vee 1)^{-1} \leq 1$ and $|t_\a| \leq C$, so we will
omit these factors in the bounds in certain applications.

\begin{lemma}\label{lemma:entrywisebounds}
For all $z \in U_\delta$,
\begin{enumerate}[(a)]
\item (Norm bounds)
\[\|G_N\| \prec (|z| \vee 1)^{-1}, \qquad \|G_o\| \prec (|z| \vee 1)^{-1},
\qquad \|G_M\| \prec 1.\]
\item (Diagonal bounds) For all $i \in \I_N$ and $\a \in \I_M$,
\[G_{ii} \prec (|z| \vee 1)^{-1}, \qquad G_{ii}^{-1} \prec |z| \vee 1,\qquad
G_{\a\a} \prec t_\a, \qquad G_{\a\a}^{-1} \prec t_\a^{-1}.\]
\item ($\cZ$ bounds) For all $i \neq j \in \I_N$ and $\a \neq \b \in \I_M$,
\[\cZ_i \prec N^{-1/2},
\qquad \cZ_\a \prec (|z| \vee 1)^{-1}N^{-1/2},
\qquad \cZ_{ij} \prec N^{-1/2},
\qquad \cZ_{\a\b} \prec (|z| \vee 1)^{-1}N^{-1/2},\]
\[\e_i'G_N^{(\a)}\x_\a \prec (|z| \vee 1)^{-1}N^{-1/2},\qquad
\x_i'G_M^{(i)}\e_\a \prec t_\a N^{-1/2},\]
\[\x_i'G_o^{(i)}\e_j \prec (|z| \vee 1)^{-1}N^{-1/2},\qquad
\e_\b'G_o^{(\a)}\x_\a \prec t_\b(|z| \vee 1)^{-1}N^{-1/2}.\]
\item (Off-diagonal bounds) For all $i \neq j \in \I_N$ and
$\a \neq \b \in \I_M$,
\[G_{ij} \prec (|z| \vee 1)^{-2}N^{-1/2},
\qquad G_{i\a} \prec t_\a(|z| \vee 1)^{-1}N^{-1/2},
\qquad G_{\a\b} \prec t_\a t_\b (|z| \vee 1)^{-1}N^{-1/2}.\]
\end{enumerate}
\end{lemma}
\begin{proof}
By Theorem \ref{thm:sticktobulk}, $\spec(X'TX) \subset
\supp(\mu_0)_{\delta/2}$ holds with high probability. On this event,
$\|G_N\| \leq C\min(1/\delta,1/|z|)$. As $\|X\| \prec 1$, part (a)
follows from (\ref{eq:G}).

For (b), the bounds on $G_{ii}$ and $G_{\a\a}$ follow from (\ref{eq:G}) and
part (a). The bounds on $G_{ii}^{-1}$ and
$G_{\a\a}^{-1}$ follow from Lemma \ref{lemma:schurcomplement}(a), 
$\|G_N^{(\a)}\| \prec 1$ and $\|G_M^{(i)}\| \prec 1$ in part (a),
and $\|\x_i\| \prec 1$ and $\|\x_\a\| \prec 1$.

For (c), note that part (a) implies $\|G_M^{(i)}\|_\HS \prec N^{1/2}$
and $\|G_N^{(\a)}\|_\HS \prec (|z| \vee 1)^{-1}N^{1/2}$. Then
the bounds for $\cZ_i$ and $\cZ_\a$ follow from Lemma \ref{lemma:hansonwright}
applied conditionally on $X^{(i)}$ and $X^{(\a)}$. For $\cZ_{ij}$, as $\x_i$ is
independent of $G_M^{(ij)}\x_j$, we have that $\x_i'G_M^{(ij)}\x_j$ is Gaussian
conditional on $X^{(i)}$ and $|\x_i'G_M^{(ij)}\x_j| \prec
N^{-1/2}\|G_M^{(ij)}\x_j\| \prec N^{-1/2}\|G_M^{(ij)}\| \prec N^{-1/2}$. The
remaining five bounds are similar.

Finally, (d) follows from Lemma \ref{lemma:schurcomplement}(b) and parts (b)
and (c).
\end{proof}

\subsection{Linear functions of the resolvent}
We prove Lemma \ref{lemma:matrixlocallaw}.
Let $\E_i$ and $\E_\a$ be the partial expectations over column $\x_i$ and row
$\x_\a$ of $X$. For $S \subset \I_N$ or $S \subset \I_M$, let
$\E_S$, $\Q_S$, and $\prec_\ell$ be as in Section \ref{subsec:fluctuationavg}.
Note that $\E_i[\cZ_i]=0$, $\E_\a[\cZ_\a]=0$, and
$\E_\a[\cZ_{\a\b}]=\E_\b[\cZ_{\a\b}]=0$. We verify that these quantities
satisfy the conditions of Lemma \ref{lemma:fluctuationavg} (with $N$ or with $M
\asymp N$ in place of $n$).

\begin{lemma}\label{lemma:Zgood}
For $z \in U_\delta^\C$, each $\cZ_* \in \{\cZ_i,\cZ_\a,\cZ_{\a\b}\}$, and some
constants $C_1,C_2,\ldots>0$, we have $\E[|\cZ_*|^\ell] \leq N^{C_\ell}$ for
each $\ell>0$. Furthermore, for any constant $\ell>0$,
\begin{enumerate}[(a)]
\item For $S \subset \I_N$ with $i \notin S$ and $|S| \leq \ell$,
$\Q_S\cZ_i \prec_\ell N^{-1/2-|S|/2}$.
\item For $S \subset \I_M$ with $\a \notin S$ and $|S| \leq \ell$,
$\Q_S\cZ_\a \prec_\ell N^{-1/2-|S|/2}$.
\item For $S \subset \I_M$ with $\a,\b \notin S$ and $|S| \leq \ell$,
$\Q_S\cZ_{\a\b} \prec_\ell N^{-1/2-|S|/2}$.
\end{enumerate}
\end{lemma}
\begin{proof}
Let $C_\ell>0$ denote an $\ell$-dependent constant that may change from instance
to instance. Taking the expectation first over $\x_i$, we have
$\E[|\x_i'G_M^{(i)}\x_i|^\ell] \leq \E[\|G_M^{(i)}\|^\ell\|\x_i\|^{2\ell}] \leq
C_\ell\E[\|G_M^{(i)}\|^\ell]$.
Note that $\|G_N^{(i)}\| \leq 1/|\Im z| \leq N^2$ for $z \in U_\delta^\C$,
so $\|G_M^{(i)}\| \leq C(N^2\|X^{(i)}\|^2+1)$ by (\ref{eq:G}). Then
$\E[|\x_i'G_M^{(i)}\x_i|^\ell] \leq N^{C_\ell}$ follows.
Also $\E[|N^{-1}\Tr G_M^{(i)}|^\ell]
\leq C_\ell\E[\|G_M^{(i)}\|^\ell] \leq N^{C_\ell}$,
so $\E[|\cZ_i|^\ell] \leq N^{C_\ell}$.
Similar arguments show $\E[|\cZ_\a|^\ell] \leq N^{C_\ell}$ and
$\E[|\cZ_{\a\b}|^\ell] \leq N^{C_\ell}$.

For the remaining statements, the argument is similar to the type of resolvent
expansion performed in \cite{bloemendaletal}. We begin with (a):
For $S=\emptyset$, this follows from Lemma \ref{lemma:entrywisebounds}.
For $|S| \geq 1$, observe that
$\cZ_i=\Q_i[\x_i'G_M^{(i)}\x_i]$, so $\Q_S[\cZ_i]=\Q_{S \cup
\{i\}}[\x_i'G_M^{(i)}\x_i]$. Define
\[G_{\x j}^{(i)}=\sum_{\a \in \I_M} X_{i \a}G_{\a j}^{(i)}=\x_i'G_o^{(i)}\e_j.\]
Suppose $j \in S$. We may
apply Lemma \ref{lemma:schurcomplement}(c) to write
\[\x_i'G_M^{(i)}\x_i=\sum_{\a,\b} X_{\a i}G_{\a\b}^{(i)}X_{\b
i}=L(\{j\})+R(\{j\})\]
where
\[L(\{j\})=\sum_{\a,\b} X_{\a i}G_{\a\b}^{(ij)}X_{\b i}=\x_i'G_M^{(ij)}\x_i,
\qquad R(\{j\})=\sum_{\a,\b} X_{\a i}\frac{G_{\a j}^{(i)}
G_{\b j}^{(i)}}{G_{jj}^{(i)}}X_{\b i}
=\frac{(G_{\x j}^{(i)})^2}{G_{jj}^{(i)}}.\]
Here, $L(\{j\})$ no longer depends on $\x_j$.
Note that Lemma \ref{lemma:schurcomplement}(c) yields, for $j \neq k$,
\begin{equation}\label{eq:Gijkidentities}
G_{\x j}=G_{\x j}^{(k)}+\frac{G_{\x k}G_{jk}}{G_{kk}},\qquad
\frac{1}{G_{jj}}=\frac{1}{G_{jj}^{(k)}}-\frac{G_{jk}^2}
{G_{jj}G_{jj}^{(k)}G_{kk}}.
\end{equation}
Then if $|S| \geq 2$ and $k \in S$, let us apply these identities
to the numerator and denominator of $R(\{j\})$ to further write
\[\x_i'G_M^{(i)}\x_i=L(\{j,k\})+R(\{j,k\}),\]
where
\[L(\{j,k\})=L(\{j\})+\frac{(G_{\x j}^{(ik)})^2}{G_{jj}^{(ik)}}\]
collects terms which no longer depend on at least one of $\x_j$ or $\x_k$,
and the remainder is
\[R(\{j,k\})=-(G_{\x j}^{(ik)})^2
\frac{(G_{jk}^{(i)})^2}{G_{jj}^{(i)}G_{jj}^{(ik)}G_{kk}^{(i)}}
+G_{\x j}^{(ik)}\frac{G_{\x k}^{(i)}G_{jk}^{(i)}}{G_{kk}^{(i)}}
\frac{1}{G_{jj}^{(i)}}+\frac{G_{\x k}^{(i)}G_{jk}^{(i)}}{G_{kk}^{(i)}}
G_{\x j}^{(i)}\frac{1}{G_{jj}^{(i)}}.\]
Recursively using (\ref{eq:Gijkidentities}) to apply this procedure for each
index in $S$, we obtain
\[\x_i'G_M^{(i)}\x_i=L(S)+R(S),\]
where
\begin{itemize}
\item Each term of $L(S)$ does not depend on at least one of the columns
$(\x_j:j \in S)$.
\item $R(S)$ is a sum of at most $C_\ell$ summands, each summand a product of
at most $C_\ell$ terms, for a constant $C_\ell$ depending only on $\ell$ (the
maximum size of $S$).
\item Each summand of $R(S)$ is a product of 2 terms of the
form $G_{\x j}^{(\cT)}$, some number $m \geq |S|-1$ of terms of the
form $G_{jk}^{(\cT)}$ for $j \neq k$, and $m+1$ terms of the form
$(G_{jj}^{(\cT)})^{-1}$. Here $i \in \cT \subseteq S$.
\end{itemize}
We observe that $\Q_{S \cup \{i\}}[L(S)]=0$. Applying
$G_{\x j}^{(\cT)} \prec_\ell (|z| \vee 1)^{-1}N^{-1/2}$,
$G_{jk}^{(\cT)} \prec_\ell (|z| \vee 1)^{-1}N^{-1/2}$, and
$(G_{jj}^{(\cT)})^{-1} \prec_\ell |z| \vee 1$, we have
$R(S) \prec_\ell N^{-(|S|+1)/2}$.
Then part (a) follows from Lemma \ref{lemma:expectation} and the bound
\[\Q_{S \cup \{i\}}[R(S)]=\left(\prod_{j \in S \cup \{i\}}
(1-\E_j)\right)[R(S)]
 \leq \sum_{\cT:\cT \subseteq S \cup \{i\}} |\E_\cT[R(S)]| \prec_\ell R(S).\]

The proof of part (b) is similar: Define
\[\vG_{\a\b}=\frac{G_{\a\b}}{|t_\a t_\b|^{1/2}},\qquad
\vG_{\b\x}^{(\a)}=\sum_{i \in \I_N} \frac{G_{\b i}}{|t_\b|^{1/2}}
X_{i\a}=|t_\b|^{-1/2}\e_\b'G_o^{(\a)}\x_\a.\]
We apply Lemma \ref{lemma:schurcomplement}(c) in the forms
\begin{equation}\label{eq:vGidentities}
\x_\a'G^{(\a)}_N\x_\a=\x_\a'G_N^{(\a\b)}\x_\a
+\frac{(\vG_{\b\x}^{(\a)})^2}{\vG_{\b\b}^{(\a)}},\qquad
\vG_{\a\b}=\vG_{\a\b}^{(\g)}+\frac{\vG_{\a\g}\vG_{\b\g}}{\vG_{\g\g}},
\qquad \frac{1}{\vG_{\b\b}}=\frac{1}{\vG_{\b\b}^{(\g)}}
-\frac{(\vG_{\b\g})^2}{\vG_{\b\b}\vG_{\b\b}^{(\g)}\vG_{\g\g}}.
\end{equation}
This allows us to write, for each $S \subset \I_M$ with $|S| \geq 1$ and $\a
\notin S$,
\[\x_\a'G_N^{(\a)}\x_\a=L(S)+R(S),\]
where $L(S)$ contains terms not depending on at least one row $(\x_\b:\b \in
S)$, and each summand of $R(S)$ is a product of 2 terms of the form $\vG_{\b
\x}^{(\cT)}$, $m \geq |S|-1$ terms of the form $\vG_{\b\g}^{(\cT)}$ for $\b \neq
\g$, and $m+1$ terms of the form $(\vG_{\b\b}^{(\cT)})^{-1}$. Applying
$\vG_{\b\x}^{(\cT)} \prec N^{-1/2}$, $\vG_{\b\g}^{(\cT)} \prec N^{-1/2}$, and
$(\vG_{\b\b}^{(\cT)})^{-1} \prec 1$, we obtain part (b). The argument for
part (c) is similar and omitted for brevity.
\end{proof}

Define the empirical Stieltjes transform
\[m_N(z)=N^{-1} \Tr G_N(z)=N^{-1}\Tr (X'TX-z\Id)^{-1}.\]
We next establish a bound on $m_N-m_0$ for $z$ separated from $\supp(\mu_0)$.
We follow \cite{bloemendaletal,knowlesyin}, although for simplicity we will
use the result of Theorem \ref{thm:sticktobulk} to establish ``stability''
of the Marcenko-Pastur equation, rather than proving this directly using the
stochastic continuity argument of \cite{bloemendaletal}.
\begin{lemma}\label{lemma:mNapprox}
Let $z \in U_\delta^\C$. Then $m_N(z)-m_0(z) \prec N^{-1}$.
\end{lemma}
\begin{proof}
Recall the function
\[z_0(m)=-\frac{1}{m}+\frac{1}{N}\sum_\a \frac{t_\a}{1+t_\a m}.\]
We first establish the following claim: If for all $z \in U_\delta^\C$ and a
constant $\tau>0$ we have
\begin{equation}\label{eq:zdiff}
z-z_0(m_N(z)) \prec N^{-\tau},
\end{equation}
then also for all $z \in U_\delta^\C$ we have
\begin{equation}\label{eq:mdiff}
m_0(z)-m_N(z) \prec N^{-\tau}.
\end{equation}
Indeed, fix any constants $\eps,D>0$. Suppose first
that $\Im z \geq N^{-\tau+\eps}$. Let $\cE$ be the event where
$|z-z_0(m_N)|<N^{-\tau+\eps/2}$, which holds with probability at least
$1-N^{-D}$ by (\ref{eq:zdiff}). On $\cE$ we have $\Im z_0(m_N)>0$, so
Theorem \ref{thm:MP} guarantees that $m_0(z_0(m_N))$ is the
unique root $m \in \C^+$ to the equation $z_0(m_N)=z_0(m)$. Thus
$m_0(z_0(m_N))=m_N$. Applying $|\partial_z m_0| \leq C$
for all $z \in U_\delta$ and
integrating this bound along a path from $z$ to $z_0(m_N)$, we obtain
\[|m_0(z)-m_N(z)|=|m_0(z)-m_0(z_0(m_N))|<CN^{-\tau+\eps/2}.\]
Now suppose $\Im z \in (0,N^{-\tau+\eps})$.
Let $\tilde{z}$ be such that $\Re \tilde{z}=\Re z$ and $\Im
\tilde{z}=N^{-\tau+\eps}$. By the preceding argument,
$|m_0(\tilde{z})-m_N(\tilde{z})|<CN^{-\tau+\eps/2}$ with probability
at least $1-N^{-D}$. Apply again $|\partial_z m_0| \leq C$, and also
$|\partial_z m_N| \leq C$ on the event $\spec(X'TX) \subset
\supp(\mu_0)_{\delta/2}$, which holds with probability $1-N^{-D}$ by Theorem
\ref{thm:sticktobulk}. Then
\[|m_0(z)-m_N(z)| \leq |m_0(z)-m_0(\tilde{z})|+|m_0(\tilde{z})-m_N(\tilde{z})|
+|m_N(\tilde{z})-m_N(z)|<CN^{-\tau+\eps/2}\]
with probability $1-2N^{-D}$. The same arguments hold by conjugation symmetry
for $\Im z<0$, and hence in all cases we obtain (\ref{eq:mdiff}).

It remains to establish (\ref{eq:zdiff}) for $\tau=1$.
Applying Lemma \ref{lemma:schurcomplement}(a),
\begin{equation}\label{eq:Gii}
G_{ii}^{-1}=-z-\x_i'G_M^{(i)}\x_i=-z-N^{-1}\Tr
G_M^{(i)}-\cZ_i.
\end{equation}
Next, applying Lemma \ref{lemma:schurcomplement}(c),
\[N^{-1}\Tr G_M^{(i)}=N^{-1}\sum_\a G_{\a\a}^{(i)}=N^{-1}\sum_\a\left(G_{\a\a}
-\frac{G_{i\a}^2}{G_{ii}}\right)=N^{-1}\Tr G_M-G_{ii}^{-1}N^{-1}\sum_\a
G_{i\a}^2.\]
Then applying the bounds $G_{ii}^{-1} \prec |z| \vee 1$ and
$G_{i\a} \prec (|z| \vee 1)^{-1/2}N^{-1/2}$,
\begin{equation}\label{eq:Giiinv}
G_{ii}^{-1}=-z-N^{-1}\Tr G_M-\cZ_i+\O(N^{-1}).
\end{equation}
Applying $\cZ_i \prec N^{-1/2}$ and $G_{jj} \prec (|z| \vee 1)^{-1}$,
this yields $G_{jj}/G_{ii}-1=G_{jj}(G_{ii}^{-1}-G_{jj}^{-1}) \prec
(|z| \vee 1)^{-1}N^{-1/2}$ for all $i,j \in \I_N$. Then for all $i \in
\I_N$, we have $m_N/G_{ii}-1 \prec (|z| \vee 1)^{-1}N^{-1/2}$, and hence also
\begin{equation}\label{eq:GiimN}
G_{ii}/m_N-1 \prec (|z| \vee 1)^{-1}N^{-1/2}.
\end{equation}
Expanding $G_{ii}^{-1}$ around $m_N^{-1}$,
\begin{align*}
N^{-1}\sum_i G_{ii}^{-1}&=N^{-1}\sum_i \left(m_N^{-1}-m_N^{-2}(G_{ii}-m_N)
+m_N^{-2}G_{ii}^{-1}(G_{ii}-m_N)^2\right)\\
&=m_N^{-1}+N^{-1}m_N^{-2}\sum_i G_{ii}^{-1}(G_{ii}-m_N)^2.
\end{align*}
Thus
\[m_N^{-1}=N^{-1}\sum_i G_{ii}^{-1}\left(1-(G_{ii}/m_N-1)^2\right).\]
Applying (\ref{eq:GiimN}), $G_{ii}^{-1} \prec |z| \vee 1$, and
(\ref{eq:Giiinv}), we obtain
\begin{align}
m_N^{-1}=N^{-1}\sum_i G_{ii}^{-1}+\O(N^{-1})
=-z-N^{-1}\Tr G_M-N^{-1}\sum_i \cZ_i+\O(N^{-1}).
\label{eq:mNexpand}
\end{align}

Next, applying Lemma \ref{lemma:schurcomplement} and the bounds
$G_{\a\a}^{-1} \prec t_\a^{-1}$ and
$G_{i\a} \prec t_\a N^{-1/2}$, we obtain analogously to (\ref{eq:Giiinv})
\begin{equation}\label{eq:taGaa}
t_\a G_{\a\a}^{-1}=-1-t_\a\x_\a'G_N^{(\a)}\x_\a
    =-1-t_\a m_N-t_\a \cZ_\a+\O(t_\a N^{-1}).
\end{equation}
Since $G_{\a\a}/t_\a \prec 1$ and $\cZ_\a \prec N^{-1/2}$, the above implies
in particular $(1+t_\a m_N)^{-1} \prec 1$ and
$(1+t_\a m_N+t_\a \cZ_\a)^{-1} \prec 1$. Then multiplying the above by
$G_{\a\a}(1+t_\a m_N+t_\a \cZ_\a)^{-1}$, we obtain
\begin{align}
G_{\a\a}&=-\frac{t_\a}{1+t_\a m_N+t_\a \cZ_\a}+\O(t_\a^2N^{-1})\nonumber\\
&=-\frac{t_\a}{1+t_\a m_N}+\frac{t_\a^2 \cZ_\a}{(1+t_\a m_N)^2}
+\O(t_\a^2N^{-1}).\label{eq:Gaa}
\end{align}
As $\Tr G_M=\sum_\a G_{\a\a}$, combining with (\ref{eq:mNexpand}) and recalling
the definition of $z_0(m)$, we have
\[z-z_0(m_N)=-N^{-1}\sum_i \cZ_i-N^{-1}\sum_\a \frac{t_\a^2}{(1+t_\a
m_N)^2}\cZ_\a+\O(N^{-1}).\]

Applying first the bounds $\cZ_i \prec N^{-1/2}$, $\cZ_\a \prec N^{-1/2}$, and
$(1+t_\a m_N)^{-1} \prec 1$ to the above, we obtain
$z-z_0(m_N) \prec N^{-1/2}$. Then (\ref{eq:mdiff}) yields $m_0-m_N \prec
N^{-1/2}$. This allows us to replace $m_N$ by $m_0$ with an additional
$\O(N^{-1})$ error, yielding
\[z-z_0(m_N)=-N^{-1}\sum_i \cZ_i-N^{-1}\sum_\a \frac{t_\a^2}{(1+t_\a
m_0)^2}\cZ_\a+\O(N^{-1}).\]
By Proposition \ref{prop:m0regularoutsiderestate},
$|t_\a|^2/|1+t_\a m_0|^2 \leq C$ for a constant 
$C>0$. Then Lemmas \ref{lemma:fluctuationavg}(a) and
\ref{lemma:Zgood}(a--b) imply that both sums above are $\O(N^{-1})$.
So (\ref{eq:zdiff}) holds with $\tau=1$.
\end{proof}

We record an estimate from the above proof for future use:
\begin{lemma}\label{lemma:Gaa}
For $z \in U_\delta^\C$ and each $\a \in \I_M$,
\[\frac{G_{\a\a}-\Pi_{\a\a}}{t_\a^2}=\frac{1}{(1+t_\a
m_0)^2}\cZ_\a+\O(N^{-1}).\]
\end{lemma}
\begin{proof}
This follows from (\ref{eq:Gaa}), upon applying $m_N-m_0 \prec N^{-1}$
and $|1+t_\a m_0| \geq c$ from Proposition \ref{prop:m0regularoutsiderestate} to yield
$-t_\a/(1+t_\a m_N)=\Pi_{\a\a}+\O(t_\a^2 N^{-1})$ and
$t_\a^2\cZ_\a/(1+t_\a m_N)^2=t_\a^2\cZ_\a/(1+t_\a m_0)^2+\O(t_\a^2N^{-1})$.
\end{proof}

We now conclude the proof of Lemma \ref{lemma:matrixlocallaw}: By Lemma
\ref{lemma:reduction}, we may consider the case where $F=T$ is diagonal and
invertible, and $z \in U_\delta^\C$. We write
\[\Tr \Delta V=\sum_\a \Delta_{\a\a}V_{\a\a}+\sum_{\a \neq \b}
\Delta_{\a\b}V_{\a\b}.\]
Applying (\ref{eq:Deltarewrite}) and Lemma \ref{lemma:Gaa},
\[\sum_\a \Delta_{\a\a}V_{\a\a}
=\sum_\a \frac{G_{\a\a}-\Pi_{\a\a}}{t_\a^2}V_{\a\a}
=\sum_\a \frac{1}{(1+t_\a m_0)^2}V_{\a\a}\cZ_\a+\O(N^{-1/2}\|V\|_\HS).\]
As $|1+t_\a m_0|>c$ by Proposition \ref{prop:m0regularoutsiderestate}, we may apply
Lemmas \ref{lemma:fluctuationavg}(a) and \ref{lemma:Zgood}(b) to yield
\[\sum_\a \Delta_{\a\a}V_{\a\a} \prec N^{-1/2}\|V\|_\HS.\]

For the off-diagonal contribution, by (\ref{eq:Deltarewrite}) and Lemma
\ref{lemma:schurcomplement}(b) we have
\[\sum_{\a \neq \b} \Delta_{\a\b}V_{\a\b}
=\sum_{\a \neq \b}\frac{G_{\a\b}}{t_\a t_\b}V_{\a\b}
=\sum_{\a \neq \b} \frac{G_{\a\a}G_{\b\b}^{(\a)}}{t_\a t_\b}V_{\a\b}
\cZ_{\a\b}.\]
As $\cZ_{\a\b} \prec N^{-1/2}$ and $\sum_{\a \neq \b} |V_{\a\b}| \leq M
(\sum_{\a \neq \b} |V_{\a\b}|^2)^{1/2} \prec N\|V\|_\HS$, we may make
$\O(N^{-1})$ adjustments of the coefficients of $V_{\a\b}\cZ_{\a\b}$ while
incurring an $\O(N^{-1/2}\|V\|_\HS)$ error in the sum. Then, applying
$G_{\b\b}^{(\a)}/t_\b=G_{\b\b}/t_\b+\O(N^{-1})$ by Lemma
\ref{lemma:schurcomplement}(c), followed by Lemma \ref{lemma:Gaa}, we have
\[\sum_{\a \neq \b} \Delta_{\a\b}V_{\a\b}=\mathrm{I}+\mathrm{II}+\mathrm{III}
+\mathrm{IV}+\O(N^{-1/2}\|V\|_\HS)\]
where
\begin{align*}
\mathrm{I}&=\sum_{\a \neq \b}
\frac{\Pi_{\a\a}\Pi_{\b\b}}{t_\a t_\b}V_{\a\b}\cZ_{\a\b},\\
\mathrm{II}&=\sum_{\a \neq \b} \frac{t_\a}{(1+t_\a m_0)^2} \cZ_\a
\frac{\Pi_{\b\b}}{t_\b}V_{\a\b}\cZ_{\a\b},\\
\mathrm{III}&=\sum_{\a \neq \b} \frac{\Pi_{\a\a}}{t_\a}
\frac{t_\b}{(1+t_\b m_0)^2} \cZ_\b V_{\a\b}\cZ_{\a\b},\\
\mathrm{IV}&=\sum_{\a \neq \b} \frac{t_\a}{(1+t_\a m_0)^2} \cZ_\a
\frac{t_\b}{(1+t_\b m_0)^2} \cZ_\b V_{\a\b}\cZ_{\a\b}.
\end{align*}

Lemmas \ref{lemma:fluctuationavg}(b) and \ref{lemma:Zgood}(c) yield
$\mathrm{I} \prec N^{-1/2}\|V\|_\HS$. For $\mathrm{II}$, first fixing $\a$ and
summing over $\beta$, Lemmas \ref{lemma:fluctuationavg}(a) and
\ref{lemma:Zgood}(c) yield
\[\sum_{\b \notin \{\a\}} \frac{\Pi_{\b\b}}{t_\b} V_{\a\b} \cZ_{\a\b}
\prec N^{-1/2}\|\v_\a\|\]
where $\v_\a$ is row $\a$ of $V$. Then, applying $\cZ_\a \prec N^{-1/2}$,
\[\mathrm{II} \prec \sum_\a N^{-1}\|\v_\a\| \prec N^{-1/2}\|V\|_\HS.\]
Similarly $\mathrm{III} \prec N^{-1/2}\|V\|_\HS$. Finally,
the direct bounds $\cZ_\a,\cZ_\b,\cZ_{\a\b} \prec N^{-1/2}$ and
$\sum_{\a \neq \b} |V_{\a\b}| \prec N\|V\|_\HS$ yield $\mathrm{IV} \prec
N^{-1/2}\|V\|_\HS$. Thus $\Tr \Delta V \prec N^{-1/2}\|V\|_\HS$ as desired.

\subsection{Quadratic functions of the resolvent}
We now prove Lemma \ref{lemma:secondorderapprox}.
We will apply the fluctuation averaging mechanism, Lemma
\ref{lemma:fluctuationavg}, to the quantities
\[\cY_{\a\b\g\r}=(\x_\a'G_N^{(\a\b\g\r)}\x_\b)(\x_\g'G_N^{(\a\b\g\r)}\x_\r),
\qquad \cY_{\a\b\g}=(\x_\a'G_N^{(\a\b\g)}\x_\b)(\x_\a'G_N^{(\a\b\g)}\x_\g),\]
\[\tilde{\cY}_{\a\b\g}=G_{\a\a}^{(\b\g)}
(\x_\a'G_N^{(\a\b\g)}\x_\b)(\x_\a'G_N^{(\a\b\g)}\x_\g),\]
\[\cY_{\a\b,1}=\cZ_{\a\b}^2-N^{-1}\x_\a'(G_N^{(\a\b)})^2\x_\a, \qquad
\cY_{\a\b,2}=N^{-1}\x_\a'(G_N^{(\a\b)})^2\x_\a-N^{-2}\Tr[(G_N^{(\a\b)})^2]\]
where $\a,\b,\g,\r$ above are distinct.
Note that each $\cY_*$ above satisfies $\cY_* \prec N^{-1}$, and furthermore
\[\E_\a[\cY_{\a\b\g\r}]=\E_\b[\cY_{\a\b\g\r}]
=\E_\g[\cY_{\a\b\g\r}]=\E_\r[\cY_{\a\b\g\r}]=0,\]
\[\E_\b[\cY_{\a\b\g}]=\E_\g[\cY_{\a\b\g}]=0,\qquad
\E_\b[\tilde{\cY}_{\a\b\g}]=\E_\g[\tilde{\cY}_{\a\b\g}]=0,
\qquad \E_\b[\cY_{\a\b,1}]=0, \qquad \E_\a[\cY_{\a\b,2}]=0.\]
The following verifies the conditions of Lemma \ref{lemma:fluctuationavg} (with
$N$ or with $M \asymp N$ in place of $n$).

\begin{lemma}\label{lemma:Zprodgood}
For $z \in U_\delta^\C$, each $\cY_* \in
\{\cY_{\a\b\g\r},\cY_{\a\b\g},\tilde{\cY}_{\a\b\g},
\cY_{\a\b,1},\cY_{\a\b,2}\}$, and some constants $C_1,C_2,\ldots>0$,
we have $\E[|\cY_*|^\ell] \leq N^{C_\ell}$ for all $\ell>0$. Furthermore,
for any constant $\ell>0$,
\begin{enumerate}[(a)]
\item For $S \subset \I_M$ with $\a,\b,\g,\r \notin S$ and $|S| \leq \ell$,
$\Q_S \cY_{\a\b\g\r} \prec_\ell N^{-1-|S|/2}$.
\item For $S \subset \I_M$ with $\a,\b,\g \notin S$ and $|S| \leq \ell$,
$\Q_S \cY_{\a\b\g} \prec_\ell N^{-1-|S|/2}$.
\item For $S \subset \I_M$ with $\a,\b,\g \notin S$ and $|S| \leq \ell$,
$\Q_S \tilde{\cY}_{\a\b\g} \prec_\ell N^{-1-|S|/2}$.
\item For $S \subset \I_M$ with $\a,\b \notin S$ and $|S| \leq \ell$,
$\Q_S \cY_{\a\b,1} \prec_\ell N^{-1-|S|/2}$.
\item For $S \subset \I_M$ with $\a,\b \notin S$ and $|S| \leq \ell$,
$\Q_S \cY_{\a\b,2} \prec_\ell N^{-1-|S|/2}$.
\end{enumerate}
\end{lemma}
\begin{proof}
The bound $\E[|\cY_*|^\ell] \leq N^{C_\ell}$ follows from
$\|G_N^{(*)}\| \leq 1/|\Im z| \leq N^2$ for $z \in U_\delta^\C$ and the same
arguments as in Lemma \ref{lemma:Zgood}.

The remainder of the proof is also similar to Lemma \ref{lemma:Zgood}(b--c):
For (a), define
\[\vG_{\a\b}=\frac{G_{\a\b}}{|t_\a t_\b|^{1/2}}, \qquad
\vG_{\eta\x_\a}^{(\a\b\g\r)}=\e_\eta'G_o^{(\a\b\g\r)}\x_\a/|t_\eta|^{1/2}=
\sum_i
\frac{G_{\eta i}^{(\a\b\g\r)}}{|t_\eta|^{1/2}} X_{i\a}.\]
We iterate through $S$ and expand both of the terms
$\x_\a'G_N^{(\a\b\g\r)}\x_\b$ and $\x_\g'G_N^{(\a\b\g\r)}\x_\r$ simultaneously,
using Lemma \ref{lemma:schurcomplement}(c) in the form
\[\x_\a'G_N^{(\a\b\g\r)}\x_\b=\x_\a'G_N^{(\a\b\g\r\eta)}\x_\b
+\frac{\vG_{\eta \x_\a}^{(\a\b\g\r)}\vG_{\eta
\x_\b}^{(\a\b\g\r)}}{\vG_{\eta\eta}^{(\a\b\g\r)}}\]
together with the latter two identities of (\ref{eq:vGidentities}). This yields,
for each $S \subset \I_M$ with $|S| \geq 1$ and $\a,\b,\g,\r \notin S$, a
decomposition
\[\cY_{\a\b\g\r}=L(S)+R(S)\]
where $L(S)$ collects terms not depending on at least one row $(\x_\eta:\eta \in
S)$, and each summand of $R(S)$ is a product of $m \geq |S|+2$ ``numerator''
terms of the form $\x_\a'G_N^{(\cT)}\x_\b$, $\vG_{\eta\x_\a}^{(\cT)}$, or
$\vG_{\eta\nu}^{(\cT)}$ and $m-2$ ``denominator'' terms of the form
$(\vG_{\eta\eta}^{(\cT)})^{-1}$.
Each numerator term is $\O(N^{-1/2})$ and each
denominator term is $\O(1)$, so $R(S) \prec_\ell N^{-1-|S|/2}$. Then
$\Q_S[\cY_{\a\b\g\r}]=\Q_S[R(S)] \prec_\ell N^{-1-|S|/2}$.

The same argument holds for parts (b--e). For (c), we expand also the term
$G_{\a\a}^{(\b\g)}$ together with the other two terms,
using the second identity of (\ref{eq:vGidentities}).
For (d) and (e) we apply this argument separately to
\[\cZ_{\a\b}^2=(\x_\a'G_N^{(\a\b)}\x_\b)^2, \qquad
\x_\a'(G_N^{(\a\b)})^2\x_\a=\sum_i (\x_\a'G_N^{(\a\b)}\e_i)^2, \qquad
\Tr [(G_N^{(\a\b)})^2]=\sum_{i,j} (\e_i'G_N^{(\a\b)}\e_j)^2\]
and to each of the above summands.
We obtain the additional numerator terms
$\x_\a'G_N^{(\a\b)}\e_i$, $\e_i'G_N^{(\a\b)}\e_j$, and
$\vG_{i\eta}^{(\cT)}=G_{i\eta}^{(\cT)}/|t_\eta|^{1/2}$ in the expansions,
which are still $\O(N^{-1/2})$.
\end{proof}

Using this, we prove Lemma \ref{lemma:secondorderapprox}. By Lemma
\ref{lemma:reduction}, we may consider $F=T$ diagonal and invertible, and $z \in
U_\delta^\C$. For convenience,
let us normalize so that $\|V\|=\|W\|=1$. We write
\begin{equation}\label{eq:TrDVDW}
\Tr \Delta V \Delta W=\sum_{\a,\b,\g,\r} \Delta_{\a\b}V_{\b\g}\Delta_{\g\r}
W_{\r\a}.
\end{equation}
Fixing $\a,\b$, summing over $\g,\r$, and applying
Lemma \ref{lemma:matrixlocallaw},
\begin{equation}\label{eq:gtsum}
\sum_{\g,\r \notin \{\a,\b\}} V_{\b\g}\Delta_{\g\r}W_{\r\a} \prec N^{-1/2}.
\end{equation}
Combining with the bound $\Delta_{\a\b} \prec N^{-1/2}$ and then summing over
$\a,\b$, we see that $\Tr \Delta V \Delta W \prec N$.

We show that the terms where $\a=\b$, $\a=\g$, $\b=\r$, and/or $\g=\r$ are
$\O(1)$: Consider first $\a=\b$. Applying again (\ref{eq:gtsum}) and
$\Delta_{\a\a} \prec N^{-1/2}$, we obtain
\[\sum_{\a,\g,\r} \Delta_{\a\a}V_{\a\g}\Delta_{\g\r}W_{\r\a}
\prec \sum_\a |\Delta_{\a\a}| N^{-1/2} \prec 1.\]
Symmetrically, for $\g=\r$,
\[\sum_{\a,\b,\g} \Delta_{\a\b}V_{\b\g}\Delta_{\g\g}W_{\g\a} \prec 1.\]
For $\a=\g$, let $\v_\a$ and $\w_\a$ be columns
$\a$ of $V$ and $W$. Summing first over $\b,\r$, we
have by Lemma \ref{lemma:matrixlocallaw}
\[\sum_{\a,\b,\r} \Delta_{\a\b}V_{\b\a}\Delta_{\a\r}W_{\r\a}
=\sum_\a \e_\a' \Delta \v_\a \e_\a'\Delta \w_\a
\prec \sum_\a N^{-1/2} \cdot N^{-1/2} \prec 1.\]
Symmetrically, for $\b=\r$,
\[\sum_{\a,\b,\g} \Delta_{\a\b}V_{\b\g}\Delta_{\g\b}W_{\b\a} \prec 1.\]
When two or more of these four cases hold simultaneously, for example
$\a=\b=\g$ or $\a=\g,\;\b=\r$ or $\a=\b=\g=\r$, we have
\begin{align*}
\sum_{\a,\r} \Delta_{\a\a}V_{\a\a}\Delta_{\a\r}W_{\r\a}
&\prec \sum_\a |\Delta_{\a\a}V_{\a\a}|N^{-1/2} \prec 1,\\
\sum_{\a,\b} \Delta_{\a\b}V_{\b\a}\Delta_{\a\b}W_{\b\a}
&\prec N^{-1} \sum_{\a,\b} |V_{\b\a}W_{\b\a}| \prec 1,\\
\sum_\a \Delta_{\a\a}V_{\a\a}\Delta_{\a\a}W_{\a\a}
&\prec N^{-1}\sum_\a |V_{\a\a}W_{\a\a}| \prec 1.
\end{align*}
Then we may eliminate all of these cases from the sum (\ref{eq:TrDVDW}) by
inclusion-exclusion.

The remaining cases are when possibly $\a=\r$ and/or $\b=\g$. We write the
contributions from these cases as
\[\mathrm{I}=\sum_{\a,\b,\g,\r}^* \Delta_{\a\b}
V_{\b\g}\Delta_{\g\r}W_{\r\a},\qquad
\mathrm{II}=\sum_{\a,\b,\g}^* \Delta_{\a\b}V_{\b\g}\Delta_{\g\a}W_{\a\a},\]
\[\mathrm{III}=\sum_{\a,\b,\r}^* \Delta_{\a\b}V_{\b\b}\Delta_{\b\r}W_{\r\a},
\qquad \mathrm{IV}=\sum_{\a,\b}^* \Delta_{\a\b}V_{\b\b}\Delta_{\b\a}W_{\a\a},\]
where summations with $*$ denote that all indices are restricted to be distinct.

For $\mathrm{I}$, let us first apply
\[G_{\a\a}/t_\a=\Pi_{\a\a}/t_\a+\O(N^{-1/2}),\qquad
G_{\b\b}^{(\a)}/t_\b=\Pi_{\b\b}/t_\b+\O(N^{-1/2})\]
from Lemma \ref{lemma:Gaa}.
Then, by (\ref{eq:Deltarewrite}) and Lemma \ref{lemma:schurcomplement}(b), we
have
\begin{equation}\label{eq:Deltaab}
\Delta_{\a\b}=\frac{G_{\a\b}}{t_\a t_\b}
=\frac{G_{\a\a}G_{\b\b}^{(\a)}}{t_\a t_\b}\cZ_{\a\b}
=\frac{\Pi_{\a\a}}{t_\a}\frac{\Pi_{\b\b}}{t_\b}\cZ_{\a\b}+\O(N^{-1}).
\end{equation}
Note that (\ref{eq:gtsum}) holds also with the summation further restricted to
$\g \neq \r$, by Lemma \ref{lemma:matrixlocallaw}. Then, as the $\O(N^{-1})$
remainder term in (\ref{eq:Deltaab})
does not depend on $\g$ and $\r$,
\begin{equation}\label{eq:Istep}
\mathrm{I}=\sum_{\a,\b,\g,\r}^*
\left(\frac{\Pi_{\a\a}}{t_\a}\frac{\Pi_{\b\b}}{t_\b}\cZ_{\a\b}\right)
V_{\b\g}\Delta_{\g\r}W_{\r\a}+\O(N^{1/2}).
\end{equation}
For fixed $\g$ and $\r$,
applying Lemma \ref{lemma:fluctuationavg}(b) and
Lemma \ref{lemma:Zgood}(c), we also have
\[\sum_{\a,\b \notin \{\g,\r\}}^* 
\left(\frac{\Pi_{\a\a}}{t_\a}\frac{\Pi_{\b\b}}{t_\b}
\cZ_{\a\b}\right)V_{\b\g}W_{\r\a} \prec N^{-1/2}.\]
Then we may apply the approximation (\ref{eq:Deltaab}) to $\Delta_{\g\r}$
in (\ref{eq:Istep}), yielding
\begin{equation}\label{eq:IZ}
\mathrm{I}=\sum_{\a,\b,\g,\r}^*
\left(\frac{\Pi_{\a\a}}{t_\a}\frac{\Pi_{\b\b}}{t_\b}
\cZ_{\a\b}\right)V_{\b\g}\left(\frac{\Pi_{\g\g}}{t_\g}
\frac{\Pi_{\r\r}}{t_\r}\cZ_{\g\r}\right)W_{\r\a}
+\O(N^{1/2}).
\end{equation}

Next, let us apply Lemma \ref{lemma:schurcomplement}(b--c) and write
\begin{align}
\cZ_{\a\b}&=\x_\a'G_N^{(\a\b)}\x_\b\nonumber\\
&=\sum_{i,j} X_{\a i}X_{\b j}\left(G_{ij}^{(\a\b\g)}+
\frac{G_{i\g}^{(\a\b)}G_{j\g}^{(\a\b)}}{G_{\g\g}^{(\a\b)}}\right)\nonumber\\
&=\sum_{i,j} X_{\a i}X_{\b j}\left(G_{ij}^{(\a\b\g)}+
G_{\g\g}^{(\a\b)}(\e_i'G_N^{(\a\b\g)}\x_\g)(\e_j'G_N^{(\a\b\g)}\x_\g)\right)
\nonumber\\
&=\x_\a'G_N^{(\a\b\g)}\x_\b+G_{\g\g}^{(\a\b)}
(\x_\a'G_N^{(\a\b\g)}\x_\g)(\x_\b'G_N^{(\a\b\g)}\x_\g).\label{eq:Zabremoveg}
\end{align}
Applying these steps again to the first term of (\ref{eq:Zabremoveg}), we obtain
$\cZ_{\a\b}=\cZ_{\a\b}^{(\g\r)}+R_{\a\b\g\r}$ where
\begin{align*}
\cZ_{\a\b}^{(\g\r)}&=\x_\a'G_N^{(\a\b\g\r)}\x_\b,\\
R_{\a\b\g\r}&=G_{\g\g}^{(\a\b)}
(\x_\a'G_N^{(\a\b\g)}\x_\g)(\x_\b'G_N^{(\a\b\g)}\x_\g)
+G_{\r\r}^{(\a\b\g)}
(\x_\a'G_N^{(\a\b\g\r)}\x_\r)(\x_\b'G_N^{(\a\b\g\r)}\x_\r).
\end{align*}
By Lemmas \ref{lemma:fluctuationavg}(b) and \ref{lemma:Zprodgood}(c),
for fixed $\g$ and $\r$, we have
\[\sum_{\a,\b \notin \{\g,\r\}}^*
\left(\frac{\Pi_{\a\a}}{t_\a}\frac{\Pi_{\b\b}}{t_\b}R_{\a\b\g\r}\right)
V_{\b\g}W_{\r\a} \prec N^{-1}.\]
Then, applying this and $\cZ_{\g\r} \prec N^{-1/2}$, (\ref{eq:IZ}) holds
with $\cZ_{\a\b}$ replaced by $\cZ_{\a\b}^{(\g\r)}$. Applying the symmetric
argument to replace $\cZ_{\g\r}$ by $\cZ_{\g\r}^{(\a\b)}$, we obtain
\[\mathrm{I}=\sum_{\a,\b,\g,\r}^*
\left(\frac{\Pi_{\a\a}}{t_\a}\frac{\Pi_{\b\b}}{t_\b}
\cZ_{\a\b}^{(\g\r)}\right)V_{\b\g}\left(\frac{\Pi_{\g\g}}{t_\g}
\frac{\Pi_{\r\r}}{t_\r}\cZ_{\g\r}^{(\a\b)}\right)W_{\r\a}
+\O(N^{1/2}).\]
Recognizing $\cZ_{\a\b}^{(\g\r)}\cZ_{\g\r}^{(\a\b)}=\cY_{\a\b\g\r}$ and
applying Lemmas \ref{lemma:fluctuationavg}(c) and
\ref{lemma:Zprodgood}(a),
the summation above is $\O(1)$. Then $\mathrm{I} \prec N^{1/2}$.

A similar argument holds for $\mathrm{II}$:
Lemma \ref{lemma:matrixlocallaw} yields for fixed $\a,\b$
\[\sum_{\g \notin \{\a,\b\}} V_{\b\g}\Delta_{\g\a}W_{\a\a} \prec N^{-1/2}.\]
Then applying (\ref{eq:Deltaab}),
\[\mathrm{II}=\sum_{\a,\b,\g}^* 
\left(\frac{\Pi_{\a\a}}{t_\a}\frac{\Pi_{\b\b}}{t_\b}
\cZ_{\a\b}\right)V_{\b\g}\Delta_{\g\a}W_{\a\a}+\O(N^{1/2}).\]
For fixed $\a,\g$, Lemmas \ref{lemma:fluctuationavg}(a) and Lemma
\ref{lemma:Zgood}(c) then yield
\[\sum_{\b \notin \{\a,\g\}}
\left(\frac{\Pi_{\a\a}}{t_\a}\frac{\Pi_{\b\b}}{t_\b}
\cZ_{\a\b}\right) V_{\b\g}W_{\a\a} \prec N^{-1/2},\]
so applying (\ref{eq:Deltaab}) again to approximate $\Delta_{\g\a}$ yields
\begin{equation}\label{eq:IIZ}
\mathrm{II}=\sum_{\a,\b,\g}^*
\left(\frac{\Pi_{\a\a}}{t_\a}\frac{\Pi_{\b\b}}{t_\b}
\cZ_{\a\b}\right)V_{\b\g}
\left(\frac{\Pi_{\g\g}}{t_\g}\frac{\Pi_{\a\a}}{t_\a}
\cZ_{\a\g}\right)W_{\a\a}+\O(N^{1/2}).
\end{equation}
Note that
\[\sum_{\b \notin \{\a,\g\}} \frac{\Pi_{\b\b}}{t_\b}G_{\g\g}^{(\a\b)}
(\x_\a'G_N^{(\a\b\g)}\x_\g)(\x_\b'G_N^{(\a\b\g)}\x_\g) V_{\b\g}
\prec N^{-1}\]
by Lemmas \ref{lemma:fluctuationavg}(a) and
\ref{lemma:Zprodgood}(c). Then applying (\ref{eq:Zabremoveg}) and $\cZ_{\a\g}
\prec N^{-1/2}$,
we may replace $\cZ_{\a\b}$ by $\cZ_{\a\b}^{(\g)}=\x_\a'G_N^{(\a\b\g)}\x_\b$
in (\ref{eq:IIZ}). Applying the symmetric argument to replace $\cZ_{\a\g}$ by
$\cZ_{\a\g}^{(\b)}$, we obtain
\[\mathrm{II}=\sum_{\a,\b,\g}^*
\left(\frac{\Pi_{\a\a}}{t_\a}\frac{\Pi_{\b\b}}{t_\b}
\cZ_{\a\b}^{(\g)}\right)V_{\b\g}
\left(\frac{\Pi_{\g\g}}{t_\g}\frac{\Pi_{\a\a}}{t_\a}
\cZ_{\a\g}^{(\b)}\right)W_{\a\a}+\O(N^{1/2}).\]
Recognizing $\cZ_{\a\b}^{(\g)}\cZ_{\a\g}^{(\b)}=\cY_{\a\b\g}$
and applying Lemmas \ref{lemma:fluctuationavg}(b) and
\ref{lemma:Zprodgood}(b),
\[\sum_{\b,\g \notin \{\a\}}^*
\left(\frac{\Pi_{\a\a}}{t_\a}\frac{\Pi_{\b\b}}{t_\b}
\cZ_{\a\b}^{(\g)}\right)V_{\b\g}
\left(\frac{\Pi_{\g\g}}{t_\g}\frac{\Pi_{\a\a}}{t_\a}
\cZ_{\a\g}^{(\b)}\right) \prec N^{-1}\|V\|_\HS \prec N^{-1/2}.\]
Then $\mathrm{II} \prec N^{1/2}$. By symmetry, $\mathrm{III} \prec N^{1/2}$
also.

For $\mathrm{IV}$, a direct bound using (\ref{eq:Deltaab}),
$|V_{\b\b}| \leq 1$, and $|W_{\a\a}| \leq 1$ yields
\[\mathrm{IV}=\sum_{\a,\b}^*
\left(\frac{\Pi_{\a\a}}{t_\a}\frac{\Pi_{\b\b}}{t_\b}
\cZ_{\a\b}\right)^2V_{\b\b}W_{\a\a}+\O(N^{1/2}).\]
Summing first over $\b$, Lemmas \ref{lemma:fluctuationavg}(a) and 
\ref{lemma:Zprodgood}(d) yield
\[\sum_{\b \notin \{\a\}}
\left(\frac{\Pi_{\b\b}}{t_\b}
\right)^2V_{\b\b}\Big(\cZ_{\a\b}^2-\E_\b[\cZ_{\a\b}^2]\Big)
\prec N^{-1/2}.\]
Then summing over $\a$ and applying $|W_{\a\a}| \leq 1$,
\[\mathrm{IV}=\sum_{\a,\b}^*
\left(\frac{\Pi_{\a\a}}{t_\a}\frac{\Pi_{\b\b}}{t_\b}\right)^2\E_\b[\cZ_{\a\b}^2]
V_{\b\b}W_{\a\a}+\O(N^{1/2}).\]
Next, summing first over $\a$, Lemmas \ref{lemma:fluctuationavg}(a) and 
\ref{lemma:Zprodgood}(e) yield
\[\sum_{\a \notin \{\b\}} \left(\frac{\Pi_{\a\a}}{t_\a}\right)^2W_{\a\a}
\Big(\E_\b[\cZ_{\a\b}^2]-\E_{\a\b}[\cZ_{\a\b}^2]\Big) \prec N^{-1/2}.\]
Summing over $\b$ and applying $|V_{\b\b}| \leq 1$,
\[\mathrm{IV}=\sum_{\a,\b}^* 
\left(\frac{\Pi_{\a\a}}{t_\a}\frac{\Pi_{\b\b}}{t_\b}
\right)^2 \E_{\a\b}[\cZ_{\a\b}^2]V_{\b\b}W_{\a\a}+\O(N^{1/2}).\]

Finally, let us verify
\begin{equation}\label{eq:EabZab2}
\E_{\a\b}[\cZ_{\a\b}^2]=N^{-1}\partial_z m_0+\O(N^{-3/2}).
\end{equation}
First note that $\E_{\a\b}[\cZ_{\a\b}^2]=N^{-2}\Tr [(G_N^{(\a\b)})^2]$. Writing
$G_N^{(\a\b)}=G_N+R$, Lemma \ref{lemma:schurcomplement}(c) implies that each
entry of $R$ is $\O(N^{-1})$. Then $\Tr [(G_N^{(\a\b)})^2]=\Tr G_N^2+2\Tr
G_NR+\Tr R^2$. We have
\[\Tr G_NR=\sum_i G_{ii}R_{ii}+\sum_{i \neq j} G_{ij}R_{ij}
\prec \sum_i 1 \cdot N^{-1}+\sum_{i \neq j} N^{-1/2} \cdot N^{-1}
\prec N^{1/2}, \qquad \Tr R^2=\sum_{i,j} R_{ij}^2 \prec 1.\]
Hence $\E_{\a\b}[\cZ_{\a\b}^2]=N^{-2}\Tr G_N^2+\O(N^{-3/2})$. Next, note
that $N^{-1}\Tr G_N^2=\partial_z m_N$ by the spectral representation of $G_N$.
From Lemma \ref{lemma:mNapprox}, $m_N-m_0 \prec N^{-1}$. Applying the same
Lipschitz continuity and Cauchy integral argument as in Section
\ref{subsec:LLN}, we obtain $\partial_z m_N-\partial_z m_0 \prec N^{-1}$, and
hence (\ref{eq:EabZab2}).

Combining these arguments,
\[\Tr \Delta V\Delta W=\mathrm{I}+\mathrm{II}+\mathrm{III}+\mathrm{IV}+\O(1)
=N^{-1}(\partial_z m_0)
\sum_{\a,\b}^* \left(\frac{\Pi_{\a\a}}{t_\a}\frac{\Pi_{\b\b}}{t_\b}\right)^2
V_{\b\b}W_{\a\a}+\O(N^{1/2}).\]
Including the $\a=\b$ case into the sum introduces an $\O(1)$ error.
Then writing $\sum_\b V_{\b\b}(\Pi_{\b\b}/t_\b)^2=\Tr (V[\Id+m_0T]^{-2})$ and
similarly for $W$ concludes the proof.

\bibliographystyle{alpha}
\bibliography{references}

\newcommand{\etalchar}[1]{$^{#1}$}
\begin{thebibliography}{EKYY13b}

\bibitem[Ame85]{amemiya}
Yasuo Amemiya.
\newblock What should be done when an estimated between-group covariance matrix
  is not nonnegative definite?
\newblock {\em The American Statistician}, 39(2):112--117, 1985.

\bibitem[BAC{\etalchar{+}}15]{blowsetal}
Mark~W Blows, Scott~L Allen, Julie~M Collet, Stephen~F Chenoweth, and Katrina
  McGuigan.
\newblock The phenome-wide distribution of genetic variance.
\newblock {\em The American Naturalist}, 186(1):15--30, 2015.

\bibitem[BBC{\etalchar{+}}17]{belinschietal}
Serban~T Belinschi, Hari Bercovici, Mireille Capitaine, Maxime F{\'e}vrier,
  et~al.
\newblock Outliers in the spectrum of large deformed unitarily invariant
  models.
\newblock {\em The Annals of Probability}, 45(6A):3571--3625, 2017.

\bibitem[BBP05]{baiketal}
Jinho Baik, Gerard {Ben~Arous}, and Sandrine P{\'e}ch{\'e}.
\newblock Phase transition of the largest eigenvalue for nonnull complex sample
  covariance matrices.
\newblock {\em The Annals of Probability}, 33(5):1643--1697, 2005.

\bibitem[BEK{\etalchar{+}}14]{bloemendaletal}
Alex Bloemendal, L{\'a}szl{\'o} Erdos, Antti Knowles, Horng-Tzer Yau, and Jun
  Yin.
\newblock Isotropic local laws for sample covariance and generalized {W}igner
  matrices.
\newblock {\em Electronic Journal of Probability}, 19(33):1--53, 2014.

\bibitem[BGGM11]{benaychgeorgesetal}
Florent Benaych-Georges, Alice Guionnet, and Myl{\`e}ne Maida.
\newblock Fluctuations of the extreme eigenvalues of finite rank deformations
  of random matrices.
\newblock {\em Electronic Journal of Probability}, 16:1621--1662, 2011.

\bibitem[BGN11]{benaychgeorgesnadakuditi}
Florent Benaych-Georges and Raj~Rao Nadakuditi.
\newblock The eigenvalues and eigenvectors of finite, low rank perturbations of
  large random matrices.
\newblock {\em Advances in Mathematics}, 227(1):494--521, 2011.

\bibitem[Blo07]{blows}
Mark~W Blows.
\newblock A tale of two matrices: {M}ultivariate approaches in evolutionary
  biology.
\newblock {\em Journal of Evolutionary Biology}, 20(1):1--8, 2007.

\bibitem[BM15]{blowsmcguigan}
Mark~W Blows and Katrina McGuigan.
\newblock The distribution of genetic variance across phenotypic space and the
  response to selection.
\newblock {\em Molecular Ecology}, 24(9):2056--2072, 2015.

\bibitem[BS98]{baisilverstein}
Zhidong Bai and Jack~W Silverstein.
\newblock No eigenvalues outside the support of the limiting spectral
  distribution of large-dimensional sample covariance matrices.
\newblock {\em The Annals of Probability}, 26(1):316--345, 1998.

\bibitem[BS06]{baiksilverstein}
Jinho Baik and Jack~W Silverstein.
\newblock Eigenvalues of large sample covariance matrices of spiked population
  models.
\newblock {\em Journal of Multivariate Analysis}, 97(6):1382--1408, 2006.

\bibitem[BY08]{baiyaoCLT}
Zhidong Bai and Jianfeng Yao.
\newblock Central limit theorems for eigenvalues in a spiked population model.
\newblock {\em Annales de l'Institut Henri Poincar{\'e}, Probabilit{\'e}s et
  Statistiques}, 44(3):447--474, 2008.

\bibitem[BY12]{baiyaogeneralized}
Zhidong Bai and Jianfeng Yao.
\newblock On sample eigenvalues in a generalized spiked population model.
\newblock {\em Journal of Multivariate Analysis}, 106:167--177, 2012.

\bibitem[CMA{\etalchar{+}}18]{colletetal}
Julie~M Collet, Katrina McGuigan, Scott~L Allen, Stephen~F Chenoweth, and
  Mark~W Blows.
\newblock Mutational pleiotropy and the strength of stabilizing selection
  within and between functional modules of gene expression.
\newblock {\em Genetics}, 208(4):1601--1616, 2018.

\bibitem[CR48]{comstockrobinson}
Ralph~E Comstock and Harold~F Robinson.
\newblock The components of genetic variance in populations of biparental
  progenies and their use in estimating the average degree of dominance.
\newblock {\em Biometrics}, 4(4):254--266, 1948.

\bibitem[EKYY13a]{erdoslocalSC}
L{\'a}szl{\'o} Erd{\H{o}}s, Antti Knowles, Horng-Tzer Yau, and Jun Yin.
\newblock The local semicircle law for a general class of random matrices.
\newblock {\em Electronic Journal of Probability}, 18(59):1--58, 2013.

\bibitem[EKYY13b]{erdosER}
L{\'a}szl{\'o} Erd{\H{o}}s, Antti Knowles, Horng-Tzer Yau, and Jun Yin.
\newblock Spectral statistics of {E}rd{\H{o}}s--{R}{\'e}nyi graphs {I}: {L}ocal
  semicircle law.
\newblock {\em The Annals of Probability}, 41(3B):2279--2375, 2013.

\bibitem[EYY11]{erdosbernoulli}
L{\'a}szl{\'o} Erd{\H{o}}s, Horng-Tzer Yau, and Jun Yin.
\newblock Universality for generalized {W}igner matrices with {B}ernoulli
  distribution.
\newblock {\em Journal of Combinatorics}, 2(1):15--81, 2011.

\bibitem[EYY12]{erdosyauyin}
L{\'a}szl{\'o} Erd{\H{o}}s, Horng-Tzer Yau, and Jun Yin.
\newblock Rigidity of eigenvalues of generalized {W}igner matrices.
\newblock {\em Advances in Mathematics}, 229(3):1435--1515, 2012.

\bibitem[FBSG{\etalchar{+}}15]{finucaneetal}
Hilary~K Finucane, Brendan Bulik-Sullivan, Alexander Gusev, et~al.
\newblock Partitioning heritability by functional annotation using genome-wide
  association summary statistics.
\newblock {\em Nature Genetics}, 47(11):1228, 2015.

\bibitem[Fis18]{fisher}
Ronald~A Fisher.
\newblock The correlation between relatives on the supposition of {M}endelian
  inheritance.
\newblock {\em Transactions of the Royal Society of Edinburgh},
  52(02):399--433, 1918.

\bibitem[FJ16]{fanjohnstonebulk}
Zhou Fan and Iain~M Johnstone.
\newblock Eigenvalue distributions of variance components estimators in
  high-dimensional random effects models.
\newblock {\em arXiv preprint 1607.02201v2}, 2016.

\bibitem[FJ17]{fanjohnstoneedges}
Zhou Fan and Iain~M Johnstone.
\newblock Tracy-{W}idom at each edge of real covariance estimators.
\newblock {\em arXiv preprint arXiv:1707.02352v2}, 2017.

\bibitem[FM96]{falconermackay}
Douglas~S Falconer and Trudy F~C Mackay.
\newblock {\em Introduction to Quantitative Genetics}.
\newblock Longman, Harlow, 1996.

\bibitem[HB06]{hineblows}
Emma Hine and Mark~W Blows.
\newblock Determining the effective dimensionality of the genetic
  variance--covariance matrix.
\newblock {\em Genetics}, 173(2):1135--1144, 2006.

\bibitem[HGO10]{houleetal}
David Houle, Diddahally~R Govindaraju, and Stig Omholt.
\newblock Phenomics: {T}he next challenge.
\newblock {\em Nature Reviews Genetics}, 11(12):855--866, 2010.

\bibitem[HH81]{hayeshill}
J~F Hayes and W~G Hill.
\newblock Modification of estimates of parameters in the construction of
  genetic selection indices (`{B}ending').
\newblock {\em Biometrics}, 37(3):483--493, 1981.

\bibitem[HMB14]{hineetal}
Emma Hine, Katrina McGuigan, and Mark~W Blows.
\newblock Evolutionary constraints in high-dimensional trait sets.
\newblock {\em The American Naturalist}, 184(1):119--131, 2014.

\bibitem[Hou10]{houle}
David Houle.
\newblock Numbering the hairs on our heads: {T}he shared challenge and promise
  of phenomics.
\newblock {\em Proceedings of the National Academy of Sciences}, 107(suppl
  1):1793--1799, 2010.

\bibitem[Joh01]{johnstone}
Iain~M Johnstone.
\newblock On the distribution of the largest eigenvalue in principal components
  analysis.
\newblock {\em The Annals of Statistics}, 29(2):295--327, 2001.

\bibitem[KM04]{kirkpatrickmeyerdirect}
Mark Kirkpatrick and Karin Meyer.
\newblock Direct estimation of genetic principal components.
\newblock {\em Genetics}, 168(4):2295--2306, 2004.

\bibitem[KP69]{klotzputter}
Jerome Klotz and Joseph Putter.
\newblock Maximum likelihood estimation of multivariate covariance components
  for the balanced one-way layout.
\newblock {\em The Annals of Mathematical Statistics}, 40(3):1100--1105, 1969.

\bibitem[KY17]{knowlesyin}
Antti Knowles and Jun Yin.
\newblock Anisotropic local laws for random matrices.
\newblock {\em Probability Theory and Related Fields}, 169(1-2):257--352, 2017.

\bibitem[LA83]{landearnold}
Russell Lande and Stevan~J Arnold.
\newblock The measurement of selection on correlated characters.
\newblock {\em Evolution}, 37(6):1210--1226, 1983.

\bibitem[LaM73]{lamotte}
Lynn~R LaMotte.
\newblock Quadratic estimation of variance components.
\newblock {\em Biometrics}, 29(2):311--330, 1973.

\bibitem[Lan79]{lande}
Russell Lande.
\newblock Quantitative genetic analysis of multivariate evolution, applied to
  brain: {B}ody size allometry.
\newblock {\em Evolution}, 33(1):402--416, 1979.

\bibitem[LTBS{\etalchar{+}}15]{lohetal}
Po-Ru Loh, George Tucker, Brendan~K Bulik-Sullivan, et~al.
\newblock Efficient {B}ayesian mixed-model analysis increases association power
  in large cohorts.
\newblock {\em Nature Genetics}, 47(3):284, 2015.

\bibitem[LW98]{lynchwalsh}
Michael Lynch and Bruce Walsh.
\newblock {\em Genetics and Analysis of Quantitative Traits}.
\newblock Sinauer Sunderland, 1998.

\bibitem[MAB15]{mcguiganmutation}
Katrina McGuigan, J~David Aguirre, and Mark~W Blows.
\newblock Simultaneous estimation of additive and mutational genetic variance
  in an outbred population of {D}rosophila serrata.
\newblock {\em Genetics}, 201(3):1239--1251, 2015.

\bibitem[MCM{\etalchar{+}}14]{mcguiganetal}
Katrina McGuigan, Julie~M Collet, Elizabeth~A McGraw, H~Ye Yixin, Scott~L
  Allen, Stephen~F Chenoweth, and Mark~W Blows.
\newblock The nature and extent of mutational pleiotropy in gene expression of
  male {D}rosophila serrata.
\newblock {\em Genetics}, 196(3):911--921, 2014.

\bibitem[Mey91]{meyerREML}
Karin Meyer.
\newblock Estimating variances and covariances for multivariate animal models
  by restricted maximum likelihood.
\newblock {\em Genetics Selection Evolution}, 23(1):67, 1991.

\bibitem[MH05]{mezeyhoule}
Jason~G Mezey and David Houle.
\newblock The dimensionality of genetic variation for wing shape in
  {D}rosophila melanogaster.
\newblock {\em Evolution}, 59(5):1027--1038, 2005.

\bibitem[MK05]{meyerkirkpatricksmoothed}
Karin Meyer and Mark Kirkpatrick.
\newblock Restricted maximum likelihood estimation of genetic principal
  components and smoothed covariance matrices.
\newblock {\em Genetics Selection Evolution}, 37(1):1, 2005.

\bibitem[MK08]{meyerkirkpatrickperils}
Karin Meyer and Mark Kirkpatrick.
\newblock Perils of parsimony: {P}roperties of reduced-rank estimates of
  genetic covariance matrices.
\newblock {\em Genetics}, 180(2):1153--1166, 2008.

\bibitem[MK10]{meyerkirkpatrickbending}
Karin Meyer and Mark Kirkpatrick.
\newblock Better estimates of genetic covariance matrices by ``bending'' using
  penalized maximum likelihood.
\newblock {\em Genetics}, 185(3):1097--1110, 2010.

\bibitem[MLH{\etalchar{+}}15]{moseretal}
Gerhard Moser, Sang~Hong Lee, Ben~J Hayes, Michael~E Goddard, Naomi~R Wray, and
  Peter~M Visscher.
\newblock Simultaneous discovery, estimation and prediction analysis of complex
  traits using a {B}ayesian mixture model.
\newblock {\em PLoS Genetics}, 11(4):e1004969, 2015.

\bibitem[MP67]{marcenkopastur}
Vladimir~A Marcenko and Leonid~Andreevich Pastur.
\newblock Distribution of eigenvalues for some sets of random matrices.
\newblock {\em Sbornik: Mathematics}, 1(4):457--483, 1967.

\bibitem[Nad08]{nadler}
Boaz Nadler.
\newblock Finite sample approximation results for principal component analysis:
  {A} matrix perturbation approach.
\newblock {\em The Annals of Statistics}, 36(6):2791--2817, 2008.

\bibitem[Pau07]{paul}
Debashis Paul.
\newblock Asymptotics of sample eigenstructure for a large dimensional spiked
  covariance model.
\newblock {\em Statistica Sinica}, 17(4):1617--1642, 2007.

\bibitem[Rao72]{rao}
C~Radhakrishna Rao.
\newblock Estimation of variance and covariance components in linear models.
\newblock {\em Journal of the American Statistical Association},
  67(337):112--115, 1972.

\bibitem[Rob59a]{robertsona}
Alan Robertson.
\newblock Experimental design in the evaluation of genetic parameters.
\newblock {\em Biometrics}, 15(2):219--226, 1959.

\bibitem[Rob59b]{robertsonb}
Alan Robertson.
\newblock The sampling variance of the genetic correlation coefficient.
\newblock {\em Biometrics}, 15(3):469--485, 1959.

\bibitem[RV13]{rudelsonvershynin}
Mark Rudelson and Roman Vershynin.
\newblock Hanson-{W}right inequality and sub-gaussian concentration.
\newblock {\em Electronic Communications in Probability}, 18, 2013.

\bibitem[SB95]{silversteinbai}
Jack~W Silverstein and Zhidong Bai.
\newblock On the empirical distribution of eigenvalues of a class of large
  dimensional random matrices.
\newblock {\em Journal of Multivariate Analysis}, 54(2):175--192, 1995.

\bibitem[SC95]{silversteinchoi}
Jack~W Silverstein and Sang-Il Choi.
\newblock Analysis of the limiting spectral distribution of large dimensional
  random matrices.
\newblock {\em Journal of Multivariate Analysis}, 54(2):295--309, 1995.

\bibitem[SCM09]{searleetal}
Shayle~R Searle, George Casella, and Charles~E McCulloch.
\newblock {\em Variance Components}.
\newblock John Wiley \& Sons, 2009.

\bibitem[Sil95]{silverstein}
Jack~W Silverstein.
\newblock Strong convergence of the empirical distribution of eigenvalues of
  large dimensional random matrices.
\newblock {\em Journal of Multivariate Analysis}, 55(2):331--339, 1995.

\bibitem[Spe83]{speed}
T~P Speed.
\newblock Cumulants and partition lattices.
\newblock {\em Australian \& New Zealand Journal of Statistics},
  25(2):378--388, 1983.

\bibitem[SR74]{searlerounsaville}
S~R Searle and T~R Rounsaville.
\newblock A note on estimating covariance components.
\newblock {\em The American Statistician}, 28(2):67--68, 1974.

\bibitem[SS78]{swallowsearle}
William~H Swallow and S~R Searle.
\newblock Minimum variance quadratic unbiased estimation ({MIVQUE}) of variance
  components.
\newblock {\em Technometrics}, 20(3):265--272, 1978.

\bibitem[VHW08]{visscheretal}
Peter~M Visscher, William~G Hill, and Naomi~R Wray.
\newblock Heritability in the genomics era---concepts and misconceptions.
\newblock {\em Nature Reviews Genetics}, 9(4):255--266, 2008.

\bibitem[WB09]{walshblows}
Bruce Walsh and Mark~W Blows.
\newblock Abundant genetic variation+strong selection=multivariate genetic
  constraints: {A} geometric view of adaptation.
\newblock {\em Annual Review of Ecology, Evolution, and Systematics},
  40:41--59, 2009.

\bibitem[Wri35]{wright}
Sewall Wright.
\newblock The analysis of variance and the correlations between relatives with
  respect to deviations from an optimum.
\newblock {\em Journal of Genetics}, 30(2):243--256, 1935.

\bibitem[YLGV11]{yangetal}
Jian Yang, S~Hong Lee, Michael~E Goddard, and Peter~M Visscher.
\newblock {GCTA}: a tool for genome-wide complex trait analysis.
\newblock {\em The American Journal of Human Genetics}, 88(1):76--82, 2011.

\bibitem[ZCS13]{zhouetal}
Xiang Zhou, Peter Carbonetto, and Matthew Stephens.
\newblock Polygenic modeling with {B}ayesian sparse linear mixed models.
\newblock {\em PLoS Genetics}, 9(2):e1003264, 2013.

\end{thebibliography}

\end{document}